\documentclass[10pt,reqno,a4paper]{article}
\usepackage{amsmath,amsfonts,amscd,amsthm,amssymb}
\usepackage{graphicx}
\usepackage{subcaption}
\usepackage[dvipsnames]{xcolor}
\usepackage{cite}
\usepackage{tikz}
\definecolor{lightblue}{rgb}{0.22,0.45,0.70}% for references
\usepackage[colorlinks=true,breaklinks=true]{hyperref}\hypersetup{urlcolor=blue, citecolor=lightblue}
\usepackage{scalerel,stackengine}
\usepackage{bm}
\usepackage{enumerate}
\usepackage{multicol}
\allowdisplaybreaks

\newcommand{\R}{\mathbb{R}}                                                              
\newcommand{\pt}{\partial_t}

\newcommand{\bw}{{\bm w}}

\newcommand{\bv}{{\bm v}}

\newcommand{\bS}{{\bm S}}
\newcommand{\bV}{{\bm V}}

\newcommand{\bH}{{\bm H}}
\newcommand{\bL}{{\bm L}}

\newcommand{\f}{{\bm f}}

\newcommand{\bu}{{\bm u}}

\newcommand{\hbu}{\widehat{\bu}}

\newcommand{\bn}{{\bm n}}

\newcommand{\tc}{\tilde{c}}
\newcommand{\Dt}{\Delta t}

\newcommand{\bphi}{\mbox{\boldmath $\phi$}}
\newcommand{\bpsi}{\mbox{\boldmath $\psi$}}
\newcommand{\bxi}{\mbox{\boldmath $\xi$}}

\newcommand{\bta}{\mbox{\boldmath $\eta$}}
\newcommand{\hbta}{\hat{\bta}}

\newcommand{\e}{{\bm e}}
\newcommand{\eu}{{\bm e}_u}
\newcommand{\heu}{\hat{\bm e}_u}
\newcommand{\eci}{e_{c_i}}
\newcommand{\eco}{e_{c_1}}
\newcommand{\ect}{e_{c_2}}
\newcommand{\ephi}{e_\phi}
\newcommand{\ep}{e_p}

\newcommand{\Eh}{\mathcal{E}_h}
\newcommand{\dx}{\mathrm{\,d}x}
\newcommand{\ds}{\mathrm{\,d}s}
\newcommand{\dt}{\mathrm{\,d}t}

\newcommand{\cA}{\mathcal{A}}
\newcommand{\cN}{\mathcal{N}}
\newcommand{\cD}{\mathcal{D}}
\newcommand{\cG}{\mathcal{G}}
\newcommand{\cT}{\mathcal{T}}

\stackMath 
\newcommand\reallywidehat[1]{%
\savestack{\tmpbox}{\stretchto{%
		\scaleto{%
			\scalerel*[\widthof{\ensuremath{#1}}]{\kern-.6pt\bigwedge\kern-.6pt}%
			{\rule[-\textheight/2]{1ex}{\textheight}}%WIDTH-LIMITED BIG WEDGE
		}{\textheight}% 
	}{0.5ex}}%
\stackon[1pt]{#1}{\tmpbox}%
}
\newcommand{\vertiiidg}[1]{{\left\vert\kern-0.25ex\left\vert\kern-0.25ex\left\vert #1 
	\right\vert\kern-0.25ex\right\vert\kern-0.25ex\right\vert}}
\newcommand{\vertiii}[1]{{\vert\kern-0.25ex\vert #1 
\vert\kern-0.25ex\vert}_{dG}}
\newcommand{\jump}[1]{[\kern-0.85ex[ #1 ]\kern-0.85ex]} %
\newcommand{\avr}[1]{\{\kern-0.85ex\{ #1 \}\kern-0.85ex\}}	
\DeclareSymbolFont{myletters}{OML}{ztmcm}{m}{it}
\DeclareMathSymbol{\uplambda}{\mathord}{myletters}{"15}

\numberwithin{equation}{section}
\numberwithin{figure}{section}
\numberwithin{table}{section}

\begin{document}

\newcommand{\se}{\setcounter{equation}{0}}
\def\theequation{\thesection.\arabic{equation}}

\newtheorem{theorem}{Theorem}[section]
\newtheorem{lemma}{Lemma}[section]
\newtheorem{remark}{Remark}[section]
\newtheorem{example}{Example}[section]
\newtheorem{proposition}{Proposition}[section]

\def\cydot{\leavevmode\raise.4ex\hbox{.}}

\title
{ \large\bf Discontinuous Galerkin IMEX Pressure Correction Scheme for the Poisson-Nernst-Planck-Navier-Stokes Equations}
%{ \large\bf Optimal Error Estimates of  IMEX Discontinuous Galerkin Scheme for the Poisson-Nernst-Planck-Navier-Stokes Equations}

\author{Bikram Bir\thanks{Department of Mathematics, 
		The Assam Royal Global University, Guwahati-781035, India. Email: bikram@math.iitb.ac.in, bbir@rgi.edu.in}
	~~and
	Amiya K. Pani\thanks{Department of Mathematics, 
		BITS-Pilani, KK Birla Goa Campus, NH 17 B, Zuarinagar, Goa-403726, India. Email: amiyap@goa.bits-pilani.ac.in,  akp@math.iitb.ac.in}}

\date{}
\maketitle

\begin{abstract}
	Based on a discontinuous Galerkin method in the spatial directions and an improved implicit-explicit pressure-correction scheme in the temporal direction, this paper discusses a fully discrete scheme for the Poisson-Nernst-Planck-Navier-Stokes equations. Optimal error estimates are derived in $L^2$ and in the energy norms for the concentrations of positive and negative ions, the electrostatic potential, the fluid velocity, and the $L^2$ norm of the fluid pressure. The discrete mass conservation properties of both ions are established. Finally, numerical simulations are performed, whose results confirm our theoretical findings.
\end{abstract}

\vspace{1em} 

\section{Introduction}
The coupled system of the Poisson-Nernst-Planck and Navier-Stokes (PNP-NS) equations has been studied in the context of the dynamics of electrically charged fluids. It describes the motion and interactive behaviour of charged colloidal particles, ions, or molecules under the influence of electric fields and surrounding fluids. These phenomena have widely appeared in many engineering and biological processes, such as battery and fuel cell technology, microfluidic devices or systems, electrokinetic flows in electrophysiology, biomedical lab-on-a-chip devices, lens micro-circulation, capillary electrophoresis, manipulation of biological cells and vesicles, self-assembly or separation of colloidal particles, desalination of water, etc., see \cite{BTA04, LW20, PLZJH24} and references, therein.

The PNP-NS system consists of three parts: the Poisson equation for the internal electric potential, the Planck-Nernst equations for the concentrations of charged particles, and the incompressible Navier-Stokes equations for the motion of the fluid field under the influence of the internal and external electric fields. The mathematical model gives rise to  the following coupled PNP-NS system in a space-time domain $\Omega\times (0, T],~T>0$ for $i=\{1,2\}$ 
\begin{align}
	& - \mu\Delta \phi  = c_1-c_2, & \mbox {in}~& \Omega\times(0,T], \label{eqphi}\\
	&\pt c_i - \kappa_i \Delta c_i + (\bu\cdot\nabla) c_i - \beta_i \nabla\cdot( c_i\nabla \phi) = 0, &\mbox {in}~& \Omega\times(0,T], \label{eqc}\\
	&\pt\bu - \nu\Delta\bu + (\bu\cdot\nabla)\bu + \nabla p =  -(c_1-c_2)\nabla\phi, &\mbox {in}~& \Omega\times(0,T], \label{equ}\\
	&\nabla \cdot \bu=0, & \mbox {in}~& \Omega\times(0,T] \label{eqdiv}.
\end{align}
Here,  $\Omega$ is a convex bounded polygonal domain in $\mathbb{R}^2$ with boundary $\partial \Omega$. $\pt$ represents the partial derivative with respect to time and $c_i(x,t),~i=\{1,2\}$ are the concentrations of positive and negative charged particles, respectively. $\phi(x,t)$ denotes the quasi-electrostatic potential generated by the heterogeneous distribution of positively and negatively charged particles, and $\bu(x,t)$ and $p(x,t)$ denote the fluid velocity and pressure, respectively. The value of the constant $\beta_1=1$ and $\beta_2=-1$. The constants $\kappa_i>0, i=\{1,2\}$ are the diffusion/mobility coefficients, $\mu>0$ is the dielectric coefficient and $\nu>0$ is the viscosity of the fluid. 

This system is appended by the no-slip boundary condition for fluid velocity $\bu$ and homogeneous Neumann boundary conditions for the concentrations of charged particles $c_1, c_2$ and electrostatic potential $\phi$:
\begin{equation}
	\nabla c_1\cdot \bn = \nabla c_2\cdot \bn = \nabla \phi\cdot \bn = 0, ~~\bu = 0,\quad\mbox {on}~\partial\Omega\times (0,T], \label{eqbd}
\end{equation}
where $\bn$ is the outward normal to $\partial\Omega$ and the following initial conditions:
\begin{align}
	c_1(\cdot,0) =  c_{10},~~ c_2(\cdot,0) = c_{20},~~\bu(\cdot,0) = \bu_0,\quad\mbox {in}~\Omega. \label{eqint}
\end{align}
Here, the initial data $ c_{10}, c_{20}, \bu_0$ are given functions in their respective domains of definition. 

It is observed that for non-negative initial data $c_{10}\ge 0, c_{20}\ge 0$ a.e. in $\Omega$, the solutions of \eqref{eqc} are non-negative a.e on $\Omega\times (0,T]$ \cite[Lemma 1, pp. 1003]{Sch09}. In particular, if the initial data $ c_{10}, c_{20}$ are strictly positive, then the solution of \eqref{eqc} is also strictly positive. 

To show the well-posedness of the PNP-NS system with homogeneous boundary conditions \eqref{eqbd}, in particular, the Poisson equation with pure Neumann boundary condition, one needs the following compatibility (or initial electroneutrality) condition:
\begin{equation}\label{inmass}
	\int_\Omega \left(c_1(x,t) - c_2(x,t)\right) \dx = 0 \quad \text{i.e} \int_\Omega c_1(x,t) \dx = \int_\Omega c_2(x,t) \dx.
\end{equation}
From \eqref{eqc} and \eqref{inmass}, it is observed that the mass is conserved for both the concentrations of charged ions, and they are equal \cite{Sch09}, that is,
\begin{equation*}
	\int_\Omega c_1(x,t) \dx = \int_\Omega c_{10}(x) \dx = \int_\Omega c_2(x,t) \dx = \int_\Omega c_{20}(x) \dx.
\end{equation*}
The system \eqref{eqphi}-\eqref{eqbd} satisfy the following electric potential energy dissipation \cite{HC24}
\begin{equation} \label{electricpotential}
	\frac{d}{dt}E_{elec} (\bu, \phi) \le -\int_\Omega \left(\nu|\nabla\bu|^2 + \frac{1}{\mu}\left|\sqrt{\kappa_1}c_1-\sqrt{\kappa_2}c_2\right|^2 + (c_1+c_2)|\nabla\phi|^2\right) \dx \quad \forall t>0,
\end{equation}
where 
\begin{equation}
	E_{elec}(\bu, \phi) = \int_\Omega \left(\frac{1}{2}|\bu|^2 + \frac{\mu}{2}|\nabla\phi|^2\right) \dx.
\end{equation}
Note that the equality of \eqref{electricpotential} holds when $\kappa_1=\kappa_2$. 
If $c_{10}>0, c_{20}>0$ a.e. in $\Omega$, the solution of the system \eqref{eqphi}-\eqref{eqbd} satisfy the following total free energy dissipation \cite{HC24}
\begin{equation}\label{total_energy}
	\frac{d}{dt}E_{total} (c_1,c_2,\bu, \phi) = -\int_\Omega \left(\nu|\nabla\bu|^2 + c_1|\nabla(\kappa_1\log c_1+\phi)|^2 + c_2|\nabla(\kappa_2\log c_2 - \phi)|^2\right) \dx \quad \forall t>0,
\end{equation}
where 
\begin{equation}
	E_{total} (c_1,c_2,\bu, \phi) = \int_\Omega \left(\frac{1}{2}|\bu|^2 + \frac{\mu}{2}|\nabla\phi|^2 + \kappa_1c_1(\log c_1-1)+\kappa_2c_2(\log c_2-1)\right) \dx.
\end{equation}
Note that the electrostatic potential $\phi$ and the fluid pressure $p$ are unique up to a constant. For uniqueness, it is then assumed the zero mean solutions $\phi$ and $p$, that is,   $\int_\Omega \phi(x,t) \dx = 0$ and $\int_\Omega p(x,t) \dx = 0$, respectively.

This model has been studied by several researchers from different perspectives. From a theoretical point of view, several works have been developed on existence and uniqueness, regularity results, and asymptotic behaviour, see \cite{Jer02, Ryh09, Sch09, BFS14, ZY15, LW20, ZL24} and references therein. 
The local existence and uniqueness results for the weak solution of smooth initial data with nonnegativity-preserving properties for the ion concentrations have been established by Jerome \cite{Jer02} using a semigroup-theoretic approach.
In \cite{Sch09}, the global existence of weak solutions in two and three dimensions and uniqueness in two dimensions have been proved for a bounded domain with blocking boundary conditions on the ion concentrations and the homogeneous Neumann boundary condition for the electric potential with no-slip boundary condition for the fluid velocity. In particular, it is shown that for non-negative initial data $c_{10}\ge 0, c_{20}\ge 0$, the solutions of \eqref{eqc} are non-negative a.e. on $\Omega\times (0, T]$ \cite[Lemma 1, pp. 1003]{Sch09}.
In \cite{Ryh09}, the author established the well-posedness and regularity of the solution for nonsmooth initial data in two dimensions and, in three dimensions, under smallness assumptions on the data.
Existence and uniqueness of the local strong solution for $\R^d, d\ge 2$ and the existence of a unique global strong solution and exponential convergence have been shown for $d=2$ in \cite{BFS14}. 
Furthermore, Constantin and their group have discussed the existence, uniqueness, long-time behaviour and regularity results of a global solution for $d=2, 3$ in a series of papers \cite{CIL22, Lee22} and references, therein.  Recently, the global existence of a unique large solution to the three-dimensional  PNP-NS system has been proved by Zhao and Li \cite{ZL24}.

From a numerical point of view, there are only a few results available in the literature. For example, Prohl and Schmuck \cite{PS10} have first designed two fully discrete, nonlinear, and first-order accurate-in-time schemes based on the finite element method.
Their scheme is an energy-stable scheme with and without unconditional preservation of the ions' non-negativity, and convergence analysis is discussed for both schemes.
Subsequently, Liu and Xu \cite{LX17} have developed several first- and second-order non-negativity-preserving schemes based on a time-stepping spectral method and discussed their stability. 
A mixed finite element method for the PNP-Stokes system has been analyzed by He and Sun \cite{HS18}, who obtained optimal error bounds. On the computational side, at each time step, they must solve a nonlinear algebraic system. 
Later, they have extended this work for modified PNP-NSE \cite{HS21}, where the Poisson equation in the PNP system is replaced by a fourth-order elliptic equation and have obtained optimal convergence error bounds in energy norm.
Recently, Zhou and Xu \cite{ZX23} have applied an efficient first- and second-order time-stepping scheme based on the SAV approach to the PNP-NS system, and they have derived stability, discrete energy decay, discrete positivity preservation, and discrete mass conservation properties. It is to be noted that no error analysis has been carried out here.
In \cite{QW25}, Qin and Wang have discussed a second-order accurate scheme for both time and space based on a cell finite difference scheme and have shown that it preserves positivity, is uniquely solvable, is total-energy stable, and converges at an optimal rate under the assumption of high regularity of the solutions. 
Furthermore, in \cite{LL24}, a finite element method along with the backward Euler scheme has been applied to the PNP-NS system, and an error analysis has been derived, but it is sub-optimal in the $\ell^\infty(L^2)$-norm. 
Pan \textit{et. al.} \cite{PLZJH24} have developed a linear, second-order-accurate, positivity-preserving, and unconditionally energy-stable scheme based on a staggered-grid finite-difference method in the spatial direction. 
Subsequently, He and Chen \cite{HC25} have implemented a non-negativity-preserving, mass-conserving SAV-type pressure correction scheme and have derived an optimal convergence rate in the temporal direction only.
Recently, Li \textit{et. al.} \cite{LZW26} have developed a new linear, second-order accurate, fully-decoupled and unconditionally energy stable finite element method scheme based on two-step backward differentiation formula (BDF2) and rotational pressure-projection method. A Banach space-based mixed finite element method has been applied to the system in \cite{CGHRS24}, and the associated convergence analysis has been established.

The present work develops an efficient decoupled linearized scheme based on the discontinuous Galerkin method in the spatial direction. The discontinuous Galerkin finite element method, initially introduced in the 1970s  for the elliptic, parabolic and hyperbolic problems by Douglas and Dupont \cite{DD75}, Arnold \cite{Arn82}, and Reed and Hill \cite{RH73}, is becoming increasingly popular in the numerical approximation of a wide variety of mathematical problems. Its broad use can be attributed to its adaptability and noteworthy features, such as arbitrary-order precision, local mesh adaptivity, local mass conservation, and a comparatively simple implementation compared to other approaches, such as the finite volume technique.  The application of the discontinuous Galerkin technique to problems involving the Stokes and Navier-Stokes equations has been the subject of a significant amount of research, with noteworthy references found in  \cite{GRW05domain, GRW05spliting, MLR23, BGR23} and references, therein.  On the other hand, an implicit-explicit (IMEX) pressure correction scheme has been applied in the temporal direction in this paper. It is well known that an IMEX scheme helps linearize the nonlinear terms, and the pressure-correction scheme helps decouple the velocity and pressure of the Navier-Stokes part. The main idea of the pressure correction scheme is that we first find an intermediate solution by solving a liniearized convection-diffusion type problem, then we use that intermediate solution to find the pressure by solving a Poisson-type problem, and finally we update the velocity using the pressure and the intermediate solution. 
For this reason, the use of the pressure correction scheme in the context of the Navier-Stokes system and coupled Navier-Stokes system \cite{LSL22, WZWZ23, HC25} has received growing interest nowadays.

The main contributions of this paper are:
\begin{itemize}
    \item An efficient decoupled linearized scheme based on the discontinuous Galerkin method in the spatial direction and the  Euler incremental pressure correction known as pressure projection scheme is employed to discretize in the temporal direction is developed an analysed.  
    \item The discrete mass conservation property is established. 
	\item Optimal error bounds for the fully discrete solution of order $\mathcal{O}(h^{k+1}+\Delta t)$ in $L^2$-norm and order $\mathcal{O}(h^{k}+\Delta t)$ in energy-norm as well as $L^2$-norm for the fluid pressure are derived using a modified version of elliptic projection, Stokes projection, duality argument and the discrete Gronwall's type Lemma, where $h$ and $\Delta t$ are the spatial and temporal discretizing parameter and $k$ denotes the degree of polynomial for spatial approximations.
	\item Numerical simulations are performed to prove the accuracy of the fully discrete scheme, as well as to show mass and positivity preserving properties of the fully discrete ions and the energy dissipation of the proposed fully discrete scheme.
\end{itemize}

The rest of the paper is organized as follows: In section \ref{sec:pre}, we define the functional spaces, broken Sobolev spaces and their respective norms. We also define bilinear and trilinear forms, as well as standard inequalities and regularity assumptions. In Section \ref{sec:ds}, we construct discrete schemes based on the pressure-correction method along with the discontinuous Galerkin method. A few discrete inequalities and the properties of bilinear and trilinear forms are discussed here. Lastly, we define some projection operators and their approximation properties. Section \ref{Sec:ee} discusses on {\it a priori} error analysis for the fully discrete scheme. 
Finally, in Section \ref{Sec:numerical}, a few numerical results have been presented to test the accuracy of the schemes and support the discrete preserving properties.

Throughout the paper, we will use $C$ as a generic constant that does not depend on the discretizing parameter $h$ and $\Delta t$ but may depend on the given data. Let $a\lesssim b$ denotes $a\le Cb$.

\section{Preliminaries} \label{sec:pre}
%\se

We begin this section by introducing some standard functional spaces.
For a non-negative integer $m$ and $p$ with $1\le p\le \infty$, we denote by $W^{m,p}(\Omega)$ and $L^p(\Omega)$ the usual Sobolev spaces and the Lebesgue spaces with associated norms $\|\cdot\|_{W^{m,p}}$ and $\|\cdot\|_{L^p}$, respectively. If $p=2$, $W^{m,2}(\Omega)=H^m(\Omega)$ is the Hilbert spaces equipped with norm $\|\cdot\|_{H^m}$. We define $(\cdot,\cdot)$ as the $L^2$-inner product equipped with induced norm $\|\cdot\|.$
Let $H^m/\mathbb{R}$ be the quotient space of equivalent classes of functions in $H^m(\Omega)$ differ by constant with norm $\|\phi\|_{H^m/\mathbb{R}}=\inf_{c\in \mathbb{R}}\|\phi+c\|_{H^m}$. When $m=0$, it is denoted by $L_0^2(\Omega)$, that is, $L_0^2(\Omega) = \left\{\psi\in L^2(\Omega): \int_\Omega \psi(s) \rm{d}s = 0\right\}$. Further, $H_0^1(\Omega)$ is defined as a subspace of $H^1(\Omega)$ whose elements vanish on $\partial\Omega$ in the sense of trace. For simplicity, we denote $H^m(\Omega)$ as $H^m$ and its norm by $\|\cdot\|_{m}$. 
For our subsequent analysis, we denote the vector-valued function spaces by boldface letters such as
\begin{align*}
	\bH^m=[H^m(\Omega)]^2, \quad \bH_0^1 = [H_0^1(\Omega)]^2, \quad\mbox{and}\quad\bL^2 = [L^2(\Omega)]^2.
\end{align*}

\subsection{dG Formulation}
To motivate dG formulation, we consider a shape-regular family of triangulations $\Eh$ of $\bar{\Omega}$ \cite{Cia78}, consisting of triangles or quadrilaterals of maximum diameter $h$. Let $h_E$ denote the diameter of a triangle $E$ and $\pi_E$ the radius of the largest ball inscribed in $E$.  Then, the shape regular triangulation means there exists a constant $\varrho>0$ (mesh regularity parameter) such that
\begin{equation}
	\varrho h_E \leq \pi_E \quad \forall~ E\in \Eh.\label{regular}
\end{equation}
We denote by $\Gamma_h^0$ the set of all interior edges of $\Eh$ and set $\Gamma_h = \Gamma_h^0\cup\partial\Omega$. With each edge $e$, we associate a unit normal vector $\bn_e$. If $e$ is on the boundary $\partial\Omega$, then $\bn_e$ is taken to be the unit outward vector normal to $\partial\Omega$. Let $e$ be an edge shared by two elements $E_i$ and $E_j$ of $\Eh$; we associate with $e$, once and for all, a unit normal vector $\bn_e$ directed from $E_i$ to $E_j$. We define the jump $\jump{\psi}$ and the average $\avr{\psi}$ of a function $\psi$ on $e$ by
\[ \jump{\psi} = (\psi|_{E_i})-(\psi|_{E_j}), \quad \avr{\psi}=\frac{1}{2}\left((\psi|_{E_i})+(\psi|_{E_j})\right).\] 
If $e$ is adjacent to $\partial\Omega$, then the jump and the average of $\psi$ on $e$ coincide with the value of $\psi$ on $e$. Similar definitions apply to vector-valued functions.

\noindent
For our subsequent analysis, define the following discontinuous spaces \cite{GRW05domain}(broken Sobolev spaces):
\begin{align*}
	X^0 &= \left\{ \chi\in L_0^2(\Omega) : \chi|_E \in W^{2, \frac{4}{3}}(E)~~\forall~ E\in\Eh \right\},\\
	\bV &= \left\{ \bv\in \bL^2 : \bv|_E \in \left(W^{2, \frac{4}{3}}(E)\right)^2~~\forall~ E\in\Eh \right\},\\
	M &= \left\{ q\in L_0^2(\Omega) : q|_E \in W^{1, \frac{4}{3}}(E)~~\forall~ E\in\Eh \right\}.
\end{align*} 
For $\chi\in X_0$ or $M$, define the norm
\begin{align*}
	\vertiii{\chi} &= \left(\sum_{E\in \Eh} \|\nabla\chi\|_{L^2(E)}^{2} + \sum_{e\in \Gamma_h} \frac{\sigma_e}{h_e} \|\jump{\chi}\|_{L^2(e)}^{2}\right)^\frac{1}{2},
\end{align*}
and for $\bv\in \bV$, set
\begin{align*}
	\vertiii{\bv} &= \left(\sum_{E\in \Eh} \|\nabla\bv\|_{L^2(E)}^{2} + \sum_{e\in \Gamma_h} \frac{\sigma_e}{h_e} \|\jump{\bv}\|_{L^2(e)}^{2}\right)^\frac{1}{2},
\end{align*}
where $h_e = |e|: $ length of edge and $\sigma_e >0$ is the penalty parameter. 

Moving towards the dG method, we define the bilinear forms: $\cA_1: Z\times Z\to\R$,  $\cA_2:\bV\times\bV\to\R$ and $\cD:\bV\times M\to\R$ as
\begin{align}
	\cA_1(\psi,\chi) = & \sum_{E\in \Eh} \int_E \nabla\psi \cdot \nabla\chi \dx - \sum_{e\in\Gamma_h^0} \int_e \avr{\nabla\psi}\cdot\bn_e\jump{\chi}\ds \nonumber\\
	&\qquad\qquad - \sum_{e\in\Gamma_h^0} \int_e \avr{\nabla\chi}\cdot\bn_e \jump{\psi} \ds + \sum_{e\in\Gamma_h^0} \frac{\sigma_e}{h_e}\int_e \jump{\psi}\jump{\chi}\ds\quad\forall~ \psi, \chi \in Z, \label{bia1}\\
	\cA_2(\bv,\bw) = & \sum_{E\in \Eh} \int_E \nabla\bv : \nabla\bw \dx - \sum_{e\in\Gamma_h} \int_e \avr{\nabla\bv}\bn_e\cdot\jump{\bw}\ds \nonumber\\
	&\quad - \sum_{e\in\Gamma_h} \int_e \avr{\nabla\bw}\bn_e\cdot\jump{\bv} \ds + \sum_{e\in\Gamma_h} \frac{\sigma_e}{h_e}\int_e \jump{\bv}\cdot\jump{\bw}\ds \quad\forall~ \bv, \bw \in \bV, \label{bia2} \\
	\cD(\bv,q) = & \sum_{E\in\Eh} \int_E q \nabla\cdot \bv \dx - \sum_{e\in\Gamma_h} \int_e \avr{q}\bn_e\cdot\jump{\bv} \ds \nonumber\\
	= & -\sum_{E\in\Eh} \int_E \bv \cdot\nabla q \dx + \sum_{e\in\Gamma_h} \int_e \avr{\bv}\cdot\bn_e \jump{q} \ds \quad \forall~ \bv\in \bV, q\in M,
\end{align}
where $Z = X^0$ or $M$.

\noindent
The trilinear forms $\cN_1:\bV \times X^0 \times X^0 \to \R$, $\cN_2:\bV \times \bV \times \bV \to \R$, $\cG:X^0\times X^0\times X^0\to\R$ and $\cT:X^0\times X^0\times \bV\to\R$ are defined as
\begin{align}
	\cN_1(\bv,\psi, \chi) =& \sum_{E\in\Eh} \int_E (\bv\cdot\nabla\psi) \chi \dx + \frac{1}{2}\sum_{E\in\Eh}\int_E (\nabla\cdot\bv)\psi \chi \dx \nonumber
	\\
	&\quad - \sum_{e\in\Gamma_h^0} \int_{e} \avr{\bv}\cdot\bn_e \jump{\psi}\avr{ \chi} \ds  - \frac{1}{2}\sum_{e\in\Gamma_h^0}\int_e \jump{\bv}\cdot\bn_e \avr{\psi \chi} \ds \quad\forall~ \bv\in\bV, \psi, \chi\in X^0, \label{tribc}
	\\
	\cN_2(\bv,\bw,\bphi) =& \sum_{E\in\Eh}  \int_E (\bv\cdot\nabla)\bw\cdot\bphi \dx + \frac{1}{2}\sum_{E\in\Eh}\int_E (\nabla\cdot\bv)\bw\cdot\bphi \dx  \nonumber\\
	&~~ - \sum_{e\in\Gamma_h} \int_{e} \avr{\bv}\cdot\bn_e \jump{\bw}\cdot\avr{\bphi} \ds - \frac{1}{2}\sum_{e\in\Gamma_h}\int_e \jump{\bv}\cdot\bn_e \avr{\bw\cdot\bphi} \ds \quad\forall~\bv, \bw, \bphi \in \bV. \label{tribu} 
	\\
	\cG(\chi,\psi,\zeta) =& \sum_{E\in\Eh} \int_E \chi\nabla\psi\cdot\nabla\zeta \ds - \sum_{e\in\Gamma_h^0} \int_e \avr{\chi\nabla\psi}\cdot\bn_e \jump{\zeta} \ds\nonumber\\
	&\qquad - \sum_{e\in\Gamma_h^0} \int_e \avr{\chi\nabla\zeta}\cdot\bn_e \jump{\psi} \ds \quad \forall~ \chi, \psi,\zeta \in X^0. \label{defg}
	\\
	\cT(\chi,\psi,\bw) =& \sum_{E\in\Eh} \int_E \chi\nabla\psi\cdot\bw \ds - \sum_{e\in\Gamma_h^0} \int_e \avr{\chi}\jump{\psi}\avr{\bw}\cdot\bn_e \ds \quad \forall~ \chi, \psi \in X^0, ~ \bw\in \bV. \label{defT}
\end{align}

Now, the dG formulation of \eqref{eqphi}-\eqref{eqint} reads as: Find $(\phi, c_1, c_2, \bu, p)\in X^0 \times X^0 \times X^0 \times \bV \times M$ such that for all $t>0$
\begin{align}
	& \mu \cA_1(\phi,\psi) = (\tc_1-\tc_2,\psi) \quad\forall~\psi\in X^0, \label{dg1}\\
	&(\pt \tc_i,\chi) + \kappa_i \cA_1(\tc_i,\chi) + \cN_1(\bu, \tc_i, \chi) + \beta_i \cG(\tc_i+m_0, \phi, \chi) = 0 \quad\forall~\chi\in X^0, \label{dg2}\\
	&(\pt\bu,\bv) + \nu \cA_2(\bu,\bv) + \cN_2(\bu,\bu,\bv) - \cD(\bv,p) = -\cT((\tc_1-\tc_2),\phi,\bv)  \quad\forall~\bv\in\bV, \label{dg3}\\
	&\cD(\bu,q)=0\quad\forall~ q\in M, \label{dg4}
\end{align}
where $\tc_i = c_i-m_0,~i=\{1,2\}$ is a new variable with $m_0=\frac{1}{|\Omega|}\int_\Omega c_{10}\dx =\frac{1}{|\Omega|}\int_\Omega c_{20}\dx$ and the new initial data $\tc_{i0} = c_{i0}-m_0$.

\subsection{Some Useful Inequalities}
The following standard trace inequalities will be used for our subsequence analysis.
\begin{lemma}[Continuous trace inequality {\cite[Section 2.1.3]{Riv08}}]\label{CTI}
	For each element $E\in\Eh$ with diameter $h_E$, there exists a constant $C$ may depends on dimension $d$ and mesh regularity parameter $\varrho$, but independent of $h_E$ and $\psi$ such that the followings hold:
	\begin{align*}
		\|\psi\|_{L^2(e)} & \lesssim h_E^{-\frac{1}{2}}\|\psi\|_{L^2(E)} + h_E^{\frac{1}{2}}\|\nabla\psi\|_{L^2(E)} \quad \forall~\psi\in H^1(E),\\
		\|\nabla\psi\|_{L^2(e)} & \lesssim h_E^{-\frac{1}{2}}\|\nabla\psi\|_{L^2(E)} + h_E^{\frac{1}{2}}\|\nabla^2\psi\|_{L^2(E)} \quad \forall~\psi\in H^2(E),
	\end{align*}
	where $e$ is an edge of the element $E$. These results also hold for the vector-valued function $\bpsi$. Additionally, if $\psi\in W^{1,p}(E)$, then the following holds \cite[(1.18)]{DE11} %\cite[Lemma 1.31]{}
    \begin{align} \label{LPCTI}
        \|\psi\|_{L^p(e)} \lesssim \|\psi\|_{L^p(E)}^{1-\frac{1}{p}}\|\nabla\psi\|_{L^p(E)}^{\frac{1}{p}}.
    \end{align}
\end{lemma}
We recall the following $L^p$ bound in broken Sobolev space {\cite[Lemma 6.2]{GRW05domain}, \cite[Remark 3.5]{GRW05spliting}, \cite[Lemma 5.2]{MLR23}}: For any $p\in [2, \infty)$, there exist a constant $C_p$ depends on $p$ but not in $h$ such that 
	\begin{equation} \label{Lptoenergy}
		\|\psi\|_{L^p(\Omega)} \le C_p \vertiii{\psi} \quad \forall~\psi_h\in X^0~ \mbox{or}~ \bV.
	\end{equation}
    
\noindent
We also recall the Gagliardo-Nirenberg inequality \cite[pp. 638]{HS00}:
\begin{equation} \label{GNI}
	\|\psi\|_{L^p(\Omega)} \lesssim \|\psi\|^{2/p}\|\nabla\psi\|^{1-2/p}\quad\forall~\psi\in H^1(\Omega),
\end{equation}
where $2\le p<\infty$ and also the Agmon's inequality \cite[pp. 638]{HS00}:
\begin{equation}\label{AI}
	\|\psi\|_{L^\infty(\Omega)} \lesssim \|\psi\|^{1/2}\|\nabla^2\psi\|^{1/2}\quad\forall~\psi\in H^2(\Omega).
\end{equation}

%\newpage
Now, we prove some boundedness property of $\cN_1(\cdot,\cdot,\cdot)$, $\cN_2(\cdot,\cdot,\cdot)$,  $\cG(\cdot,\cdot,\cdot)$ and $\cT(\cdot,\cdot,\cdot)$:
\begin{lemma} \label{lem:estnon}
Define $S:\bV\times \bV \times \bV \to \R$ (or $S:\bV\times X^0 \times X^0 \to \R$, or $S:X^0\times X^0 \times X^0 \to \R$) and $T:\bV\times \bV\to \R$ (or $T:\bV\times X^0\to \R$ or $T:X^0\times X^0\to \R$) such that
\begin{align*}
	S(\cdot,\cdot,\cdot)  = \left\{\|\cdot\|_{\bL^{p}(E)} \|\cdot\|_{\bL^q(E)} \|\cdot\|_{\bL^r(E)}: \frac{1}{p}+\frac{1}{q}+\frac{1}{r}=1, ~~1\le p,q,r \le \infty\right\}
\end{align*}
and
\begin{align*}
	T_p(\cdot)  = \left\{\|\cdot\|_{\bL^{p}(E)}^{1-\frac{1}{p}} \|\nabla\cdot\|_{\bL^p(E)}^{\frac{1}{p}}\right\} ~\text{and}~ T_q(\cdot)  = \left\{\|\cdot\|_{\bL^{q}(E)}^{1-\frac{1}{q}} \|\nabla\cdot\|_{\bL^q(E)}^{\frac{1}{q}}\right\}, \text{for} ~ 2\le p,q \le \infty ~\text{with}~ \frac{1}{p}+\frac{1}{q}=\frac{1}{2}.
\end{align*}
Also, we set 
		$$J(\cdot) = \left(\sum_{e\in\Gamma_h}  \frac{\sigma_e}{|e|}\|\jump{\,\cdot\,}\|_{\bL^2(e)}^2 \right)^{\frac{1}{2}}.$$ 
% {\xb
% 	Define $S:\bV\times \bV \times \bV \to \R$ (or $S:\bV\times X^0 \times X^0 \to \R$, or $S:X^0\times X^0 \times X^0 \to \R$), $T:\bV \to \R$ (or $T:X^0\to \R$) and and $J:\bV \to \R$ (or $T:X^0\to \R$) such that for all $0\le p,q,r \le \infty$ with $\frac{1}{p}+\frac{1}{q}+\frac{1}{r}=1$
% 	\begin{align*}
% 		S(\cdot,\cdot,\cdot) & = \left\{\|\cdot\|_{\bL^{p}(E)} \|\cdot\|_{\bL^q(E)} \|\cdot\|_{\bL^r(E)}\right\}, ~~ 
% 		T_p(\cdot) & = \left\{\|\cdot\|_{\bL^{p}(E)}^{1-\frac{1}{p}} \|\cdot\|_{\bL^p(E)}^{\frac{1}{p}}\right\}, ~~ J_r(\cdot) = \left(\sum_{e\in\Gamma_h}  \frac{\sigma_e}{|e|}\|\jump{\,\cdot\,}\|_{\bL^r(e)}^2 \right)^{\frac{1}{2}}.
% 	\end{align*}
	Then, the followings hold for all $\bv,\bw,\bphi\in \bV$ when $i=2$ and $\bv\in\bV$, $\bw, \bphi\in X^0$ when $i=1$
	\begin{align}
		|\cN_i(\bv, \bw, \bphi)| \lesssim &~ \sum_{E\in\Eh}S(\bv, \nabla\bw, \bphi) + J(\bw)\sum_{E\in\Eh} h_E^{\frac{1}{2}} T_p(\bv) T_q(\bphi)  \nonumber\\
		&\quad+ \sum_{E\in\Eh}S(\nabla\bv, \bw, \bphi) + J(\bv)\sum_{E\in\Eh} h_E^{\frac{1}{2}} T_p(\bw) T_q(\bphi) \label{NN1} \\
		|\cN_i(\bv, \bw, \bphi)| \lesssim &~ \sum_{E\in\Eh}S(\bv, \nabla\bw, \bphi) + J(\bw)\sum_{E\in\Eh} h_E^{\frac{1}{2}} T_p(\bv) T_q(\bphi) \nonumber \nonumber\\
		&\quad+ \sum_{E\in\Eh}S(\bv, \bw, \nabla\bphi) + J(\bphi)\sum_{E\in \Eh} h_E^{\frac{1}{2}} T_p(\bv) T_q(\bw). \label{NN2}
	\end{align}
	Also, the followings hold true: For all $\chi, \psi,\zeta \in X^0$
	\begin{align}
		|\cG( \chi, \psi, \zeta)| \lesssim &~ \sum_{E\in\Eh}S( \chi, \nabla\psi, \nabla\zeta) + J(\psi)\sum_{E\in\Eh} h_E^{\frac{1}{2}} T_p(\chi) T_q( \nabla\zeta) + J(\zeta)\sum_{E\in\Eh} h_E^{\frac{1}{2}} T_p(\chi) T_q(\nabla\psi), \label{GG}
	\end{align}
    and for all $\chi, \psi\in X^0$, $\bw\in\bV$
    \begin{align}
        |\cT( \chi, \psi, \bw)| \lesssim &~ \sum_{E\in\Eh}S( \chi, \nabla\psi, \bw) + J(\psi)\sum_{E\in\Eh}h_E^{\frac{1}{2}} T_p(\chi) T_q(\bw). \label{TT}
	\end{align}
	
\end{lemma}
\begin{proof}
	For $i=2$, we recall the definition of $\cN_2(\cdot,\cdot,\cdot)$: For all $\bv, \bw, \bphi \in \bV$,
	\begin{align} \label{tribuddd}
		\cN_2(\bv,\bw,\bphi) =& \sum_{E\in\Eh}  \int_E (\bv\cdot\nabla)\bw\cdot\bphi \dx + \frac{1}{2}\sum_{E\in\Eh}\int_E (\nabla\cdot\bv)\bw\cdot\bphi \dx  \nonumber\\
		&~~ - \sum_{e\in\Gamma_h} \int_{e} \avr{\bv}\cdot\bn_e \jump{\bw}\cdot\avr{\bphi} \ds - \frac{1}{2}\sum_{e\in\Gamma_h}\int_e \jump{\bv}\cdot\bn_e \avr{\bw\cdot\bphi} \ds = I_1+I_2+I_3+I_4.
	\end{align}
	From the H\"{o}lder inequality, it is obvious that
	\begin{align}  \label{tribuddd1}
		|I_1|+|I_2| \le \sum_{E\in \Eh} S(\bv, \nabla\bw, \bphi) + \frac{1}{2}\sum_{E\in \Eh} S(\nabla\bv, \bw, \bphi).
	\end{align}
    
	Using the H\"{o}lder inequality with $\|\bn_e\|_{\bL^\infty(e)}\le 1$ and the Cauchy-Schwarz inequality, the term $I_3$ can be bounded as 
		\begin{align}  \label{tribuddd2}
			|I_3| & \le \sum_{e\in\Gamma_h} \|\avr{\bv}\|_{\bL^p(e)}\|\bn_e\|_{\bL^\infty(e)} \|\jump{\bw}\|_{\bL^2(e)}\|\avr{\bphi}\|_{\bL^q(e)} \nonumber\\
			& \lesssim \left(\sum_{e\in\Gamma_h} \frac{\sigma_e}{|e|} \|\jump{\bw}\|_{\bL^2(e)}^2\right)^{\frac{1}{2}}
			\left(\sum_{e\in\Gamma_h} \frac{|e|}{\sigma_e} \|\avr{\bv}\|_{\bL^p(e)}^2\|\avr{\bphi}\|_{\bL^q(e)}^2\right)^{\frac{1}{2}}
		\end{align} 
		Let $e$ be the edge shared by triangle $E_e^1$ and $E_e^2$, then a use of the trace inequality \eqref{LPCTI} yields
		\begin{align}
			\|\avr{\bv}\|_{\bL^p(e)} \lesssim  \frac{1}{2}\left(\|\bv|_{E_e^1}\|_{\bL^p(e)} + \|\bv|_{E_e^2}\|_{\bL^p(e)}\right) \lesssim \|\bv\|_{\bL^p(E_e^1)}^{1-\frac{1}{p}}\|\nabla\bv\|_{\bL^p(E_e^1)}^{\frac{1}{p}} + \|\bv\|_{\bL^p(E_e^2)}^{1-\frac{1}{p}}\|\nabla\bv\|_{\bL^p(E_e^2)}^{\frac{1}{p}}. \label{tribudddd11}
		\end{align}
		We use \eqref{tribudddd11} in \eqref{tribuddd2} for both the average terms, and then apply the Cauchy-Schwarz inequality to find
		\begin{align} \label{tribuddd3}
			|I_3| \lesssim J(\bw)  \sum_{E\in\Eh} h_E^{\frac{1}{2}} \|\bv\|_{\bL^p(E)}^{1-\frac{1}{p}}\|\nabla\bv\|_{\bL^p(E)}^{\frac{1}{p}} \|\bphi\|_{\bL^q(E)}^{1-\frac{1}{q}}\|\nabla\bphi\|_{\bL^q(E)}^{\frac{1}{q}} = J(\bw)  \sum_{E\in\Eh} h_E^{\frac{1}{2}} T_p(\bv) T_q(\bphi).
		\end{align}
	An application of $\avr{ab}=\avr{a}\avr{b}+\frac{1}{4}\jump{a}\jump{b}$ with the H\"{o}lder inequality in $I_4$ yields
	\begin{align}  \label{tribuddd4}
		|I_4| & \lesssim \sum_{e\in\Gamma_h}\int_e \jump{\bv}\cdot\bn_e \left(\avr{\bw}\cdot\avr{\bphi}+\frac{1}{4}\jump{\bw}\cdot\jump{\bphi}\right) \ds \nonumber\\
        & \lesssim  \sum_{e\in\Gamma_h} \|\jump{\bv}\|_{\bL^2(e)} \left(\|\avr{\bw}\|_{\bL^p(e)}\|\avr{\bphi}\|_{\bL^q(e)}+\|\jump{\bw}\|_{\bL^p(e)}\|\jump{\bphi}\|_{\bL^q(e)}\right) \ds.
	\end{align}
    In the similar way of \eqref{tribudddd11}, one can write 
    \begin{align}
        \|\jump{\bv}\|_{\bL^p(e)} \lesssim  \left(\|\bv|_{E_e^1}\|_{\bL^p(e)} + \|\bv|_{E_e^2}\|_{\bL^p(e)}\right) \lesssim \|\bv\|_{\bL^p(E_e^1)}^{1-\frac{1}{p}}\|\nabla\bv\|_{\bL^p(E_e^1)}^{\frac{1}{p}} + \|\bv\|_{\bL^p(E_e^2)}^{1-\frac{1}{p}}\|\nabla\bv\|_{\bL^p(E_e^2)}^{\frac{1}{p}}. \label{tribudddd221}
    \end{align}
    Use \eqref{tribudddd11} for both the average terms and \eqref{tribudddd221} for both the jump terms of \eqref{tribuddd4}, one can bound $I_4$ as
    \begin{align}  \label{tribuddd41}
		|I_4| & \lesssim J(\bv)\sum_{E\in \Eh} h_E^{\frac{1}{2}} T_p(\bw) T_q( \bphi).
	\end{align}
	Substituting \eqref{tribuddd1}-\eqref{tribuddd41} in \eqref{tribuddd} complete the first proof of \eqref{NN1} for $i=2$. For the second one, that is \eqref{NN2} for $i=2$, first we use integration by parts in the second term of \eqref{tribuddd} and then, rewrite the boundary terms, we find that
	\begin{align} \label{tribuanother}
		\cN_2(\bv,\bw,\bphi) =& \frac{1}{2}\sum_{E\in\Eh}  \int_E (\bv\cdot\nabla)\bw\cdot\bphi \dx - \frac{1}{2}\sum_{E\in\Eh}\int_E (\bv\cdot \nabla)\bphi\cdot\bw \dx  \nonumber\\
		&~~ - \frac{1}{2}\sum_{e\in\Gamma_h} \int_{e} \avr{\bv}\cdot\bn_e \jump{\bw}\cdot\avr{\bphi} \ds + \frac{1}{2}\sum_{e\in\Gamma_h}\int_e \avr{\bv}\cdot\bn_e \jump{\bphi}\cdot\avr{\bw} \ds.
	\end{align}
	Now, proceeding similarly to the first case, we derive the second estimate \eqref{NN2} for $i=2$. The proof for $i=1$ follows the same pattern as the case $i=2$. 
    To establish \eqref{GG} and \eqref{TT}, we refer to the definition \eqref{defg} and \eqref{defT}, respectively. We can bound all terms on elements using the Hölder inequality, and all terms on edges exactly as in $I_4$ and $I_3$. This completes the proof.	
\end{proof}

\section{Discrete Formulation} \label{sec:ds}
This section focuses on the completely discrete schemes.

\subsection{Temporal Discretization}
For time discretization, we take a uniform partition of the time interval $[0, T]$ as $0=t_0<t_1<\cdots<t_N=T$ with time step $\Delta t = t_{m}-t_{m-1}, 0\le m\le N-1$ and $t_{m} = m\Delta t$. For a continuous function $v(t)$, we set $v^m=v(t_{m})$.

We now apply a pressure projection scheme, mainly the Euler incremental pressure correction scheme.Then the discrete formulation reads as:

\noindent
\textbf{Step I:} Given $(\tc_1^{m-1}, \tc_2^{m-1}, \bu^{m-1}, p^{m-1}) \in X^0 \times X^0 \times \bV \times M$, find $(\phi^{m}, \tc_1^{m}, \tc_2^{m}, \hbu^{m})\in X^0\times X^0 \times X^0\times \bV $ such that for $m\ge 0$
\begin{align}
	& -\mu\Delta \phi^{m} = \tc_1^{m-1}-\tc_2^{m-1}, \label{td1}\\
	& \frac{\tc_i^{m}-\tc_i^{m-1}}{ \Dt} - \kappa_i\Delta \tc_i^{m} + (\bu^{m-1}\cdot\nabla)\tc_i^{m} = \beta_i \nabla\cdot((\tc_i^{m-1}+m_0)\nabla\phi^{m}), \label{td2}\\
	&\frac{\hbu^{m}-\bu^{m-1}}{ \Dt} - \nu \Delta \hbu^{m} + (\bu^{m-1}\cdot \nabla \hbu^{m}) + \nabla p^{m-1} = -(\tc_1^{m-1}-\tc_2^{m-1})\nabla\phi^{m}, \label{td4}
\end{align}
with boundary condition $ \nabla \tc_i^{m} \cdot \bn = \nabla\phi^{m} \cdot \bn = 0, \hbu = 0$.

\noindent
\textbf{Step II:} Given $(\hbu^{m}, p^{m-1}) \in \bV \times M$, seek $(\bu^{m}, p^{m})\in \bV\times M$ such that
\begin{align}
	& \frac{\bu^{m}-\hbu^{m}}{ \Dt} + \nabla \left(p^{m}-p^{m-1}\right) = 0 \label{td6}\\
	& \nabla \cdot \bu^{m}= 0 \label{td7}\\
	& \bu^{m}|_{\partial\Omega} = 0. \label{td8}
\end{align}
%Here, the initial data $\tc_{ih}^{0}, \bu_h^0, p_h^{0}$ are %the appropriate approximation of $\tc_{i0}, \bu_0, p_0$ %respectively.
Taking divergence on the first equation of \textbf{Step II} and using the second equation, one may find
\begin{align} \label{recp}
	\Delta \left(p^{m}-p^{m-1}\right) = \frac{1}{ \Dt}\nabla \cdot \hbu^{m}
\end{align}
From \eqref{recp} with boundary conditions $\nabla(p^{m}-p^{m-1})\cdot \bn = 0$, we obtain $p^{m}$ first and then, using this $p^{m}$ we update $\bu^{m}$ by 
\begin{align}
	& \bu^{m} = \hbu^{m} - { \Dt} \nabla \left(p^{m}-p^{m-1}\right).
\end{align}

\subsection{Fully Discrete dG Formulation}
For space discretization, we first consider the dG finite element spaces $X^0_h \times \bV_h \times M_h \subset X^0 \times \bV \times M$, which are defined as follows: For any integer $k\ge 1$,
\begin{align*}
	X^0_h &= \left\{ \chi_h\in L_0^2(\Omega) : \chi_h|_E \in \mathbb{P}_{k}(E)~~\forall~ E\in\Eh \right\},\\
	\bV_h &= \left\{ \bv_h\in \bL^2 : \bv_h|_E \in \left(\mathbb{P}_{k}(E)\right)^2~~\forall~ E\in\Eh \right\},\\
	M_h &= \left\{ q_h\in L_0^2(\Omega) : q_h|_E \in \mathbb{P}_{k-1}(E)~~\forall~ E\in\Eh \right\},
\end{align*} 
where $\mathbb{P}_k(E)$ is the polynomial space of order less or equal to $k$ over $E$.

For our subsequent analysis, we use the following standard discrete versions of the trace inequalities, the inverse inequality, and the discrete Poincaré inequality.
\begin{itemize}
	\item Discrete trace inequality {\cite[Lemma 1.46]{DE11}}: 
	For each element $E\in\Eh$ with diameter $h_E$, there exists a constant $C,$ which may depend on polynomial degree $k$, dimension $d$ and mesh regularity parameter $\varrho$, but independent of $h_E$ such that the following holds: For all $ \psi_h\in X_h^0 ~\mbox{or}~\bV_h$
	\begin{align} \label{DTI}
		\| \psi_h\|_{L^2(e)} \lesssim h_E^{-\frac{1}{2}}\| \psi_h\|_{L^2(E)}, \quad \text{and}\quad 
		\|\nabla \psi_h\|_{L^2(e)}  \lesssim h_E^{-\frac{1}{2}}\|\nabla \psi_h\|_{L^2(E)},
	\end{align}
	where $e$ is an edge of the element $E$. The non-Hilbertian version of the discrete trace inequality is as follows \cite[Lemma 1.52]{DE11}: For $1\le p \le \infty$,
	\begin{align} \label{pDTI}
		\| \psi_h\|_{L^p(e)} & \lesssim h_E^{-\frac{1}{p}}\| \psi_h\|_{L^p(E)}.
	\end{align}
	\item Inverse inequality {\cite[Theorem 3.2.6]{Cia78}}: 
	For each element $E\in \Eh$ with diameter $h_E$, the following inverse hypothesis holds for all discrete $\psi_h\in  X_h^0 ~\mbox{or}~ \bV_h$,
	\begin{align} \label{inv.hypo}
		\| \psi_h\|_{W^{m,p}(E)^d} \lesssim h_E^{n-m-d(\frac{1}{q}-\frac{1}{p})} \|\psi_h\|_{W^{n,q}(E)^d},
	\end{align}
	where $0\le n \le m \le 1$, $0\le q \le p \le \infty$, $\| \cdot\|_{W^{m,p}(E)^d}$ is the norm in Sobolev space $W^{m,p}(E)^d$.
	Using the inverse inequality, one can easily show the following bound for the energy norm by $L^2$-norm:
	\begin{equation} \label{inv.DG}
		\vertiii{\psi_h} \lesssim h^{-1}\|\psi_h\| \quad \forall \psi_h \in X_h^0 ~\text{or}~ \bV_h.
	\end{equation}
	\item Discrete Poincar\'e inequality {\cite[Corollary 5.4, pp 192]{DE11}}:  
	For a positive constant $C>0$ independent of $h$ such that for all $\psi_h\in X_h^0$ (or $\bV_h$), the following holds
	\begin{align}\label{DPI}
		\|\psi_h\| \lesssim \vertiii{\psi_h}.
	\end{align}
\end{itemize}

Now, we apply dG method to the system \eqref{td1}-\eqref{td8}, then the fully discrete dG formulation reads as:

\noindent
\textbf{Step I:} Knowing $(\tc_{1h}^{m-1}, \tc_{2h}^{m-1}, \bu_h^{m-1}, p_{h}^{m-1}) \in X_h^0 \times X_h^0 \times \bV_h\times M_h$, find $\phi_h^{m}\in X_h^0$ in \eqref{fds11}, $\tc_{ih}^{m}\in X_h^0$ from \eqref{fds12} and $\hbu_{h}^{m}\in\bV_h $ from \eqref{fds2} such that for $m\ge 0$ and $i=1,2$
\begin{align}
	&  \mu\cA_1(\phi_h^{m},\psi_h) = (\tc_{1h}^{m-1}-\tc_{2h}^{m-1},\psi_h) \quad\forall~\psi_h\in X_h^0\label{fds11}\\
	\left(\frac{\tc_{ih}^{m}-\tc_{ih}^{m-1}}{ \Dt}, \chi_{ih}\right) + &\kappa_i\cA_1(\tc_{ih}^{m}, \chi_{ih}) + \cN_1(\bu_h^{m-1}, \tc_{ih}^{m}, \chi_{ih}) + \beta_i \cG( \tc_{ih}^{m-1}\hspace{-0.5em}+m_0, \phi_h^{m}, \chi_{ih}) = 0 ~\forall \chi_{ih}\in X_h^0 \label{fds12} \\
	&\left(\frac{\hbu_{h}^{m}-\bu_{h}^{m-1}}{ \Dt}, \bv_{h}\right) + \nu \cA_2(\hbu_{h}^{m}, \bv_{h}) + \cN_2(\bu_{h}^{m-1}, \hbu_{h}^{m}, \bv_{h}) - \cD(\bv_{h}, p_{h}^{m-1}) \nonumber\\
	& \qquad\qquad\qquad\qquad= -\cT((\tc_{1h}^{m-1}-\tc_{2h}^{m-1}), \phi_{h}^{m}, \bv_{h}) \quad\forall \bv_{h}\in\bV_{h}. \label{fds2}
\end{align}
\textbf{Step II:} Given $(\hbu_{h}^{m}, p_{h}^{m-1}) \in \bV_{h} \times M_{h}$, seek $p^{m}\in M_h$ such that
\begin{align}
	& \cA_1(p_{h}^{m}-p_{h}^{m-1}, q_{h}) = - \frac{1}{ \Dt}\cD(\hbu_{h}^{m}, q_{h}) \quad\forall~ q_{h} \in M_{h}. \label{fds3}
\end{align}
\textbf{Step III:} Given $(\hbu_{h}^{m}, p_{h}^{m}) \in \bV_{h} \times M_{h}$, update $\bu_{h}^{m}\in \bV_{h}$ such that
\begin{align}
	& \left(\frac{\bu_{h}^{m}-\hbu_{h}^{m}}{ \Dt}, \bw_{h}\right) -\cD \left(\bw_{h}, p_{h}^{m}-p_{h}^{m-1}\right) = 0 \quad \forall~ \bw_{h}\in \bV_{h}. \label{fds4}
\end{align}
\textbf{Step IV:} Update the solution $(c_{1h}^m, c_{2h}^{m})\in X_h^0\times X_h^0$ such that 
\begin{equation} \label{fds5}
	c_{1h}^{m} = \tc_{1h}^{m} + m_0 \quad \text{and}\quad  c_{2h}^{m} = \tc_{2h}^{m} + m_0.
\end{equation}

\noindent
Here, the initial data $\tc_{ih}^{0}, \bu_h^0, p_h^{0}$ are the appropriate approximation of $\tc_{i0}, \bu_0, p_0$ respectively.
To show the well-posedness of the above discrete system, first, we discuss the following properties:
\begin{lemma}[Discrete coercivity{\cite[Lemma 4.12]{DE11}}]\label{lem:prop_bilinear}
	For some sufficiently large positive constant $\sigma_e$ with $\sigma_e >\sigma_e^*$, there exists some $\gamma^* >0$ independent of $h$ such that the following hold
	\begin{align*}
		\cA_1(\chi_h,\chi_h) & \ge \gamma^* \vertiii{\chi_h}^2, \quad \chi_h\in X_h^0 ~\mbox{or}~ M_h,\\
		\cA_2(\bv_h,\bv_h) & \ge \gamma^* \vertiii{\bv_h}^2, \quad \bv_h\in \bV_h,
		%       \\
		% \cA_1(q_h,q_h) \ge \gamma^*\vertiii{q_h}^2, \quad q_h\in M_h.
	\end{align*}
	Additionally, for any $C>0$ independent of $h$, the following continuity properties \cite[Lemma 4.16]{DE11} hold:
	\begin{align*}
		\cA_1(\chi_h,\psi_h) &\le C\vertiii{\chi_h} \vertiii{\psi_h}, \quad \chi_h,\psi_h\in X_h^0~\mbox{or}~ M_h,\\
		\cA_2(\bv_h,\bw_h) &\le C\vertiii{\bv_h} \vertiii{\bw_h}, \quad \bv_h,\bw_h\in \bV_h, 
		%       \\
		% \cA_1(q_h, z_h) &\le C\vertiii{q_h} \vertiii{z_h}, \quad q_h,z_h\in M_h.
	\end{align*}
\end{lemma}

\noindent
We also state below the discrete inf-sup condition \cite[Theorem 6.8]{Riv08}:
\begin{lemma}\label{lem:infsup}
	There exists a constant $\beta^*>0$, independent of $h$ such that 
	\begin{equation*}
		\inf_{q_h\in M_h}\sup_{\bv_h\in\tilde{\bV}_h} \frac{\cD(\bv_h,q_h)}{\vertiii{\bv_h}\|q_h\|} \ge \beta^*,
	\end{equation*}
	where
	\[\tilde{\bV}_h = \left\{\bv_h\in\bV_h: \jump{\bv_h}|_e\cdot \bn_e = 0 \quad\forall~ e \in \Gamma_h\right\}.\]
\end{lemma}

We now state the positivity and boundedness properties of trilinear forms in the following lemmas. We refer to \cite[Lemma 6.39, Lemma 6.40]{DE11} for proof.
\begin{lemma}[Skew-symmetric property] \label{lem:skew}
	Using integration by parts, the  trilinear forms $\cN_1(\cdot,\cdot,\cdot)$ and $\cN_2(\cdot,\cdot,\cdot)$ defined in \eqref{tribc}  and \eqref{tribu} satisfy the followings:
	\begin{align}
		\cN_1(\bv_h,\psi_h,\chi_h) = -\cN_1(\bv_h,\chi_h,\psi_h) \quad \mbox{and} \quad \cN_2(\bv_h,\bw_h,\bphi_h) = - \cN_2(\bv_h,\bphi_h,\bw_h).
	\end{align}
	In particular, if $\psi_h=\chi_h$, then $\cN_1(\bv_h,\psi_h,\psi_h)= 0$ and if $\bw_h=\bphi_h$, then $\cN_2(\bv_h,\bw_h,\bw_h)= 0$.
\end{lemma}

\begin{lemma}[Boundedness of trilinear terms] \label{lem:bdd}
	There exists a positive constant $C$ such that the nonlinear terms $\cN_1(\cdot,\cdot,\cdot)$ and $\cN_2(\cdot,\cdot,\cdot)$ defined in \eqref{tribc}  and \eqref{tribu} satisfy the followings:
	\begin{align}
		\cN_1(\bv_h,\psi_h, \chi_h) &\lesssim \vertiii{\bv_h} \vertiii{\psi_h} \vertiii{ \chi_h} \quad\forall~ \bv_h\in\bV_h, \psi_h, \chi_h\in X_h^0, \\
		\cN_2(\bv_h,\bw_h,\bphi_h) &\lesssim \vertiii{\bv_h} \vertiii{\bw_h} \vertiii{\bphi_h} \quad\forall~\bv_h, \bw_h, \bphi_h \in \bV_h.
	\end{align}
\end{lemma}

\begin{lemma}[Boundedness of nonlinear terms] \label{lem:gbdd}
	There exists a positive constant $C$ such that the terms $\cG(\cdot,\cdot,\cdot)$ and $\cT(\cdot,\cdot,\cdot)$  defined in \eqref{defg} and \eqref{defT} satisfy the following:
	\begin{align}
		\cG(\psi_h,\phi_h,\chi_h) &\lesssim \|\psi_h\|_{L^\infty} \vertiii{\phi_h} \vertiii{\chi_h} \quad\forall~ \chi_h,\phi_h, \psi_h\in X_h^{0} \\
		\cT(\psi_h,\phi_h,\bw_h) &\lesssim \|\psi_h\|_{L^\infty} \vertiii{\phi_h} \|\bw_h\| \quad\forall~\phi_h, \psi_h\in X_h^{0}, \bw\in \bV_h.
	\end{align}
\end{lemma}
Now, the discrete system \eqref{fds11}-\eqref{fds5} yields a system of linear algebraic equations. A use of the discrete coercivity (Lemma \ref{lem:prop_bilinear}) with  the discrete discrete inf-sup condition (Lemma \ref{lem:infsup}), and the boundedness properties from Lemmas \ref{lem:bdd}, \ref{lem:gbdd} shows the existence of unique dG solutions $(\phi_{h}^{m}, c_{1h}^{m}, c_{2h}^{m},\bu_{h}^{m},p_{h}^{m})\in X_h^0 \times X_h^0 \times X_h^0 \times \bV_h \times M_h$ at each time level $t_m$ satisfying \eqref{fds11}-\eqref{fds5}.

\noindent 
Now, one can easily prove the following qualitative property. 
\begin{lemma}\label{lem:mass}
	The fully discrete solutions $\tc_{1h}^{m}$ and $\tc_{2h}^{m}$ satisfy the following mass conservation property:
	\begin{equation}
		\sum_{E\in \Eh}\int_E \tc_{1h}^{m}  \dx = \sum_{E\in \Eh}\int_E c_{10}  \dx = \sum_{E\in \Eh}\int_E \tc_{2h}^{m}  \dx = \sum_{E\in \Eh}\int_E c_{20}  \dx = m_0|\Omega|.
	\end{equation}
\end{lemma}

\noindent
\textbf{Projection Operator:}

\noindent
Define the elliptic projection operator $R_h\phi\in X_h$ for a given $\phi$ as
\begin{align} \label{eqritz}
	\cA_1(\phi-R_h\phi, \psi_h) = 0 \quad \forall~ \psi_h \in X_h^{0}\quad \text{with}\quad (\phi-R_h\phi,1) = 0,
\end{align}
which satisfies the following approximation properties (\cite[(1.34)]{GRW05spliting}).
\begin{lemma}\label{lem:ritz}
	For $\phi\in H^s$, there exists a positive constant $C$ independent of $h$ such that the followings hold for all $s\in\{1,2,\dots,k+1\}$:
	\begin{align*}
		\|\phi-R_h\phi\| + h \vertiii{\phi-R_h\phi} \le C h^{s}\|\phi\|_{H^{s}}.
	\end{align*}
\end{lemma}
\noindent
We now define a modified elliptic projection operator $Q_h\tc_{i}\in X_h^0$ of the solution $\tc_{i}$ of \eqref{dg2} for a given $\phi$ and $R_h\phi$, satisfying
\begin{align}\label{modi:eq}
	\kappa_{i} \cA_1(\tc_{i}-Q_h\tc_{i}, \chi_h) +  \beta_i \cG(\tc_{i}+m_0, \phi-R_h\phi,\chi_h) = 0 \quad \forall~ \chi_h \in X_h^{0}.
\end{align}
Using the duality argument, one can easily find the following approximation properties for the operator mentioned above \cite{BHP25}.
\begin{lemma}\label{lem:modi}
	There exists a positive constant $C$ independent of $h$ such that the followings hold for all $s\in\{1,2,\dots,k+1\}$:
	\begin{align*}
		\| \tc_{i}-Q_h \tc_{i}\| + h  \vertiii{\tc_{i}-Q_h \tc_{i}} \le C h^{s}\left(\| \tc_{i}\|_{H^{s}}+\|\phi\|_{H^s}\right)\quad\forall~ \tc_{i}, \phi\in H^s.
	\end{align*}
\end{lemma}

\noindent
We also define the Stokes operator  $(\bS_h\bu,\bS_h p)\in \bV_h \times M_h$ for the weak solution $(\bu(t), p(t))$ of \eqref{dg3}-\eqref{dg4}  satisfying
\begin{equation}\label{stokesproj}
	\left.
	\begin{aligned}
		&\nu \cA_2(\bu-\bS_h\bu, \bv_h) - \cD(\bv_h, p-\bS_hp) = 0 \quad \forall~ \bv_h \in \bV_h, \\
		& ~\cD(\bu-\bS_h\bu, q_h) = 0 \quad \forall~ q_h\in M_h.
	\end{aligned}
	\right\}
\end{equation}
One can easily derive the following approximation properties for the operator mentioned above. For a proof, see \cite{GR79}.
\begin{lemma}\label{lem:stokes}
	There exists a positive constant $C$ independent of $h$ such that the followings hold for all $s\in\{1,2,\dots,k+1\}$:
	\begin{align*}
		&\|\bu-\bS_h\bu\| + h \left(\vertiii{\bu-\bS_h\bu} + \|p-\bS_hp\|\right) \le C h^{s}\left(\|\bu\|_{\bH^{s}} + \|p\|_{H^{s-1}}\right).
	\end{align*}
\end{lemma}

\section{A Priori Error Analysis} \label{Sec:ee}
%\se

This section analyses the error due to space and temporal discretization.
Define any continuous in time function $\phi(t)$ at a time level $t_{m}$ as $\phi^m = \phi(t_{m})$. Set
\begin{align}\label{errorsplit}
	e_\phi^m = \phi^m - \phi_h^m, \quad
	e_{c_{i}}^m = \tc_{i}^m - \tc_{ih}^m = c_{i}^m - c_{ih}^m, \quad \e_u^m = \bu^m - \bu_h^m, \quad \heu^m = \bu^m - \hbu_h^m, \quad e_p^m = p^m - p_h^m.	
\end{align}
Then, from \eqref{dg1}-\eqref{dg4} and \eqref{fds11}-\eqref{fds5}, there hold
\begin{align}
	&\quad \mu \cA_1(\ephi^{m}, \psi_h) = \left(\eco^{m-1}-\ect^{m-1}, \psi_h\right) + \left((\tc_1^{m}-\tc_1^{m-1})-(\tc_2^{m}-\tc_2^{m-1}), \psi_h\right) \quad \forall ~ \psi \in X_h^{0} \label{eephi} \\
	& \left(\frac{\eci^{m}-\eci^{m-1}}{ \Dt}, \chi_{ih}\right) + \kappa_i\cA_1(\eci^{m},\chi_{ih}) = E_{c_i}(\chi_{ih}) + \Lambda_{c_i}(\chi_{ih}) + F_{c_i}(\chi_{ih}) \quad \forall~ \chi_{ih}\in X_h^{0} \label{eeci} \\
	& 
	\begin{aligned}
		\left(\frac{\heu^{m}-\eu^{m-1}}{\Dt}, \bv_h \right) + \nu \cA_2(\heu^{m},\bv_h) - \cD(\bv_h, \ep^{m-1}) =  & ~  E_{u}(\bv_h) + \Lambda_u(\bv_h) + F_u(\bv_h)  \\
		& + \cD(\bv_h, p^{m}-p^{m-1}) \quad \forall~ \bv_h \in \bV_h
	\end{aligned} \label{eeu}\\
	& \cA_1(\ep^{m}-\ep^{m-1},q_h) = \cA_1(p^{m}-p^{m-1}, q_h) - \frac{1}{\Dt} \cD(\heu^{m},q_h), \quad \forall~ q_h \in M_h \label{eep} \\
	& \left(\frac{\eu^{m}-\heu^{m}}{\Dt}, \bw_h\right) - \cD(\bw_h, \ep^{m}-\ep^{m-1}) = -\cD(\bw_h, p^{m}-p^{m-1}) \quad \forall ~ \bw_h \in \bV_h, \label{eeuf}
\end{align}
where
\begin{align}
	& E_{c_i}(\chi_{ih}) = \left(\frac{\tc_{i}^{m}-\tc_{i}^{m-1}}{\Dt} -  \pt \tc_{i}^{m}, \chi_{ih}\right), \quad  %\label{Ec}\\
	E_u(\bv_h) = \left(\frac{\bu^{m}-\bu^{m-1}}{\Dt}-  \pt \bu^{m}, \bv_h\right) \label{Eu} \\
	& \Lambda_{c_i}(\chi_{ih}) = \cN_1(\bu_h^{m-1}, \tc_{ih}^{m}, \chi_{ih}) - \cN_1(\bu^{m},\tc_{i}^{m},\chi_{ih}) \label{Lambdac} \\
	& \Lambda_{u}(\bv_h) = \cN_2(\bu_h^{m-1}, \hbu_h^{m}, \bv_h) - \cN_2(\bu^{m},\bu^{m},\bv_h) \label{Lambdau}\\
	& F_{c_i}(\chi_{ih}) = \beta_i \cG(\tc_{ih}^{m-1}+m_0, \phi_h^{m}, \chi_{ih}) - \beta_i \cG(\tc_{i}^{m}+m_0, \phi^{m}, \chi_{ih}) \label{Fc}\\
	& F_u(\bv_h) = \cT(\tc_{1h}^{m-1}-\tc_{2h}^{m-1}, \phi_h^{m}, \bv_h) - \cT(\tc_1^{m} - \tc_2^{m}, \phi^{m}, \bv_h).\label{Fu}
\end{align}

\noindent
Using the projection operator, we split the errors as follows
\begin{equation}\label{usplit}
	\left.
	\begin{aligned}
		\ephi^{m} = \phi^m - \phi_h^{m} = \underbrace{\phi^{m} - R_h\phi^{m}}_{\xi_\phi^{m}} + \underbrace{R_h\phi^{m} - \phi_h^{m}}_{\eta_{\phi}^{m}},
		\quad & \quad
		\eci^{m} = \tc_i^m - \tc_{ih}^{m} = \underbrace{\tc_i^{m} - Q_h\tc_i^{m}}_{\xi_{c_i}^m} + \underbrace{Q_h\tc_{i}^{m} - \tc_{ih}^{m}}_{\eta_{c_i}^m} \\
		\eu^m = \bu^m - \bu_h^m = \underbrace{\bu^m - \bS_h\bu^m}_{\bxi_u^m} + \underbrace{\bS_h\bu^m - \bu_h^m}_{\bta_u^m},
		\quad & \quad
		\heu^m = \bu^m - \hbu_h^m =  \underbrace{\bu^m - \bS_h\bu^m}_{\bxi_u^m} + \underbrace{\bS_h\bu^m - \hbu_h^m}_{\hbta_u^m}, \\
		e_p^m = p^m - p_h^m = & \underbrace{p^m - \bS_hp^m}_{\xi_p^m} + \underbrace{\bS_hp^m - p_h^m}_{\eta_p^m},		
	\end{aligned}\right\}
\end{equation}
Now, we rewrite \eqref{eephi}-\eqref{eeuf} using \eqref{eqritz}, \eqref{modi:eq} and \eqref{stokesproj} as: For all $(\psi_h, \chi_{ih}, \bv_h, q_h, \bw_h)\in X_h^{0}\times X_h^{0}\times \bV_h \times M_h \times \bV_h$
\begin{align}
	&\quad \mu \cA_1(\eta_\phi^{m}, \psi_h) = \left(\xi_{c_1}^{m-1}-\xi_{c_2}^{m-1}, \psi_h\right) + \left(\eta_{c_1}^{m-1}-\eta_{c_2}^{m-1}, \psi_h\right) + \left((\tc_1^{m}-\tc_1^{m-1})-(\tc_2^{m}-\tc_2^{m-1}), \psi_h\right), \label{erretaphi}
	\\
	& \left(\frac{\eta_{c_i}^{m}-\eta_{c_i}^{m-1}}{ \Dt}, \chi_{ih}\right) + \kappa_i\cA_1(\eta_{c_i}^{m},\chi_{ih}) = - \left(\frac{\xi_{c_i}^{m}-\xi_{c_i}^{m-1}}{ \Dt}, \chi_{ih}\right) + E_{c_i}(\chi_{ih}) + \Lambda_{c_i}(\chi_{ih}) + H_{c_i}(\chi_{ih}), \label{erretaci} 
	\\
	& 
	\begin{aligned}\label{erretau}
		\left(\frac{\hbta_u^{m}-\bta_u^{m-1}}{\Dt}, \bv_h \right) + \nu \cA_2(\hbta_u^{m},\bv_h) - \cD(\bv_h, \eta_p^{m-1}) = -\left(\frac{\bxi_u^{m}-\bxi_u^{m-1}}{\Dt}, \bv_h \right) +   ~  E_{u}(\bv_h) + \Lambda_u(\bv_h)   \\
		+ F_u(\bv_h) + \cD(\bv_h, p^{m}-p^{m-1}) - \cD(\bv_h, \xi_p^{m}-\xi_p^{m-1}),
	\end{aligned} \\
	& \cA_1(\eta_p^{m}-\eta_p^{m-1},q_h) = - \cA_1(\xi_p^{m}-\xi_p^{m-1},q_h) + \cA_1(p^{m}-p^{m-1}, q_h) - \frac{1}{\Dt} \cD(\hbta_u^{m},q_h), \label{erretap} 
	\\
	& \left(\frac{\bta_u^{m}-\hbta_u^{m}}{\Dt}, \bw_h\right) - \cD(\bw_h, \eta_p^{m}-\eta_p^{m-1}) =  \cD(\bw_h, \xi_p^{m}-\xi_p^{m-1}) - \cD(\bw_h, p^{m}-p^{m-1}), \label{erretauf}
\end{align}
where 
\begin{align} \label{Hc}
	H_{c_i}(\chi_{ih}) & = F_{c_i}(\chi_{ih}) + \beta_i \cG(\tc_{i}^{m}+m_0, \xi_\phi^{m}, \chi_{ih}) \nonumber\\
	& = \beta_i \cG(\tc_{ih}^{m-1}+m_0, \phi_h^{m}, \chi_{ih}) - \beta_i \cG(\tc_{i}^{m}+m_0, \phi^{m}, \chi_{ih}) + \beta_i \cG(\tc_{i}^{m}+m_0, \xi_\phi^{m}, \chi_{ih}).
\end{align}

Before the error estimates, we prove several lemmas that help our subsequent analysis.
\begin{lemma} \label{lem:trilinear}
	For any $\eta_{c_i}^m\in X_h^0$ and $\hbta_u^m\in \bV_h$ defined in \eqref{usplit}, there holds
	\begin{align*}%\label{lelambdaest}
		\Lambda_{c_i}(\eta_{c_i}^{m}) \lesssim & ~ \vertiii{\eta_{c_i}^{m}} \bigg(h^{-1}\|\bxi_{u}^{m-1}\|\vertiii{\xi_{c_i}^{m}} + h^{-1}\|\xi_{c_i}^{m}\|\vertiii{\bxi_{u}^{m-1}} + \vertiii{\bxi_u^{m-1}}\vertiii{\xi_{c_i}^{m}}   \nonumber\\
		& + \left(\|\bxi_u^{m-1}\| + h\vertiii{\bxi_u^{m-1}} \right) \|\tc_{i}^{m}\|_{H^2} +\left( \|\bu^{m}-\bu^{m-1}\|+ \|\nabla(\bu^{m}-\bu^{m-1})\|\right)\|\tc_{i}^{m}\|_{H^2}  \nonumber\\
		& + \left(\|\xi_{c_i}^{m}\|+ h\vertiii{\xi_{c_i}^{m}}\right) \|\bu^{m-1}\|_{H^2} + \left(h^{-2}\|\xi_{c_i}^{m}\| + h^{-1}\vertiii{\xi_{c_{i}}^{m}} + \|\tc_{i}^{m}\|_{H^2}^{2}\right)\|\bta_u^{m-1}\|\bigg),
	\end{align*}
	and 
	\begin{align*}%\label{lelambdaest}
		\Lambda_{u}(\hbta_u^{m}) \lesssim & ~ \vertiii{\hbta_u^{m}} \bigg(h^{-1}\|\bxi_{u}^{m-1}\|\vertiii{\bxi_u^m} + h^{-1}\|\bxi_u^m\|\vertiii{\bxi_{u}^{m-1}} + \vertiii{\bxi_u^{m-1}}\vertiii{\bxi_u^m}   \nonumber\\
		& + \left(\|\bxi_u^{m-1}\| + h\vertiii{\bxi_u^{m-1}} \right) \|\bu^{m}\|_{H^2} +\left( \|\bu^{m}-\bu^{m-1}\|+ \|\nabla(\bu^{m}-\bu^{m-1})\|\right)\|\bu^{m}\|_{H^2}  \nonumber\\
		& + \left(\|\bxi_u^m\|+ h\vertiii{\bxi_u^m}\right) \|\bu^{m-1}\|_{H^2} + \left(h^{-2}\|\bxi_u^m\| + h^{-1}\vertiii{\xi_{c_{i}}^{m}} + \|\bu^{m}\|_{H^2}^{2}\right)\|\bta_u^{m-1}\|\bigg).
	\end{align*}
\end{lemma}
\begin{proof}
	Choose $\chi_{ih} = \eta_{c_i}^m$ in \eqref{Lambdac}. Then,  we rewrite it using \eqref{usplit} as
	\begin{align*}
		\Lambda_{c_i}(\eta_{c_i}^{m}) & = \cN_1(\eu^{m-1},\eci^{m},\eta_{c_i}^{m}) - \cN_1(\eu^{m-1},\tc_{i}^{m},\eta_{c_i}^{m}) - \cN_1(\bu^{m-1},\eci^{m},\eta_{c_i}^{m}) - \cN_1(\bu^{m}-\bu^{m-1},\tc_{i}^{m},\eta_{c_i}^{m}) \nonumber\\
		& = \cN_1(\bxi_u^{m-1},\xi_{c_i}^{m},\eta_{c_i}^{m}) + \cN_1(\bta_u^{m-1},\xi_{c_{i}}^{m},\eta_{c_i}^{m}) + \cN_1(\bxi_u^{m-1},\eta_{c_i}^{m},\eta_{c_i}^{m}) +
		\cN_1(\bta_u^{m-1},\eta_{c_i}^{m},\eta_{c_i}^{m}) \nonumber\\
		& \quad - \cN_1(\bxi_u^{m-1},\tc_{i}^{m},\eta_{c_i}^{m}) - \cN_1(\bta_u^{m-1},\tc_{i}^{m},\eta_{c_i}^{m}) - \cN_1(\bu^{m-1},\xi_{c_i}^{m},\eta_{c_i}^{m}) - \cN_1(\bu^{m-1},\eta_{c_i}^{m},\eta_{c_i}^{m}) \nonumber\\
		&\quad - \cN_1(\bu^{m}-\bu^{m-1},\tc_{i}^{m},\eta_{c_i}^{m}) = I_{11}+\cdots+I_{19}. 
	\end{align*} %\label{lelambda}
	Due to the skew-symmetry property of trilinear terms (Lemma \ref{lem:skew}), $I_{13}=I_{14}=I_{18}=0$. A use of \eqref{NN1} of Lemma \ref{lem:estnon} with the inverse hypothesis \eqref{inv.hypo} and the discrete Poincar\'e inequality \eqref{DPI} yields
	\begin{align*}
		 |I_{11}| & \lesssim \sum_{E\in\Eh} \left(\|\bxi_u^{m-1}\|_{L^2(E)}\|\nabla\xi_{c_{i}}^{m}\|_{L^2(E)}\|\eta_{c_i}^{m}\|_{L^\infty(E)} + \|\nabla\bxi_u^{m-1}\|_{L^2(E)}\|\xi_{c_{i}}^{m}\|_{L^2(E)}\|\eta_{c_i}^{m}\|_{L^\infty(E)}\right) \nonumber\\
		& \qquad + J(\xi_{c_i}^{m})\sum_{E\in\Eh}\left( h_E^{\frac{1}{2}} \|\bxi_u^{m-1}\|_{\bL^2(E)}^{\frac{1}{2}}\|\nabla\bxi_u^{m-1}\|_{L^2(E)}^{\frac{1}{2}} \|\eta_{c_i}^{m}\|_{L^\infty(E)} \right) \nonumber\\
		&\qquad + J(\bxi_{u}^{m-1})\sum_{E\in\Eh}\left(h_E^{\frac{1}{2}}\|\xi_{c_i}^{m}\|_{L^2(E)}^{\frac{1}{2}}\|\nabla\xi_{c_i}^{m}\|_{L^2(E)}^{\frac{1}{2}} \|\eta_{c_i}^{m}\|_{L^\infty(E)} \right) \nonumber\\
		&	\lesssim h^{-1} \left(\|\bxi_{u}^{m-1}\|\vertiii{\xi_{c_i}^{m}} + \|\xi_{c_i}^{m}\|\vertiii{\bxi_{u}^{m-1}} + h\vertiii{\bxi_u^{m-1}}\vertiii{\xi_{c_i}^{m}}\right)\vertiii{\eta_{c_i}^{m}}.
	\end{align*}
	To estimate $I_{12}$, we apply the estimate \eqref{NN2} of the Lemma \ref{lem:estnon} with the inverse hypothesis \eqref{inv.hypo} and the discrete Poincar\'e inequality \eqref{DPI} to derive
	\begin{align*}
		|I_{12}| & \lesssim \sum_{E\in\Eh} \left(\|\bta_u^{m-1}\|_{L^2(E)}\|\nabla\xi_{c_{i}}^{m}\|_{L^2(E)}\|\eta_{c_i}^{m}\|_{L^\infty(E)} + \|\bta_u^{m-1}\|_{L^2(E)}\|\xi_{c_{i}}^{m}\|_{L^2(E)}\|\nabla\eta_{c_i}^{m}\|_{L^\infty(E)}\right) \nonumber\\
		&  \qquad+ J(\xi_{c_i}^{m})\sum_{E\in\Eh}\left( h_E^{\frac{1}{2}} \|\bta_u^{m-1}\|_{L^2(E)}^{\frac{1}{2}}\|\nabla\bta_u^{m-1}\|_{L^2(E)}^{\frac{1}{2}}\|\eta_{c_i}^{m}\|_{L^\infty(E)}  \right) \nonumber\\
		&\qquad + J(\eta_{c_i}^{m})\sum_{E\in\Eh}\left( h_E^{\frac{1}{2}} \|\bta_u^{m-1}\|_{L^\infty(E)}\|\xi_{c_i}^{m}\|_{L^2(E)}^{\frac{1}{2}}\|\nabla\xi_{c_i}^{m}\|_{L^2(E)}^{\frac{1}{2}} \right) \nonumber\\
		&\lesssim  h^{-2}\left(\|\xi_{c_{i}}^{m}\|+ h\vertiii{\xi_{c_{i}}^{m}}\right) \|\bta_u^{m-1}\| \vertiii{\eta_{c_i}^{m}}.
	\end{align*}
	One can also find the following bound for $I_{15}$ using \eqref{NN2} of Lemma \ref{lem:estnon} with the $L^p$ bound in broken Sobolev space \eqref{Lptoenergy}, the Gagliardo-Nirenberg inequality \eqref{GNI} and the Agmon's inequality \eqref{AI} as
	\begin{align*}
		|I_{15}| & \lesssim \|\bxi_{u}^{m-1}\|\|\nabla\tc_i^{m}\|_{L^4(\Omega)}\|\eta_{c_i}^{m}\|_{L^4(\Omega)} + \|\bxi_{u}^{m-1}\| \|\nabla\eta_{c_i}^{m}\| \|\tc_i^{m}\|_{L^\infty(\Omega)}  + J(\eta_{c_i}^{m})\left( h^{\frac{1}{2}} \|\bxi_{u}^{m-1}\|^{\frac{1}{2}} \|\nabla\bxi_{u}^{m-1}\|^{\frac{1}{2}}\|\tc_i^{m}\|_{L^\infty(\Omega)} \right) \\
		& \lesssim \left(\|\bxi_{u}^{m-1}\|+ h\vertiii{\bxi_{u}^{m-1}}\right) \|\tc_i^{m}\|_{H^2} \vertiii{\eta_{c_i}^{m}},
	\end{align*}
	In a similar way, additionally use of \eqref{inv.hypo} and \eqref{inv.DG}, $I_{16}$ can be bounded as
	\begin{align*}
		|I_{16}| \lesssim \left(\|\bta_{u}^{m-1}\|+ h\vertiii{\bta_{u}^{m-1}}\right) \|\tc_i^{m}\|_{H^2} \vertiii{\eta_{c_i}^{m}} \lesssim \|\bta_{u}^{m-1}\| \|\tc_i^{m}\|_{H^2} \vertiii{\eta_{c_i}^{m}}.
	\end{align*}
	For the term $I_{17}$, we first rewrite it using Lemma \ref{lem:skew} and then use Lemma \ref{lem:estnon} with the H\"{o}lder's inequality, the Young inequality,  the inverse hypothesis \eqref{inv.hypo}, the discrete Poincar\'e inequality \eqref{DPI}, the Gagliardo-Nirenberg inequality \eqref{GNI} and the Agmon's inequality \eqref{AI} to deduce
	\begin{align*}
		|I_{17}|  = |-\cN_1(\bu^{m-1},\eta_{c_i}^{m},\xi_{c_{i}}^{m})|
		\lesssim  \left(\|\xi_{c_i}^{m}\|+ h\vertiii{\xi_{c_i}^{m}}\right) \|\bu^{m-1}\|_{H^2}\vertiii{\eta_{c_i}^{m}}.
	\end{align*}
	One can bound the last term $I_{19}$ using Lemma \ref{lem:estnon} as
	\begin{align*}
		|I_{19}| & \lesssim  \left(\|\bu^{m}-\bu^{m-1}\|+ \|\nabla(\bu^{m}-\bu^{m-1})\|\right) \|\tc_{i}^{m}\|_{H^2}\vertiii{\eta_{c_i}^{m}}.
	\end{align*}
	We now combine all the bounds from $I_{11}$ to $I_{19}$ to complete the proof of the estimate $\Lambda_{c_i}(\eta_{c_i}^{m})$. Exactly in the similar way by replacing $\eta_{c_i}^m$ and $\tc_i^m$ by $\hbta_u^m$ and $\bu^m$, we obtain the estimate $\Lambda_u{(\hbta_{u}^{m})}$ of this lemma. This completes the rest of the proof.
\end{proof}

\begin{lemma} \label{lem:Hci}
	For any $\eta_{c_i}^m\in X_h^0$ defined in \eqref{usplit}, the following holds true:
	\begin{align*} 
		H_{c_i}(\eta_{c_i}^{m}) &\lesssim \bigg( \left(\|\phi^{m}\|_{H^3}+ h^{-1}\vertiii{\xi_\phi^{m}}\right)\left(\|\xi_{c_i}^{m-1}\| + h \vertiii{\xi_{c_i}^{m-1}} + \|\tc_{i}^{m}-\tc_{i}^{m-1}\| + h\|\nabla(\tc_{i}^{m}-\tc_{i}^{m-1})\|\right) \nonumber\\
		& \quad\qquad + \left(\|\phi^{m}\|_{H^3}+ h^{-1}\vertiii{\xi_\phi^{m}} \right) \|\eta_{c_i}^{m-1}\| 
		+ \left(h^{-1} \|\xi_{c_i}^{m-1}\| +  \vertiii{\xi_{c_i}^{m-1}}  + \|\tc_{i}^{m-1}\|_{H^2}\right) \vertiii{\eta_{\phi}^{m}} \nonumber\\
		&\quad\qquad +   h^{-1}\|\eta_{c_i}^{m-1}\|\vertiii{\eta_{\phi}^{m}}\bigg) \vertiii{\eta_{c_i}^{m}}.
	\end{align*}
\end{lemma}

\begin{proof}
	Choose $\chi_{ih} = \eta_{c_i}^m$ in \eqref{Hc} and use \eqref{usplit} to rewrite it as
	\begin{align*} %\label{l2g1}
		H_{c_i}(\eta_{c_i}^{m}) &= \beta_i\left( \cG(\eci^{m-1},\eta_{\phi}^{m},\eta_{c_i}^{m}) - \cG(\eci^{m-1},R_h\phi^{m},\eta_{c_i}^{m})  - \cG(\tc_{i}^{m-1}+m_0, \eta_{\phi}^{m},\eta_{c_i}^{m}) - \cG(\tc_{i}^{m}-\tc_{i}^{m-1},R_h\phi^{m},\eta_{c_i}^{m})\right) \nonumber\\
		&= I_{21}+I_{22}+I_{23}+I_{24}.
	\end{align*}
	An application of \eqref{GG} of Lemma \ref{lem:estnon} with a proper use of the H\"{o}lder's inequality and the inverse hypothesis \eqref{inv.hypo} helps to bound the following term as
	\begin{align*}%\label{l2g2}
		|I_{21}|  
		& \lesssim \sum_{E\in \Eh}\|e_{c_i}^{m-1}\|_{L^2(E)}\|\nabla\eta_{\phi}^{m}\|_{L^2(E)}\|\nabla\eta_{c_i}^{m}\|_{L^\infty(E)} \nonumber\\
		&~~ + J(\eta_{\phi}^{m})\sum_{E\in \Eh} \left( h_E^{\frac{1}{2}} \|e_{c_i}^{m-1}\|_{L^2(E)}^{\frac{1}{2}}\|\nabla e_{c_i}^{m-1}\|_{L^2(E)}^{\frac{1}{2}}\|\nabla\eta_{c_i}^{m}\|_{L^\infty(E)} \right) \nonumber\\
		&~~ + J(\eta_{c_i}^{m})\sum_{E\in \Eh} \left(h_E^{\frac{1}{2}} \|e_{c_i}^{m-1}\|_{L^2(E)}^{\frac{1}{2}}\|\nabla e_{c_i}^{m-1}\|_{L^2(E)}^{\frac{1}{2}}\|\nabla\eta_{\phi}^{m}\|_{L^\infty(E)} \right) \nonumber\\
		& \lesssim \left(h^{-1}\|\eci^{m-1}\| + \vertiii{\eci^{m-1}}\right) \vertiii{\eta_{\phi}^{m}} \vertiii{\eta_{c_i}^{m}}.	
	\end{align*}
	Similarly, one can bound $I_{22}$ as
		\begin{align*}%\label{l2g3}
			|I_{22}|  
			& \lesssim  \sum_{E\in \Eh}\|e_{c_i}^{m-1}\|_{L^2(E)}\|\nabla R_h{\phi}^{m}\|_{L^\infty(E)} \|\nabla\eta_{c_i}^{m}\|_{L^2(E)} \nonumber\\
			 &\quad + J(R_h{\phi}^{m})\sum_{E\in \Eh} \left(h_E^{\frac{1}{2}}\|e_{c_i}^{m-1}\|_{L^2(E)}^{\frac{1}{2}}\|\nabla e_{c_i}^{m-1}\|_{L^2(E)}^{\frac{1}{2}}\|\nabla\eta_{c_i}^{m}\|_{L^\infty(E)} \right) \nonumber\\
			&\quad + J(\eta_{c_i}^{m})\sum_{E\in \Eh} \left(h^{\frac{1}{2}}\|e_{c_i}^{m-1}\|_{L^2(E)}^{\frac{1}{2}} \|\nabla e_{c_i}^{m-1}\|_{L^2(E)}^{\frac{1}{2}}\|\nabla R_h{\phi}^{m}\|_{L^\infty(E)}\right).
	\end{align*}
	Since $\jump{\phi^{m}} = 0$, we replace $J(R_h\phi^{m})$ by $J(R_h\phi^{m}-\phi^{m})$ and using the fact $\|\nabla  R_h\phi^m\|_{L^\infty(E)} \lesssim \| \phi^m\|_{H^3}$, we can bound the $I_{22}$ term as follows
	\begin{align*}
		|I_{22}| \lesssim \left(\|\phi^{m}\|_{H^3}+ h^{-1}\vertiii{\xi_\phi^{m}} \right)\left(\|\eci^{m-1}\| + h \vertiii{\eci^{m-1}} \right) \vertiii{\eta_{c_i}^{m}}.
	\end{align*}
	To bound $I_{23}$, we apply \eqref{GG} of Lemma \ref{lem:estnon} with appropriate choice of $p$ and $q$. Further, we use the Gagliardo-Nirenberg inequality \eqref{GNI} and the Agmon's inequality \eqref{AI} to find
	\begin{align*}
		|I_{23}| \lesssim \|\tc_{i}^{m}\|_{H^2} \vertiii{\eta_{\phi}^{m}} \vertiii{\eta_{c_i}^{m}}.
	\end{align*}
	Arguing a same way of $I_{22}$ with the appropriate choice of $p$ and $q$, we bound $I_{24}$ as follows
	\begin{align*}%\label{l2g6}
		|I_{24}| \lesssim \left(\|\phi^{m}\|_{H^3}+ h^{-1}\vertiii{\xi_\phi^{m}} \right)\left(\|\tc_{i}^{m}-\tc_{i}^{m-1}\| + h\|\nabla(\tc_{i}^{m}-\tc_{i}^{m-1})\| \right) \vertiii{\eta_{c_i}^{m}}.
	\end{align*}	
	Collecting all the bounds of $I_{21}$ to $I_{24}$ completes the rest of the proof.
\end{proof}

\begin{lemma} \label{lem:Tu}
	For any $\hbta_u^m\in \bV_h$ defined in \eqref{usplit}, there holds
	\begin{align*} %\label{eqfufinal}
		|F_u(\hbta_{u}^{m})| & \lesssim \Big( \left(h^{-1}\|\xi_{c_1}^{m-1}-\xi_{c_2}^{m-1}\|+\vertiii{\xi_{c_1}^{m-1}-\xi_{c_2}^{m-1}}\right)\left(\vertiii{\xi_{\phi}^{m}} + \vertiii{\eta_{\phi}^{m}}\right)  \nonumber\\
		& + \|\eta_{c_1}^{m-1}-\eta_{c_2}^{m-1}\|\left(\vertiii{\xi_{\phi}^{m}} + \vertiii{\eta_{\phi}^{m}}\right) +  \left(\|\xi_{c_1}^{m-1}-\xi_{c_2}^{m-1}\|+ \|\eta_{c_1}^{m-1}-\eta_{c_2}^{m-1}\|\right)\|\phi^{m}\|_{H^2} \nonumber\\
		& + \left(\|\xi_{\phi}^{m}\| + h\vertiii{\xi_{\phi}^{m}} + \|\eta_{\phi}^{m}\| \right)\|\tc_{1}^{m-1}-\tc_{2}^{m-1}\|_{H^2} + \left(\|\tc_1^{m}-\tc_1^{m-1}\|^2+\|\tc_2^{m}-\tc_2^{m-1}\|\right)\|\phi^{m}\|_{H^2} \Big)\vertiii{\hbta_{u}^{m}}.
	\end{align*}
\end{lemma}

\begin{proof}
	Choose $\bv_h=\hbta_{u}^m$ in \eqref{Fu} and rewrite it as
	\begin{align*}
		F_u(\hbta_{u}^{m}) & = \cT(\eco^{m-1}-\ect^{m-1}, \ephi^{m},\hbta_{u}^{m}) - \cT(\eco^{m-1}-\ect^{m-1}, \phi^{m}, \hbta_{u}^{m}) - \cT(\tc_1^{m-1}-\tc_2^{m-1}, \ephi^{m}, \hbta_{u}^{m}) \nonumber\\
		& \qquad- \cT((\tc_1^{m}-\tc_1^{m-1})-(\tc_2^{m}-\tc_2^{m-1}), \phi^{m}, \hbta_u^{m}) = I_{31}+I_{32}+I_{33}+I_{34}.
	\end{align*}
	An application of \eqref{TT} of Lemma \ref{lem:estnon} with the H\"{o}lder's inequality, the Young's inequality, the inverse hypothesis \eqref{inv.hypo}, and the discrete Poincar\'e inequality \eqref{DPI} helps to bound the $I_{31}$ term as
	\begin{align*}
		|I_{31}| & \lesssim \sum_{E\in \Eh}\|\eco^{m-1}-\ect^{m-1}\|_{L^2(E)} \|e_{\phi}^{m}\|_{L^2(E)} \|\hbta_u^{m}\|_{L^\infty(E)}  \\
        &\qquad+ J(e_{\phi}^{m})\sum_{E\in \Eh} h^{\frac{1}{2}}\|\eco^{m-1}-\ect^{m-1}\|_{L^2(E)}^{\frac{1}{2}}\|\nabla(\eco^{m-1}-\ect^{m-1})\|_{L^2(E)}^{\frac{1}{2}}  \|\hbta_u^{m}\|_{L^\infty(E)} \nonumber\\
		& \lesssim  \left(h^{-1}\|\eco^{m-1}-\ect^{m-1}\| + \vertiii{\eco^{m-1}-\ect^{m-1}}\right)\vertiii{\ephi^{m}} \vertiii{\hbta_{u}^{m}}.
	\end{align*}
	The terms $I_{32}$ and $I_{34}$ can be estimated by applying Lemma \ref{lem:estnon} with the fact $J(\phi^{m}) = 0$ and the inverse hypothesis \eqref{inv.hypo}, the discrete Poincar\'e inequality \eqref{DPI} and the Agmon's inequality \eqref{AI} as follows
	\begin{align*}
		|I_{32}|+|I_{34}| & \lesssim \left( \|\eco^{m-1}-\ect^{m-1}\| +    \|\tc_1^{m}-\tc_1^{m-1}\|+\|\tc_2^{m}-\tc_2^{m-1}\|\right)\|\phi^{m}\|_{H^2} \vertiii{\hbta_{u}^{m}}.
	\end{align*}
	To estimate $I_{33}$, we first use integration by parts and then apply the H\"{o}lder's inequality, the Young's inequality, Lemma  \ref{CTI}, the discrete trace inequality \eqref{DTI}, the inverse inequality \eqref{inv.hypo}, the discrete Poincar\'e inequality \eqref{DPI}, the Gagliardo-Nirenberg inequality \eqref{GNI} and the Agmon's inequality \eqref{AI} to find the following
	\begin{align*}
		|I_{33}| & = \bigg|-\sum_{E\in \Eh}\int_E\left(\ephi^{m}\nabla(\tc_{1}^{m-1}-\tc_{2}^{m-1})\cdot \hbta_u^{m} + \ephi^{m}(\tc_{1}^{m-1}-\tc_{2}^{m-1})\nabla\cdot\hbta_{u}^{m}\right) \dx\nonumber\\
		&\quad + \sum_{e\in\Gamma_h^0} \int_e \left(\avr{\ephi^{m}}\jump{\tc_{1}^{m-1}-\tc_{2}^{m-1}}\avr{\hbta_{u}^{m}}\cdot \bn_e + \avr{\ephi^{m}(\tc_{1}^{m-1}-\tc_{2}^{m-1})}\jump{\hbta_{u}^{m}}\cdot\bn_e\right) \ds\bigg| \nonumber\\
		& \lesssim \left(\|\ephi^{m}\| + h\vertiii{\ephi^{m}}\right) \|\tc_{1}^{m-1}-\tc_{2}^{m-1}\|_{H^2} \vertiii{\hbta_{u}^{m}}.
	\end{align*}
	Combining all the above three estimates and using \eqref{usplit} and Lemma \ref{inv.DG}, we complete the rest of the proof.
\end{proof}

\subsection{Optimal \texorpdfstring{$L^2$}{L2} Error Bounds}

This section deals with optimal error bounds for a fully discrete solution in $L^2$-norm. The main result of this section is given below:
\begin{theorem}\label{thm:l2}
	Suppose that $(\phi(t_m), c_1(t_m), c_2(t_m), \bu(t_m), p(t_m)) $ and $(\phi_h^m, c_{1h}^{m}, c_{2h}^{m}, \bu_h^{m}, p_{h}^{m})$ are the solution of the PNP-NS system \eqref{eqphi}-\eqref{eqdiv} and the fully discrete system \eqref{fds11}-\eqref{fds5}, respectively. Then, under the assumption
	\begin{align*}
		& \bu\in L^\infty(0,T;\bH^2)\cap  L^2(0,T;\bH^{k+1}),~\pt\bu\in L^2(0,T;\bH^{k+1}), ~\partial_{tt}\bu\in L^2(0,T;\bL^2), \\
		& \tc_i\in L^\infty(0,T;H^2)\cap  L^2(0,T;H^{k+1}),~ \pt\tc_i\in L^2(0,T;H^{k+1}), ~\partial_{tt}\tc_{i}\in L^2(0,T;L^2),\\
		& p\in L^2(0,T;H^{k}),~p_t\in L^2(0,T;H^{k}),~\phi\in L^\infty(0,T;H^3)\cap L^2(0,T;H^{k+1}),
	\end{align*}
	there exists a positive constant $C$ such that the following relation holds for any $N\ge 1$
	\begin{align*}
		\|c_1(t_N)-c_{1h}^{N}\|  + \|c_2(t_N)-c_{2h}^{N}\|  + \|\bu(t_N)-\bu_h^{N}\|  \le  C \left(h^{k+1}  + \Dt \right).
	\end{align*}
\end{theorem}

Before proving this theorem, we first obtain a couple of lemmas that help to prove the optimal results.
\begin{lemma}\label{lem:estphif}
	Suppose that the assumptions of Theorem \ref{thm:l2} hold true. Then, the following bound holds:
	\begin{align*}
		\Dt \sum_{m=1}^{N} \left(\vertiii{\eta_\phi^{m}}^2 + \|\Delta_h\eta_\phi^m\|^2\right)  
		\lesssim h^{2k+2} +  (\Dt)^2    +  \Dt \sum_{m=1}^{N} \left(\|\eta_{c_1}^{m-1}\|^2  + \|\eta_{c_2}^{m-1}\|^2\right).
	\end{align*}
	Here, the discrete operator $\Delta_h:X_h^0\to X_h^0$ is defined via $(\Delta_h\psi_h,\xi_h) = \cA_1(\psi_h,\xi_h)$.
\end{lemma}
\begin{proof}
	Choose $\psi_{h} = \eta_\phi^{m}$ and $\psi_h=-\Delta_h\eta_\phi^m$ in \eqref{erretaphi} and add the resulting equation. Then, we use the Cauchy-Schwartz inequality with the discrete Poincar\'e inequality \eqref{DPI} to obtain
	\begin{align} \label{l2est1}
		\vertiii{\eta_\phi^{m}}^2 + \|\Delta_h\eta_\phi^m\|^2 \lesssim \|\tc_1^{m}-\tc_1^{m-1}\|^2 + \|\tc_2^{m}-\tc_2^{m-1}\|^2 + \|\xi_{c_1}^{m-1}-\xi_{c_2}^{m-1}\|^2 + \|\eta_{c_1}^{m-1}-\eta_{c_2}^{m-1}\|^2.
	\end{align}
	We use the H\"{o}lder inequality on the first two terms of the right-hand side of \eqref{l2est1} to find
	\begin{align*}
		\|\tc_1^{m}-\tc_1^{m-1}\|^2 + \|\tc_2^{m}-\tc_2^{m-1}\|^2 \le \Dt \int_{t_{m-1}}^{t_{m}} \left(\|\pt\tc_1(s)\|^2 + \|\pt\tc_2(s)\|^2\right) \ds.
	\end{align*}
	Plug the above bound in \eqref{l2est1} and use approximation property of Lemma \ref{lem:modi} to derive
	\begin{align} 
		\vertiii{\eta_\phi^{m}}^2 + \|\Delta_h\eta_\phi^m\|^2 \lesssim ~& h^{2k+2}\left(\|\tc_{1}^{m-1}\|_{H^{k+1}}^{2} + \|\tc_{2}^{m-1}\|_{H^{k+1}}^{2}\right)  + \frac{\Dt}{2} \int_{t_{m-1}}^{t_{m}} \left(\|\pt\tc_1(s)\|^2 + \|\pt\tc_2(s)\|^2\right) \ds \nonumber\\
		& + \left(\|\eta_{c_1}^{m-1}\|^2 + \|\eta_{c_2}^{m-1}\|^2\right). \label{estphi1}
	\end{align}
	Now, multiply both sides by $\Dt$ and take the summation from $m=1$ to $N$. Then, a use of Lemma \ref{lem:ritz} with the assumption given in Theorem \ref{thm:l2} completes the rest of the proof.
\end{proof}
\begin{lemma}\label{lem:l2cfinal}
	Under the assumption of Theorem \ref{thm:l2}, the following holds:
	\begin{align*}
		\|\eta_{c_i}^{N}\|^2 + \kappa_i \Dt\sum_{m=1}^{N}\vertiii{\eta_{c_i}^{m}}^2 
		\lesssim   h^{2k+2}  + (\Dt)^2  + \Dt\sum_{m=1}^{N}  \left(\|\bta_u^{m-1}\|^2 + \|\eta_{c_i}^{m-1}\|^2 + \vertiii{\eta_{\phi}^{m}}^2   +  h^{-2}\|\eta_{c_i}^{m-1}\|^2\vertiii{\eta_{\phi}^{m}}^2\right).
	\end{align*}
\end{lemma}
\begin{proof}
	Take $\chi_{ih}=\eta_{c_i}^{m}$ in \eqref{erretaci} and use the fact $2(a-b)a=a^2-b^2+(a-b)^2$ to find
	\begin{align}\label{l2cest1}
		\frac{1}{2\Dt}\left(\|\eta_{c_i}^{m}\|^2-\|\eta_{c_i}^{m-1}\|^2 + \|\eta_{c_i}^{m}-\eta_{c_i}^{m-1}\|^2\right) + \kappa_i \vertiii{\eta_{c_i}^{m}}^2 = -\frac{1}{\Dt} \left( \xi_{c_i}^{m}-\xi_{c_i}^{m-1}, \eta_{c_i}^{m}\right) + E_{c_i}(\eta_{c_i}^{m}) \nonumber\\
		+ \Lambda_{c_i}(\eta_{c_i}^{m}) + H_{c_i}(\eta_{c_i}^{m}).
	\end{align}
	Using Cauchy-Schwartz inequality and Young's inequality with the discrete Poincar\'e inequality \eqref{DPI}, the first two terms on the right-hand side of \eqref{l2cest1} can be bounded as
	\begin{align}\label{l2cest2}
		|\left( \frac{\xi_{c_i}^{m}-\xi_{c_i}^{m-1}}{\Dt}, \eta_{c_i}^{m}\right)| \le \left(\frac{1}{\Dt}\int_{t_{m-1}}^{t_{m}}\|\pt\xi_{c_i}(s)\|\ds\right) \|\eta_{c_i}^{m}\| \le \frac{C}{\Dt}\int_{t_{m-1}}^{t_{m}}\|\pt\xi_{c_i}(s)\|^2\ds + \frac{\kappa_i}{8}\vertiii{\eta_{c_i}^{m}}^2,
	\end{align}
	and
	\begin{align}\label{l2cest3}
		|E_{c_i}(\eta_{c_i}^{m})| \le \left(\frac{1}{\Dt}\int_{t_{m-1}}^{t_{m}}(s-t_{m-1})\|\partial_{tt}\tc_i(s)\|\ds\right) \|\eta_{c_i}^{m}\| \le C \Dt \int_{t_{m-1}}^{t_{m}} \|\partial_{tt}\tc_i(s)\|^2\ds +  \frac{\kappa_i}{8}\vertiii{\eta_{c_i}^{m}}^2
	\end{align}
	Using Lemma \ref{lem:trilinear} with the Young's inequality, we bound the third term of the right-hand side of \eqref{l2cest1} as
	\begin{align}\label{lelambdaest}
		\Lambda_{c_i}(\eta_{c_i}^{m}) \le & ~ \frac{\kappa_{i}}{8} \vertiii{\eta_{c_i}^{m}}^{2} +  C \bigg(h^{-2}\|\bxi_{u}^{m-1}\|^{2}\vertiii{\xi_{c_i}^{m}}^2 + h^{-2}\|\xi_{c_i}^{m}\|^2\vertiii{\bxi_{u}^{m-1}}^2 + \vertiii{\bxi_u^{m-1}}^{2}\vertiii{\xi_{c_i}^{m}}^{2}   \nonumber\\
		& + \left(\|\bxi_u^{m-1}\|^{2} + h^{2}\vertiii{\bxi_u^{m-1}}^{2} \right)\|\tc_{i}^{m}\|_{H^2}^{2} +\left( \|\bu^{m}-\bu^{m-1}\|^{2}+ \|\nabla(\bu^{m}-\bu^{m-1})\|^{2}\right)\|\tc_{i}^{m}\|_{H^2}^{2}  \nonumber\\
		& + \left(\|\xi_{c_i}^{m}\|^2+ h^{2}\vertiii{\xi_{c_i}^{m}}^2\right) \|\bu^{m-1}\|_{H^2}^2 + \left(h^{-4}\|\xi_{c_i}^{m}\|^2 + h^{-2}\vertiii{\xi_{c_{i}}^{m}}^2 + \|\tc_{i}^{m}\|_{H^2}^{2}\right)\|\bta_u^{m-1}\|^2\bigg).
	\end{align}
	An application of Lemma \ref{lem:Hci} with the Young's inequality, we derive the following bound
	\begin{align} \label{l2hf}
		H_{c_1}(\eta_{c_i}^{m}) &\le C \bigg( \left(\|\phi^{m}\|_{H^3}^2+ h^{-2}\vertiii{\xi_\phi^{m}}^2\right)\left(\|\xi_{c_i}^{m-1}\|^2 + h^2 \vertiii{\xi_{c_i}^{m-1}}^2 + \|\tc_{i}^{m}-\tc_{i}^{m-1}\|^2 + h^2\|\nabla(\tc_{i}^{m}-\tc_{i}^{m-1})\|^2\right) \nonumber\\
		& \quad\qquad + \left(\|\phi^{m}\|_{H^3}^2+ h^{-2}\vertiii{\xi_\phi^{m}}^2 \right) \|\eta_{c_i}^{m-1}\|^2 
		+ \left(h^{-2} \|\xi_{c_i}^{m-1}\|^2 +  \vertiii{\xi_{c_i}^{m-1}}^2  + \|\tc_{i}^{m-1}\|_{H^2}^2\right) \vertiii{\eta_{\phi}^{m}}^2 \nonumber\\
		&\quad\qquad +   h^{-2}\|\eta_{c_i}^{m-1}\|^2\vertiii{\eta_{\phi}^{m}}^2\bigg) + \frac{\kappa_i}{8}\vertiii{\eta_{c_i}^{m}}.
	\end{align}
	Substitute \eqref{l2cest2}-\eqref{l2hf} in \eqref{l2cest1}. Then, we multiply by $\Dt$ and take the summation from $m=1$ to $N$. Finally, complete the rest of the proof using the approximation estimates with the assumption of Theorem \ref{thm:l2}.
\end{proof} 
\begin{lemma}\label{lem:eqetaf}
	Under the assumption of Theorem \ref{thm:l2}, the following holds:
	\begin{align*} 
		\|\bta_u^{N}\|^2  +  \nu\Dt\sum_{m=1}^{N}\vertiii{\hbta_u^{m}}^{2}    + \frac{(\Dt)^2}{2} \vertiii{\eta_p^{N}}^{2}  \lesssim  h^{2k+2} + (\Dt)^2   + \Dt\sum_{m=1}^{N}  \left(\|\bta_u^{m-1}\|^2   +   \|\eta_{c_1}^{m}\|^2 +\|\eta_{c_2}^{m}\|^2 \right)\\
		+ \Dt\sum_{m=1}^{N} \left(1+\|\eta_{c_1}^{m-1}\|^2+\|\eta_{c_2}^{m-1}\|^2\right)  \vertiii{\eta_{\phi}^{m}}^2   + (\Dt)^3 \sum_{m=1}^{N}  \vertiii{\eta_p^{m-1}}^2.
	\end{align*}
\end{lemma}
\begin{proof}
	Choose $\bv_h = \hbta_u^{m}$ in \eqref{erretau} and $q_h = \frac{1}{2}(\eta_p^{m}+\eta_p^{m-1})$ in \eqref{erretap} and $\bw_h=\frac{1}{2}(\bta_u^{m}+\hbta_u^{m})$ in \eqref{erretauf} and add all resulting equations to obtain
	\begin{align} \label{eqetau}
		\frac{1}{2\Dt} & \left(\|\bta_u^{m}\|^2-\|\bta_u^{m-1}\|^2 + \|\hbta_u^{m}-\bta_u^{m-1}\|^2\right) + \nu\vertiii{\hbta_u^{m}}^{2}  + \frac{\Dt}{2} \left(\vertiii{\eta_p^{m}}^{2}-\vertiii{\eta_p^{m-1}}^2\right) \nonumber\\
		= & -\frac{1}{\Dt}\left(\bxi_u^{m}-\bxi_u^{m-1}, \hbta_u^{m}\right) + E_u(\hbta_u^{m}) + \Lambda_{u}(\hbta_u^{m}) + F_u(\hbta_u^{m}) + \frac{1}{2}\cD(\hbta_u^{m},p^{m}-p^{m-1}) \nonumber\\
		& - \frac{1}{2}\cD(\hbta_u^{m}, \xi_p^{m}-\xi_p^{m-1}) 
		- \frac{1}{2}\cD(\bta_u^{m},p^{m}-p^{m-1}) + \frac{1}{2}\cD(\bta_u^{m}, \eta_p^{m}-\eta_p^{m-1}) + \frac{1}{2}\cD(\bta_u^{m}, \xi_p^{m}-\xi_p^{m-1}) \nonumber\\
		& + \frac{\Dt}{2} \cA_1(p^{m}-p^{m-1},\eta_p^{m}+\eta_p^{m-1}) -  \frac{\Dt}{2} \cA_1(\xi_p^{m}-\xi_p^{m-1}, \eta_p^{m}+\eta_p^{m-1})
	\end{align}
	Arguing exactly similar way of \eqref{l2cest2}-\eqref{l2cest3}, first two terms on the right hand side of \eqref{eqetau} can be bounded as
	\begin{align}\label{eqetau1}
		|\left( \frac{\bxi_{u}^{m}-\bxi_{u}^{m-1}}{\Dt}, \hbta_{u}^{m}\right)| + |E_{u}(\hbta_{u}^{m})| \le \frac{C}{\Dt}\int_{t_{m-1}}^{t_{m}}\|\pt\bxi_{u}(s)\|^2\ds + C \Dt \int_{t_{m-1}}^{t_{m}} \|\partial_{tt}\bu(s)\|^2\ds + \frac{\nu}{4}\vertiii{\hbta_{u}^{m}}^2.
	\end{align}
	A use of Lemma \ref{lem:trilinear} with the Young's inequality gives
	\begin{align}\label{lulambdaest}
		\Lambda_{u}(\hbta_{u}^{m}) \le & ~\frac{\nu}{8} \vertiii{\hbta_{u}^{m}}^{2} + C \bigg(h^{-2}\|\bxi_{u}^{m-1}\|^{2}\vertiii{\bxi_{u}^{m}}^2 + h^{-2}\|\bxi_{u}^{m}\|^2\vertiii{\bxi_{u}^{m-1}}^2 + \vertiii{\bxi_u^{m-1}}^{2}\vertiii{\bxi_{u}^{m}}^{2}    \nonumber\\
		& + \left(\|\bxi_u^{m-1}\|^{2}+ h^{2}\vertiii{\bxi_u^{m-1}}^{2} \right)\|\bu^{m}\|_{H^2}^{2}  +\left( \|\bu^{m}-\bu^{m-1}\|^{2}+ \|\nabla(\bu^{m}-\bu^{m-1})\|^{2}\right)\|\bu^{m}\|_{H^2}^{2}  \nonumber\\
		& + \left(\|\bxi_{u}^{m}\|^2+ h^{2}\vertiii{\bxi_{u}^{m}}^2\right) \|\bu^{m-1}\|_{H^2}^2 + \left(h^{-4}\|\bxi_{u}^{m}\|^2 + h^{-2}\vertiii{\bxi_{u}^{m}}^2 + \|\bu^{m}\|_{H^2}^{2}\right)\|\bta_u^{m-1}\|^2\bigg).
	\end{align}
	From Lemma \ref{lem:Tu}, we obtain the following bound
	\begin{align} \label{eqfufinal}
		|F_u(\hbta_{u}^{m})| & \le \frac{\nu}{16}\vertiii{\hbta_{u}^{m}}^2 +  C\Big(\left(h^{-2}\|\xi_{c_1}^{m-1}-\xi_{c_2}^{m-1}\|^2 + \vertiii{\xi_{c_1}^{m-1}-\xi_{c_2}^{m-1}}^2\right)\left(\vertiii{\xi_{\phi}^{m}}^2 + \vertiii{\eta_{\phi}^{m}}^2\right)  \nonumber\\
		& + \|\eta_{c_1}^{m-1}-\eta_{c_2}^{m-1}\|^2\left(\vertiii{\xi_{\phi}^{m}}^2 + \vertiii{\eta_{\phi}^{m}}^2\right) +  \left(\|\xi_{c_1}^{m-1}-\xi_{c_2}^{m-1}\|^2+ \|\eta_{c_1}^{m-1}-\eta_{c_2}^{m-1}\|^2\right)\|\phi^{m}\|_{H^2}^2 \\
		& + \left(\|\xi_{\phi}^{m}\|^2 + h^2\vertiii{\xi_{\phi}^{m}} + \|\eta_{\phi}^{m}\|^2 \right)\|\tc_{1}^{m-1}-\tc_{2}^{m-1}\|_{H^2}^2 + \left(\|\tc_1^{m}-\tc_1^{m-1}\|^2+\|\tc_2^{m}-\tc_2^{m-1}\|^2\right)\|\phi^{m}\|_{H^2}^2 \Big). \nonumber
	\end{align}
	The fifth and sixth terms of \eqref{eqetau} can be bounded as follows:
	\begin{align}
		&|	\frac{1}{2}\cD(\hbta_u^{m},p^{m}-p^{m-1}) - \frac{1}{2}\cD(\hbta_u^{m}, \xi_p^{m}-\xi_p^{m-1})| \nonumber\\
		& \quad	\le C \left(\|p^{m}-p^{m-1}\|^2 + \|\nabla(p^{m}-p^{m-1})\|^2 + \|\xi_{p}^{m}-\xi_{p}^{m-1}\|^2 + \|\nabla(\xi_{p}^{m}-\xi_{p}^{m-1})\|^2\right) + \frac{\nu}{16}\vertiii{\hbta_{u}^{m}}^2.	 	 
	\end{align}
	Using $\cD(\e_u^m,q_h)= 0$ and the Stokes projection operator \eqref{stokesproj}, we rewrite the seventh to ninth terms as below
	\begin{align}
		- \frac{1}{2}\cD(\bta_u^{m},p^{m}-p^{m-1}) + \frac{1}{2}\cD(\bta_u^{m}, \eta_p^{m}-\eta_p^{m-1}) + \frac{1}{2}\cD(\bta_u^{m}, \xi_p^{m}-\xi_p^{m-1}) \nonumber\\ 
		=  - \frac{1}{2}\cD(\bta_u^{m},p_h^{m}-p_h^{m-1})  = - \frac{1}{2}\cD(\e_u^{m}-\bxi_u^{m},p_h^{m}-p_h^{m-1}) =0.
	\end{align}	 
	The last two terms of \eqref{eqetau} can be estimated as
	\begin{align} \label{eqa3u}
		- \frac{1}{2} &\cA_1(p^{m}-p^{m-1},\eta_p^{m}+\eta_p^{m-1}) +  \frac{\Dt}{2} \cA_1(\xi_p^{m}-\xi_p^{m-1}, \eta_p^{m}+\eta_p^{m-1}) \nonumber\\
		& \le  \frac{\Dt}{2} \left(\vertiii{p^{m}-p^{m-1}}^2 +  \vertiii{\xi_p^{m}-\xi_p^{m-1}}^2\right) + \frac{(\Dt)^2}{8}\vertiii{\eta_p^{m}+\eta_p^{m-1}}^2.
	\end{align}
	Collecting all the estimates \eqref{eqetau1}-\eqref{eqa3u} and putting them in \eqref{eqetau}. Then, multiply by $\Dt$ and take summation over $m=1$ to $ N$ and
	use approximation property with assumption of Theorem \ref{thm:l2} and triangle inequality yield
	\begin{align} \label{eqetaf}
		\|\bta_u^{N}\|^2 +  \nu\Dt\sum_{m=1}^{N}\vertiii{\hbta_u^{m}}^{2}   + (\Dt)^2 \vertiii{\eta_p^{N}}^{2}  \lesssim h^{2k+2}  + (\Dt)^2   +  \Dt\sum_{m=1}^{N}  \left(\|\bta_u^{m-1}\|^2+\|\eta_{c_1}^{m}\|^2 +\|\eta_{c_2}^{m}\|^2 \right)  \nonumber\\
		+ \Dt\sum_{m=1}^{N} \left(1+\|\eta_{c_1}^{m-1}\|^2+\|\eta_{c_2}^{m-1}\|^2\right)  \vertiii{\eta_{\phi}^{m}}^2  + \frac{(\Dt)^3}{2}\sum_{m=1}^{N} \left(  \vertiii{\eta_p^{m}}^2+\vertiii{\eta_p^{m-1}}^2\right).
	\end{align}
	The last term of the right-hand side of the above inequality \eqref{eqetaf} is rewritten as
	\begin{align*}
		\frac{(\Dt)^3}{2}\sum_{m=1}^{N} \left(  \vertiii{\eta_p^{m}}^2+\vertiii{\eta_p^{m-1}}^2\right) = \frac{(\Dt)^3}{2} \vertiii{\eta_p^{N}}^2 + (\Dt)^3\sum_{m=1}^{N} \vertiii{\eta_p^{m-1}}^2.
	\end{align*}
	Substitute it in \eqref{eqetaf} completes the rest of the proof.
\end{proof}
Now, we are ready to prove all the estimates of Theorem \ref{thm:l2}.
\begin{proof}[Proof of Theorem \ref{thm:l2}]
	From \eqref{errorsplit} and \eqref{usplit}, we find that
	\begin{align*}
		\|c_1(t_N)-c_{1h}^{N}\| \le \|\xi_{c_1}^N\| +\| \eta_{c_1}^N\|,  ~~
		\|c_2(t_N)-c_{2h}^{N}\| \le \|\xi_{c_2}^N\| + \|\eta_{c_2}^N\|,  ~~
		\|\bu(t_N)-\bu_h^{N}\| \le \|\bxi_{u}^N\| + \|\bta_{u}^N\|.
	\end{align*}
	From the approximation properties (Lemmas \ref{lem:modi} and \ref{lem:stokes}), we have the bounds for $ \|\xi_{c_1}^N\|, \|\xi_{c_2}^N\|$ and $ \|\bxi_{u}^N\|$. Hence, it is enough to find the estimates for $\| \eta_{c_1}^N\|, \|\eta_{c_2}^N\| $ and $ \|\bta_{u}^N\|$.
	First we recall the Lemma \ref{lem:l2cfinal} for $i=1$ and $i=2$ and combine these two with Lemma \ref{lem:eqetaf} to obtain
	\begin{align}
		\|\eta_{c_1}^{N}\|^2  + \|\eta_{c_2}^{N}\|^2 & + \|\bta_u^{N}\|^2  + \frac{(\Dt)^2}{2}\vertiii{\eta_p^{N}}^{2} + \sum_{m=1}^{N}\left(\kappa_1\vertiii{\eta_{c_1}^{m}}^2 + \kappa_2\vertiii{\eta_{c_2}^{m}}^2 + \nu\vertiii{\hbta_{u}^{m}}^2\right)   \nonumber\\
		\lesssim ~ &  h^{2k+2} + (\Dt)^2 + \Dt\sum_{m=1}^{N} \left( \|\bta_u^{m-1}\|^2 + \|\eta_{c_1}^{m-1}\|^2+\|\eta_{c_2}^{m-1}\|^2 + \vertiii{\eta_{\phi}^{m}}^2\right)  \nonumber\\
		&   + \Dt\sum_{m=1}^{N} h^{-2}(\|\eta_{c_1}^{m-1}\|^2+\|\eta_{c_2}^{m-1}\|^2)  \vertiii{\eta_{\phi}^{m}}^2 
		+ (\Dt)^3\sum_{m=1}^{N} \vertiii{\eta_p^{m-1}}^2.
	\end{align}	
	A use of \eqref{estphi1} with $k=1$ and hypothesis of Theorem \ref{thm:l2} and Lemma \eqref{lem:estphif} yields
	\begin{align}
		\|\eta_{c_1}^{N}\|^2 & + \|\eta_{c_2}^{N}\|^2  + \|\bta_u^{N}\|^2  + \frac{(\Dt)^2}{2}\vertiii{\eta_p^{N}}^{2} + \sum_{m=1}^{N}\left(\kappa_1\vertiii{\eta_{c_1}^{m}}^2 + \kappa_2\vertiii{\eta_{c_2}^{m}}^2 + \nu\vertiii{\hbta_{u}^{m}}^2\right)  \nonumber\\
		\lesssim & \left(h^{2k+2}  + (\Dt)^2 \right)  
		+ \Dt\sum_{m=1}^{N}  \left(\|\bta_u^{m-1}\|^2 + \|\eta_{c_1}^{m-1}\|^2+\|\eta_{c_2}^{m-1}\|^2 + \frac{(\Dt)^2}{2}\vertiii{\eta_p^{m-1}}^2\right)  \nonumber\\
		+ & \Dt\sum_{m=1}^{N} h^{-2}(\|\eta_{c_1}^{m-1}\|^2+\|\eta_{c_2}^{m-1}\|^2) \left(\|\bta_u^{m-1}\|^2 + \|\eta_{c_1}^{m-1}\|^2+\|\eta_{c_2}^{m-1}\|^2 + \frac{(\Dt)^2}{2}\vertiii{\eta_p^{m-1}}^2\right).
	\end{align}
	An application of the discrete Gronwall's lemma gives
	\begin{align}\label{estfinal}
		\|\eta_{c_1}^{N}\|^2  + \|\eta_{c_2}^{N}\|^2  + \|\bta_u^{N}\|^2  + \frac{(\Dt)^2}{2}\vertiii{\eta_p^{N}}^{2} + \sum_{m=1}^{N}\left(\kappa_1\vertiii{\eta_{c_1}^{m}}^2 + \kappa_2\vertiii{\eta_{c_2}^{m}}^2 + \nu\vertiii{\hbta_{u}^{m}}^2\right) \nonumber\\
		\le  C \left(h^{2k+2}  + (\Dt)^2 \right).
	\end{align}
	where $C = C \exp(M_N)$ with
	\begin{equation*}
		M_N = C \Dt \sum_{m=1}^{N} \left(1 + h^{-2}(\|\eta_{c_1}^{m-1}\|^2+\|\eta_{c_2}^{m-1}\|^2)\right).
	\end{equation*}
	Using method of induction with the assumption $\Delta t = \mathcal{O}(h)$, one can find the explicit form of $M_N$ as
	\begin{equation*}
		M_N = Ct_N + C(h^{2k}+1)\Dt\sum_{i=1}^{N-1}\exp(M_i).
	\end{equation*}
	Combine the result of \eqref{estfinal} with approximation estimates to complete the rest of the proof.
\end{proof}

\begin{remark} \label{l2energy}
	A use of triangle inequality with Lemmas \ref{lem:estnon}, \ref{lem:modi}  and \ref{lem:stokes} and equation \eqref{estfinal} gives
	\begin{align*}
		\Delta t\sum_{m=1}^{N}\left(\kappa_1\vertiii{e_{c_1}^{m}}^2 + \kappa_2\vertiii{e_{c_2}^{m}}^2 + \nu\vertiii{\heu^{m}}^2 + \vertiii{e_\phi^m}^2\right) \le  C \left(h^{2k}  + (\Dt)^2 \right).
	\end{align*}
	
\end{remark}

Now, we prove the $L^2$-error bound for the electrostatic potential $\phi$ using a duality argument.

\begin{theorem} \label{thm:l2phi}
	Under the hypothesis of Theorem \ref{thm:l2} , then there exists a positive constant $C$ such that following holds true for any $N\ge 1$
	\begin{equation*}
		\|\phi(t_N)-\phi_h^N\| \le C\left( h^{k+1} + \Delta t \right).
	\end{equation*}
\end{theorem}

\begin{proof}
	We first consider the following dual problem:
	\begin{align} 
		- &\Delta \Phi = g \quad \text{in}~ \Omega \label{dual1}\\
		& \Phi = 0 \quad \text{on}~ \partial\Omega,
	\end{align}
	with the following regularity estimate
	\begin{equation} \label{dualreg}
		\|\Phi\|_{H^2} \lesssim \|g\|.
	\end{equation}
	We multiply \eqref{dual1} with $\psi_h$ and take the integration over the domain $\Omega$. Then, an application of integration by parts gives
	\begin{equation}\label{dualeq2}
		\cA_1(\Phi,\psi_h) = (g, \psi_h)
	\end{equation}
	A use of \eqref{eqritz} in \eqref{dualeq2} yields
	\begin{align}\label{dualeq3}
		(g, \psi_h) = \cA_1(R_h\Phi,\psi_h)
	\end{align}
	We now choose $\psi_h=\eta_\phi^m$ in \eqref{dualeq3} and use \eqref{erretaphi} with the Cauchy-Schwarz inequality, Theorem \ref{thm:l2} and regularity assumption of Theorem \ref{thm:l2} to obtain 
	\begin{align}
		(g, \eta_\phi^m) & = \left(\xi_{c_1}^{m-1}-\xi_{c_2}^{m-1}, R_h\Phi\right) + \left(\eta_{c_1}^{m-1}-\eta_{c_2}^{m-1}, R_h\Phi\right) + \left((\tc_1^{m}-\tc_1^{m-1})-(\tc_2^{m}-\tc_2^{m-1}), R_h\Phi\right)\nonumber\\
		& \le \left(\|\xi_{c_1}^{m-1}-\xi_{c_2}^{m-1}\| + \|\eta_{c_1}^{m-1}-\eta_{c_2}^{m-1}\| + \|(\tc_1^{m}-\tc_1^{m-1})-(\tc_2^{m}-\tc_2^{m-1})\|\right) \|R_h\Phi\| \nonumber\\
		& \lesssim (h^{2k+2}+(\Delta t)^2) \|R_h\Phi\|.
	\end{align}
	We complete the rest of the proof by taking $g=\eta_\phi^m$ with regularity estimate \eqref{dualreg}.
\end{proof}

\subsection{Optimal \texorpdfstring{$H^1$}{H1} Error Bounds}

In this section, we derive error estimates for the fully discrete solution in the energy norm. 

\begin{lemma} \label{lem:h1c}
	Under the hypothesis of Theorem \ref{thm:l2}, the following holds for any $N>0$:
	\begin{equation*}
		\kappa_{i} \vertiii{\eta_{c_i}^{N}}^2  + \kappa_i \Delta t \sum_{m=1}^{N}  \vertiii{\eta_{c_i}^{m}-\eta_{c_i}^{m-}}^2 + \Delta t \sum_{m=1}^{N} \|\bar{\partial}\eta_{c_i}^{m}\|^2 \lesssim \left(h^{2k} + (\Delta t)^2\right).
	\end{equation*}
\end{lemma}

\begin{proof}
	We take $\chi_{ih} = \bar{\partial}\eta_{c_i}^{m}$ in \eqref{erretaci} to find
	\begin{align} \label{esthec1}
		\begin{aligned}
			\|\bar{\partial}\eta_{c_i}^{m}\|^2 + \frac{\kappa_{i}}{2\Delta t}\left( \vertiii{\eta_{c_i}^{m}}^2 - \vertiii{\eta_{c_i}^{m-1}}^2 + \vertiii{\eta_{c_i}^{m}-\eta_{c_i}^{m-}}^2 \right) = -\left(\bar{\partial}\xi_{c_i}^{m}, \bar{\partial}\eta_{c_i}^{m}\right)+ E_{c_i}(\bar{\partial}\eta_{c_i}^{m}) 
			\\
			+ \Lambda_{c_i}(\bar{\partial}\eta_{c_i}^{m}) + H_{c_i}(\bar{\partial}\eta_{c_i}^{m}).
		\end{aligned}
	\end{align}
	A use of the Cauchy-Schwarz inequality along with Young's inequality yields the bound for the first two terms of the right-hand side of the above equation
	\begin{align} \label{esthec2}
		|\left(\bar{\partial}\xi_{c_i}^{m}, \bar{\partial}\eta_{c_i}^{m}\right)|+|E_{c_i}(\bar{\partial}\eta_{c_i}^{m})| \le \frac{C}{\Delta t} \int_{t_{m-1}}^{t_m} \|\pt\xi_{c_i}(t)\|^2 \dt + C\Delta t \int_{t_{m-1}}^{t_m} \|\partial_{tt}\tc_{i}(t)\|^2 \dt + \frac{1}{4}\|\bar{\partial}\eta_{c_i}^{m}\|^2.
	\end{align}
	An appropriate application of Lemma \ref{lem:estnon} with the Cauchy-Schwartz inequality and Young's inequality helps us to find the following bound for the nonlinear term in \eqref{esthec1} as
	\begin{align}  \label{esthec3}
		|\Lambda_{c_i}(\bar{\partial}\eta_{c_i}^{m})| \le C & \left( h^{-2}\left(\|\e_{u}^{m-1}\|^2 \vertiii{e_{c_i}^{m}}^2 + \vertiii{\e_{u}^{m}}^2\|e_{c_i}^{m-1}\|^2\right) + \vertiii{\e_{u}^{m-1}}^2\vertiii{e_{c_i}^{m}}^2 +  \vertiii{\e_{u}^{m-1}}^2\|\tc_{i}^{m}\|_{H^2}^2 \right.  \nonumber\\
		&\left. +  \vertiii{e_{c_i}^{m}}^2\|\bu^{m-1}\|_{H^2}^2 + \|\nabla(\bu^{m}-\bu^{m-1})\|^2\|\tc_{i}^{m}\|_{H^2}^2\right) + \frac{1}{8}\|\bar{\partial}\eta_{c_i}^{m}\|^2.
	\end{align}
	In a similar way to Lemma \ref{lem:Hci} using an appropriate norm, we can bound the last term of the right-hand side of \eqref{esthec1} by
	\begin{align}  \label{esthec4}
		|H_{c_i}(\bar{\partial}\eta_{c_i}^{m})| \le C\left(h^{-4}\left(\|e_{c_i}^{m-1}\|^2+h^2\vertiii{e_{c_i}^{m-1}}^2\right)\left(\vertiii{\eta_{\phi}^{m}}^2+\vertiii{\xi_{\phi}^{m}}^2\right) +  \|\Delta_h\eta_{\phi}^{m}\|^2\|\tc_{i}\|_{H^2}^2  \right. \nonumber
		\\
		+\left.  \left(\|\tc_{i}^{m}-\tc_{i}^{m-1}\|^2 + h^2\|\nabla^2(\tc_{i}^{m}-\tc_{i}^{m-1})\|^2\right)\|\phi^m\|_{H^3}^2\right) + \frac{1}{8}\|\bar{\partial}\eta_{c_i}^{m}\|^2.
	\end{align}
	Combining all the above bounds \eqref{esthec2}-\eqref{esthec4} in \eqref{esthec1} and multiplying by $2\Delta t$ and taking summation from $m=1$ to $N$ We find that
	\begin{align} \label{esthec5}
		\kappa_{i} \vertiii{\eta_{c_i}^{N}}^2  + & \kappa_i\Delta t \sum_{m=1}^{N}  \vertiii{\eta_{c_i}^{m}-\eta_{c_i}^{m-}}^2 + \Delta t \sum_{m=1}^{N} \|\bar{\partial}\eta_{c_i}^{m}\|^2 \lesssim   \int_{0}^{t_N} \|\pt\xi_{c_i}(t)\|^2 \dt + (\Delta t)^2 \int_{0}^{t_N} \|\partial_{tt}\tc_{i}(t)\|^2 \dt\nonumber\\
		& + \Delta t\sum_{m=1}^{N}   \left( h^{-2} \|\e_{u}^{m-1}\|^2 \vertiii{e_{c_i}^{m}}^2 + h^{-2} \vertiii{\e_{u}^{m}}^2\|e_{c_i}^{m-1}\|^2  + \vertiii{\e_{u}^{m-1}}^2\vertiii{e_{c_i}^{m}}^2  \right)  \nonumber\\
		& +\Delta t \sum_{m=1}^{N} \left( \vertiii{\e_{u}^{m-1}}^2\|\tc_{i}^{m}\|_{H^2}^2 +    \vertiii{e_{c_i}^{m}}^2\|\bu^{m-1}\|_{H^2}^2 + \|\nabla(\bu^{m}-\bu^{m-1})\|^2\|\tc_{i}^{m}\|_{H^2}^2 \right) \nonumber\\
		&  + \Delta t\sum_{m=1}^{N}  \left( h^{-4}\left(\|e_{c_i}^{m-1}\|^2+h^2\vertiii{e_{c_i}^{m-1}}^2\right)\left(\vertiii{\eta_{\phi}^{m}}^2+\vertiii{\xi_{\phi}^{m}}^2\right) + \|\Delta_h\eta_{\phi}^{m}\|^2\|\tc_{i}\|_{H^2}^2\right) \nonumber\\
		& + \Delta t\sum_{m=1}^{N}\left(  \left(\|\tc_{i}^{m}-\tc_{i}^{m-1}\|^2 + h^2\|\nabla^2(\tc_{i}^{m}-\tc_{i}^{m-1})\|^2\right)\|\phi^m\|_{H^3}^2\right).
	\end{align}
	We now apply Lemma \ref{lem:estphif}, Theorem \ref{thm:l2}, Remark \ref{l2energy} and Lemma \ref{lem:modi} along with the regularity assumptions Theorem \ref{thm:l2} and $\Delta t = \mathcal{O}(h^2)$ to conclude the rest of the proof.
\end{proof}

\begin{lemma} \label{lem:h1u}
	Under the hypothesis of Theorem \ref{thm:l2}, the following estimate holds for any $N>0$:
	\begin{align*}
		\nu\vertiii{\bta_{u}^{N}}^2 & + \Delta t \sum_{m=1}^{N} \|\bar{\partial}\bta_{u}^{m}\|^2 + \Delta t \sum_{m=1}^{N} \vertiii{\bta_{u}^{m}-\bta_{u}^{m-1}}^2  \lesssim   \left(h^{2k}+(\Delta t)^2\right).
	\end{align*}
	where $\bar{\partial}\e_u^{m} = \frac{\e_u^{m}-\e_u^{m-1}}{\Dt}$. 
\end{lemma}
\begin{proof}
	We choose $\bv_h = \bar{\partial}\bta_u^{m}$ in \eqref{erretau} and $\bw_h = \bar{\partial}\bta_u^{m}$ in \eqref{erretauf}. Then, add both the resulting equations and simplify them with the help of the Stokes operator \eqref{stokesproj} and the continuity equation to find
	\begin{align} \label{h1est1}
		\|\bar{\partial}\bta_{u}^{m}\|^2 + \nu \cA_2(\hbta_{u}^{m}, \bar{\partial}\bta_{u}^{m}) = -\left( \bar{\partial}\bxi_{u}^{m},\bar{\partial}\bta_{u}^{m} \right) + E_u(\bar{\partial}\bta_{u}^{m}) + \Lambda_u(\bar{\partial}\bta_{u}^{m}) + F_u(\bar{\partial}\bta_{u}^{m}).
	\end{align}
	To estimate the second term of the left side, we first define a discrete operator $A_h:\bV_h\to\bV_h$ by $(A_h\bv_h,\bw_h) = \cA_2(\bv_h,\bw_h)$. Now we choose $\bw_h = A_h\bar{\partial}\bta_{u}^{m}$ in \eqref{erretauf} and apply \eqref{stokesproj} to obtain
	\begin{align}\label{h1est2}
		\cA_2(\hbta_{u}^{m}, \bar{\partial}\bta_{u}^{m}) & = \cA_2(\bta_{u}^{m}, \bar{\partial}\bta_{u}^{m}) - \Delta t \cD(A_h\bar{\partial}\bta_{u}^{m}, p_h^m-p_h^{m-1}) \nonumber\\
		& = \frac{1}{2\Delta t}\left(\vertiii{\bta_{u}^{m}}^2 - \vertiii{\bta_{u}^{m-1}}^2 + \vertiii{\bta_{u}^{m}-\bta_{u}^{m-1}}^2\right).
	\end{align}
	An application of the Cauchy-Schwarz inequality, along with the Young's inequality, yields
	\begin{align} \label{h1est3}
		|\left(\bar{\partial}\bxi_{u}^{m},\bar{\partial}\bta_{u}^{m} \right)| + |E_u(\bar{\partial}\bta_{u}^{m})| \le \frac{C}{\Delta t}\int_{t_{m-1}}^{t_m} \|\pt\bxi_{u}(t)\|^2 \dt + C\Delta t \int_{t_{m-1}}^{t_m} \|\partial_{tt}\bu(t)\|^2
		\dt + \frac{1}{4} \|\bar{\partial}\bta_{u}^{m}\|^2.
	\end{align}
	Arguing the similar way of \eqref{esthec3}, we find that
	\begin{align}  \label{h1est4}
		|\Lambda_{u}(\bar{\partial}\eta_{u}^{m})| \le C & \left( h^{-2}\left(\|\e_{u}^{m-1}\|^2 \vertiii{\e_{u}^{m}}^2 + \vertiii{\e_{u}^{m}}^2\|\e_{u}^{m-1}\|^2\right) + \vertiii{\e_{u}^{m-1}}^2\vertiii{\e_{u}^{m}}^2 +  \vertiii{\e_{u}^{m-1}}^2\|\bu^{m}\|_{H^2}^2 \right.  \nonumber\\
		&\left. +  \vertiii{\e_{u}^{m}}^2\|\bu^{m-1}\|_{H^2}^2 + \|\nabla(\bu^{m}-\bu^{m-1})\|^2\|\bu^{m}\|_{H^2}^2\right) + \frac{1}{8}\|\bar{\partial}\eta_{u}^{m}\|^2.
	\end{align}
	Similar to Lemma \ref{lem:Tu}, an appropriate use of the H\"{o}lder inequality gives
	\begin{align} \label{h1est5}
		|F_u(\bar{\partial}\eta_{u}^{m})|  \le C\Big( \left(h^{-2}\|e_{c_1}^{m-1}-e_{c_2}^{m-1}\|^2+\vertiii{\xi_{c_1}^{m-1}-\xi_{c_2}^{m-1}}^2\right)\vertiii{e_{\phi}^{m}}^2   + \|e_{c_1}^{m-1}-e_{c_2}^{m-1}\|^2\|\phi^{m}\|_{H^3}^2 \nonumber\\
		+ \vertiii{e_{\phi}^{m}}^2 \|\tc_{1}^{m-1}-\tc_{2}^{m-1}\|_{H^2}^2 + \left(\|\tc_1^{m}-\tc_1^{m-1}\|^2+\|\tc_2^{m}-\tc_2^{m-1}\|^2\right)\|\phi^{m}\|_{H^3}^2 \Big) + \frac{1}{8}\vertiii{\bar{\partial}\eta_{u}^{m}}^2.
	\end{align}
	After inserting all the estimates \eqref{h1est2}-\eqref{h1est5} in \eqref{h1est1}, we multiply the resulting inequality by $2\Delta t$ and take summation from $m=1$ to $N$ to derive the following
	\begin{align}
		\nu\vertiii{\bta_{u}^{N}}^2 & + \Delta t \sum_{m=1}^{N} \|\bar{\partial}\bta_{u}^{m}\|^2 + \Delta t \sum_{m=1}^{N} \vertiii{\bta_{u}^{m}-\bta_{u}^{m-1}}^2 \lesssim \int_{0}^{t_N} \|\pt\bxi_{u}(t)\|^2 \dt + \left(\Delta t\right)^2\int_0^{t_N} \|\partial_{tt}\bu(t)\|^2 \dt \nonumber\\
		& + \Delta t \sum_{m=1}^{N} \left(h^{-2}\|\e_{u}^{m-1}\|^2 \vertiii{\e_{u}^{m}}^2 + h^{-2}\vertiii{\e_{u}^{m}}^2\|\e_{u}^{m-1}\|^2 + \vertiii{\e_{u}^{m-1}}^2\vertiii{\e_{u}^{m}}^2 \right)  \nonumber\\
		& + \Delta t \sum_{m=1}^{N} \left( +  \vertiii{\e_{u}^{m-1}}^2\|\bu^{m}\|_{H^2}^2 +  \vertiii{\e_{u}^{m}}^2\|\bu^{m-1}\|_{H^2}^2 + \|\nabla(\bu^{m}-\bu^{m-1})\|^2\|\bu^{m}\|_{H^2}^2\right) \nonumber\\
		& + \Delta t \sum_{m=1}^{N}  \left(\left(h^{-2}\|e_{c_1}^{m-1}-e_{c_2}^{m-1}\|^2+\vertiii{\xi_{c_1}^{m-1}-\xi_{c_2}^{m-1}}^2\right)\vertiii{e_{\phi}^{m}}^2   + \|e_{c_1}^{m-1}-e_{c_2}^{m-1}\|^2\|\phi^{m}\|_{H^3}^2\right) \nonumber\\
		& + \Delta t \sum_{m=1}^{N}\vertiii{e_{\phi}^{m}}^2 \|\tc_{1}^{m-1}-\tc_{2}^{m-1}\|_{H^2}^2 + \left(\|\tc_1^{m}-\tc_1^{m-1}\|^2+\|\tc_2^{m}-\tc_2^{m-1}\|^2\right)\|\phi^{m}\|_{H^3}^2.
	\end{align}
	A use of Lemma \ref{lem:estphif}, Theorem \ref{thm:l2}, Remark \ref{l2energy} and Lemma \ref{lem:modi}, along with the regularity assumptions of Theorem \ref{thm:l2} and $\Delta t = \mathcal{O}(h)$, completes the rest of the proof.
\end{proof}

Combining the results of Lemmas \ref{lem:ritz}, \ref{lem:modi}, and \ref{lem:stokes} along with \eqref{estphi1}, Lemmas \ref{lem:h1c} and \ref{lem:h1u}, we obtain the following theorem.

\begin{theorem}
	Under the hypothesis of Theorem \ref{thm:l2} , then there exists a positive constant $C$ such that following holds true for any $N\ge 1$
	\begin{align}
		\vertiii{\phi(t_N)-\phi_h^N} + \vertiii{\bu(t_N)-\bu_h^N} \le C\left( h^{k} + \Delta t \right), \\
		\vertiii{c_1(t_N)-\tc_{1h}^N} + \vertiii{c_2(t_N)-\tc_{2h}^N} \le C\left( h^{k} + \Delta t\right). \label{h1energyc}
	\end{align}
\end{theorem}

\subsection{Pressure Estimate}

Before finding the pressure bounds, we prove the following lemma, which will be used to derive the pressure error bound.
\begin{lemma} \label{lem:pree}
	Under the assumptions of Theorem \ref{thm:l2}, the following estimate holds for any $N>1$:
	\begin{align*}
		\Delta t \sum_{m=1}^{N} \|\theta_p^m\|^2 \lesssim \left(h^{2k}+(\Delta t)^2\right).
	\end{align*}
\end{lemma}

\begin{proof}
	Choose $\bw_{h} = \bv_{h}$ in \eqref{eeuf} and add with \eqref{eeu} to obtain
	\begin{align}\label{pree1}
		\cD(\bv_{h},\theta_p^{m}) = -\cD(\bv_{h},\xi_p^{m}) + (\bar{\partial}\e_u^{m},\bv_{h}) + \nu \cA_2(\heu^{m},\bv_h) -  E_{u}(\bv_h) - \Lambda_u(\bv_h) - F_u(\bv_h)
	\end{align}

	A use of the Cauchy-Schwarz inequality and the Young's inequality with Lemma \ref{lem:prop_bilinear} yields
	\begin{align}\label{pree2}
		|\cD(\bv_{h},\xi_p^{m})| + |(\bar{\partial}\e_u^{m},\bv_{h})| + |\nu \cA_2(\heu^{m},\bv_h)| + |E_{u}(\bv_h)|  & \lesssim \left(\|\xi_p^{m}\| + \|\bar{\partial}\e_u^{m}\| + \vertiii{\heu^{m}}\right)\vertiii{\bv_h} \nonumber\\
		& \quad + \Dt\left(\int_{t_{m-1}}^{t_m}\|\partial_{tt}\bu\|^{2}\dt \right)^{\frac{1}{2}} \vertiii{\bv_{h}}.
	\end{align}
	To bound the fifth term on the right-hand side of \eqref{pree1}, a use of Lemma \ref{lem:bdd} and \ref{lem:estnon} gives
	\begin{align*}
		|\Lambda_u(\bv_{h})| \lesssim \left(\vertiii{\e_u^{m-1}}\vertiii{\heu^{m}} + \|\e_u^{m-1}\|\|\bu^{m}\|_{H^2} +  \|\bu^{m-1}\|_{H^2}\vertiii{\heu^{m}} + \|\bu^{m}-\bu^{m-1}\|\|\bu^{m}\|_{H^2}\right)\vertiii{\bv_{h}}.
	\end{align*}
	Last term on the right-hand side of \eqref{pree1} can be bounded as
	\begin{align}
		|F_u(\bv_{h})|\lesssim\big(\vertiii{\eco^{m-1}-\ect^{m-1}}\vertiii{\e_\phi^{m}} + \vertiii{\eco^{m-1}-\ect^{m-1}}\vertiii{\phi^{m}} +  \vertiii{\tc_{1}^{m-1}-\tc_{2}^{m-1}}\vertiii{\e_\phi^{m}} \nonumber\\
		+ \|(\tc_{1}^{m}-\tc_{1}^{m-1})-(\tc_{2}^{m}-\tc_{2}^{m-1})\|\|\phi^{m}\|_{H^2}\big)\vertiii{\bv_{h}}.
	\end{align}
	Combining all the above bounds to obtain 
	\begin{align}
		&\cD(\bv_h,\theta_p^{m})  \lesssim \vertiii{\bv_{h}} \bigg( \|\xi_p^{m}\| + \|\bar{\partial}\e_u^{m}\|  + \vertiii{\heu^{m}} + \Dt\left(\int_{t_{m-1}}^{t_m}\|\partial_{tt}\bu\|^{2}\dt \right)^{\frac{1}{2}} + \vertiii{\e_u^{m-1}}\vertiii{\heu^{m}} \nonumber
		\\
		& + \vertiii{\e_u^{m-1}}\vertiii{\bu^{m}} +  \vertiii{\bu^{m-1}}\vertiii{\heu^{m}} + \|\bu^{m}-\bu^{m-1}\|\|\bu^{m}\|_{H^2} + \vertiii{\eco^{m-1}-\ect^{m-1}}\vertiii{\e_\phi^{m}} \nonumber
		\\
		& + \vertiii{\eco^{m-1}-\ect^{m-1}}\vertiii{\phi^{m}} +  \vertiii{\tc_{1}^{m-1}-\tc_{2}^{m-1}}\vertiii{\e_\phi^{m}}  + \|(\tc_{1}^{m}-\tc_{1}^{m-1})-(\tc_{2}^{m}-\tc_{2}^{m-1})\|\|\phi^{m}\|_{H^2} \bigg)
	\end{align}
	From Lemma \ref{lem:infsup}, a use of Lemma \ref{lem:stokes} and \ref{lem:h1u}, Remark \ref{l2energy} and the regularity assumption (\textbf{A1}) yields
	\begin{align}
		\Delta t \sum_{m=1}^{N} \|\theta_p^m\|^2 
		& \le \left(\frac{1}{\beta^*}\right)^2 \Delta t \sum_{m=1}^{N} \left(\sup_{\bv_h\in\tilde{\bV}_h}\frac{\cD(\bv_h,\theta_p^m)}{\vertiii{\bv_h}}\right)^2 \nonumber\\
		& \lesssim \left(h^{2k}+(\Delta t)^2\right).
	\end{align}
	An application of Lemma \ref{lem:stokes} with triangle inequality completes the rest of the proof.
\end{proof}

We obtain the following theorem by combining the result of Lemma \ref{lem:stokes}  with Lemma \ref{lem:pree}.

\begin{theorem}
	Under the hypothesis of Theorem \ref{thm:l2} , then there exists a positive constant $C$ such that following holds true for any $N\ge 1$
	\begin{align}
		\left(\sum_{m=1}^N \|p(t_m)-p_h^m\|^2\right)^\frac{1}{2}  \le C\left( h^{k} + \Delta t \right).
	\end{align}
\end{theorem}

\section{Numerical Results} \label{Sec:numerical}

This section demonstrates a few numerical examples for the fully discrete dG formulation \eqref{fds11}-\eqref{fds5}
of the system \eqref{eqphi}-\eqref{eqdiv}. We first consider one example to verify the accuracy of the fully discrete dG scheme. All numerical simulations have been computed using the software FreeFem++ \cite{Hec12}.

\subsection{Convergence Test for dG Scheme}
Since it is not easy to find an analytical solution to the system  \eqref{eqphi}-\eqref{eqdiv},  in order to verify the accuracy of the fully discrete scheme, we consider the modified system with added source terms $(f_\phi,\psi_h)$, $(f_{c_i},\chi_h)$ and $(\f_u,\bv_h)$ on the right hand sides of \eqref{fds11}, \eqref{fds12} and \eqref{fds2}, respectively.

We now consider the following example with a known solution.
\begin{example}\label{exm1}
	For the experiment, we choose $f_\phi, f_{c_i}, \f_u$ in such a way that the solution of \eqref{eqphi}-\eqref{eqdiv} becomes as follow:
	\begin{align*}
		\phi(x,y,t) &= \exp(-t)(\cos(x\pi)\cos(y\pi) - (\cos(2x\pi)\cos(2y\pi))/4)/(2\pi^2), \\
		c_1(x,y,t) &= \exp(-t)(\cos(x\pi)\cos(y\pi) + 1)/2,\\
		c_2(x,y,t) & = \exp(-t)(\cos(2x\pi)\cos(2y\pi) + 1)/2,\\
		u_1(x,y,t) &= \exp(-t)(\sin(2y\pi) - \cos(2x\pi) \sin(2y\pi)),\\
		u_2(x,y,t) &=  -\exp(-t)(\sin(2x\pi) - \cos(2y\pi)\sin(2x\pi)),\\
		p(x,y,t)   &= \exp(-t)(\cos(2x\pi) + \sin(2y\pi)).
	\end{align*}
\end{example}

With computational domain $\Omega=[0,1]\times[0,1]$ for the simulation, now the discrete space $X_h^0$ consists of  discontinuous $\mathbb{P}_k, k=1,2$ elements. On the other hand, the discrete velocity and pressure spaces $(\bV_h, M_h)$ are constructed using stable discontinuous $(\bm{\mathbb{P}}_k,\mathbb{P}_{k-1}),k=1,2$ pairs. For the numerical simulation, we set the parameters $\mu, \nu, \kappa_1, \kappa_2, \beta_1, -\beta_2$ all are equal to $1$. We choose penalty parameter $\sigma=10$ for $k=1$ and $\sigma = 40$ for $k=2$.

To verify the convergence rate in spatial direction, a regular uniform triangulation with mesh size $h=2^{-i},~i=1,2,\dots,6$ is used to discretize the domain. The time discretization parameter 
$\Delta t$ is chosen as $\mathcal{O}(h^{k+1})$ with the final time $T=0.1$. In Figures \ref{fig:P1ex1} and \ref{fig:P2ex1}, we represent the numerical errors for both the concentrations, the electrostatic potential, the fluid velocity and the fluid pressure with respect to the spatial discretizing parameter $h$ for the DG scheme \eqref{fds11}-\eqref{fds5}. From these figures, we observe that the spatial rates of convergence are $\mathcal{O}(h^{k+1})$ in $L^2$-norm for both the concentrations, the electrostatic potential, the fluid velocity and  $\mathcal{O}(h^{k})$ in $H^1$-norm. For the pressure, the convergence rate is $\mathcal{O}(h^{k})$ in $L^2$-norm. All these results coincide with our theoretical findings.
%
%
% P1dcP0 DG
\begin{figure}[h!]
	\centering
	\begin{subfigure}[b]{0.48\textwidth}
		\includegraphics[scale=0.50]{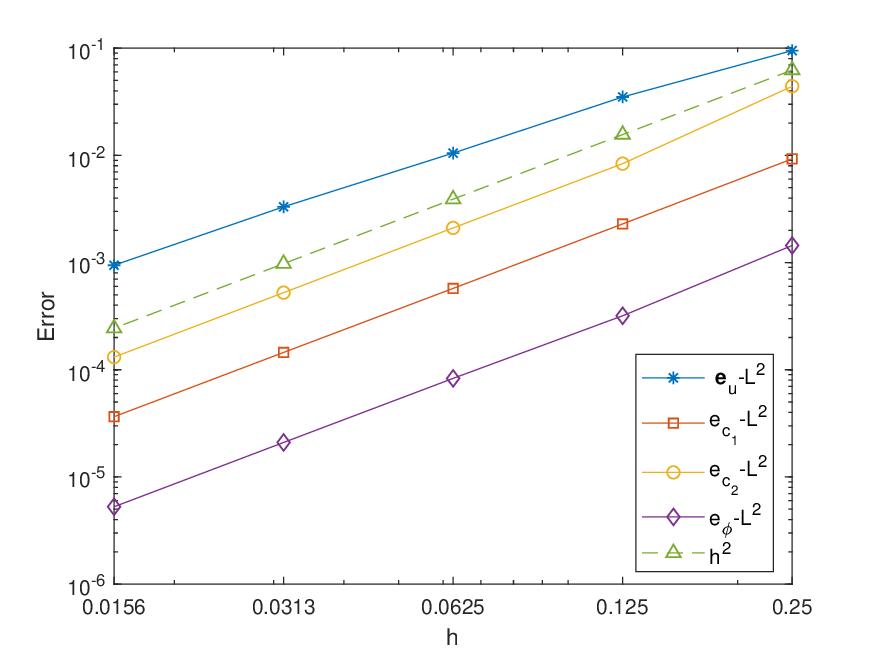}
	\end{subfigure}
	\hfill
	\begin{subfigure}[b]{0.48\textwidth}
		\includegraphics[scale=0.50]{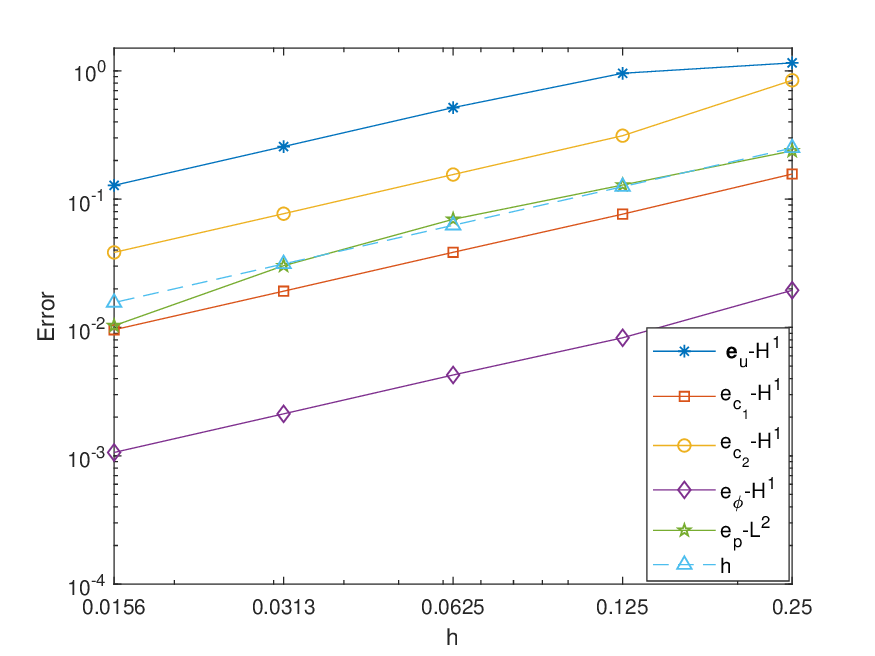}
	\end{subfigure}
	\caption{Numerical errors with $k=1$ for Example \ref{exm1} with DG method.}
	\label{fig:P1ex1}
\end{figure}

%P2P1 DG
\begin{figure}[h!]
	\centering
	\begin{subfigure}[b]{0.48\textwidth}
		\includegraphics[scale=0.50]{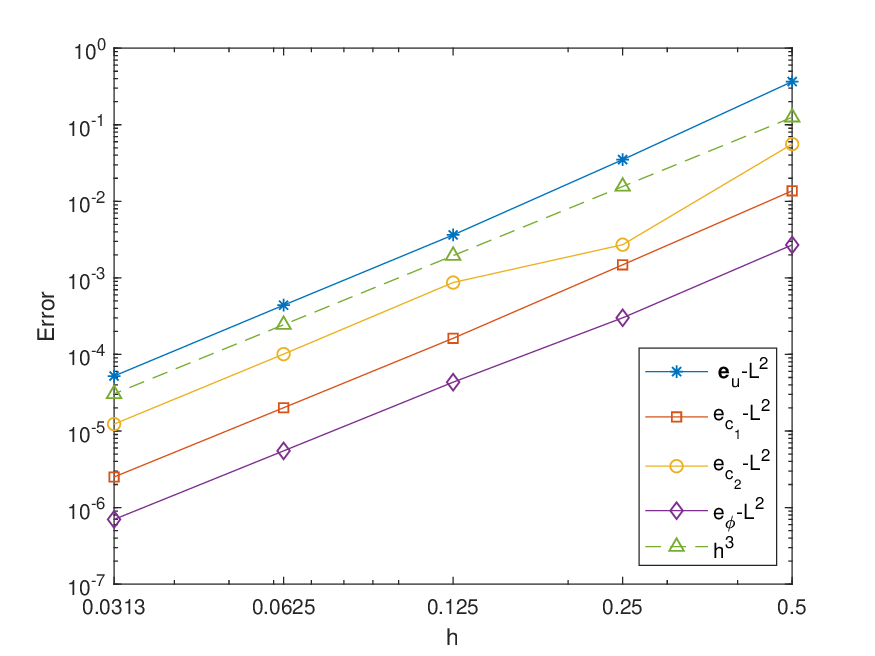}
	\end{subfigure}
	\hfill
	\begin{subfigure}[b]{0.48\textwidth}
		\includegraphics[scale=0.50]{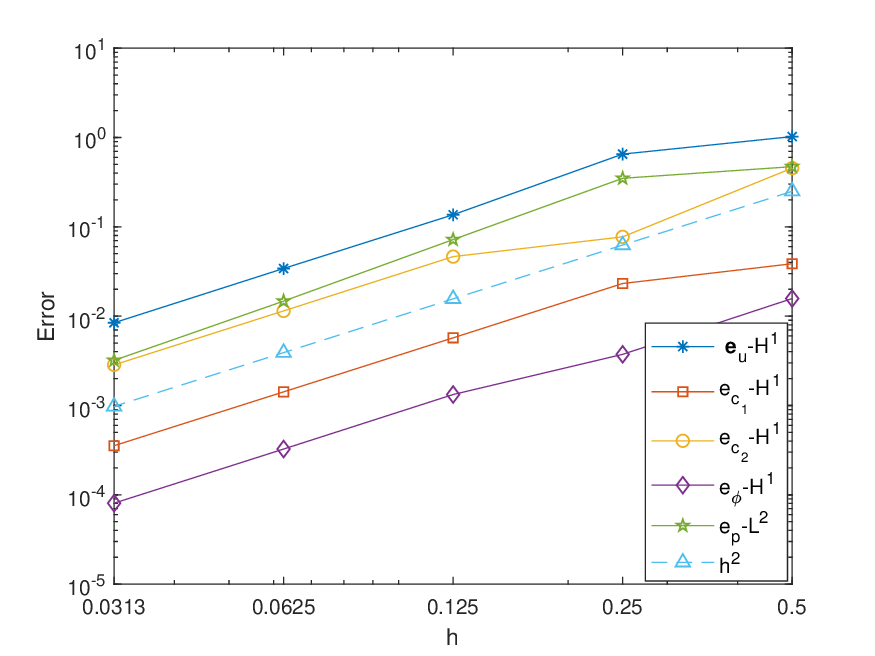}
	\end{subfigure}
	\caption{Numerical errors with $k=2$ for Example \ref{exm1} with DG method.}
	\label{fig:P2ex1}
\end{figure}

%%%%%%%%%%%%%%%%%%%%%%%%%%%%%%%%%%%%%%%%%%%%%%%%%%%%%%%%%%%%%%%%%%%%%%%%%%%%%%%%%%%%%%%%%%%%%%%%%%%%%%%%%%%%%%
%
%
In order to verify the convergence rates in temporal direction, we chose a uniform partition of the time interval $[0, 0.1]$ with $\Delta t = 0.1\times 2^{-i},~i=1,2,\dots,6$. The mesh parameter $h$ has to be fixed as $\mathcal{O}(\Delta t^{\frac{1}{k+1}})$ for $L^2$-norm and $\mathcal{O}(\Delta t^{\frac{1}{k}})$ for $H^1$-norm. The Figure \ref{fig:P1time_ex1} represents the numerical errors in $L^2$ and $H^1$-norms for $k=1,2$. From this figure, we conclude that the convergence rate is linear across all variables in the temporal direction, consistent with our theoretical results.
%
%
% Conv time DG
\begin{figure}[!h]
	\centering
	\begin{subfigure}[b]{0.48\textwidth}
		\includegraphics[scale=0.50]{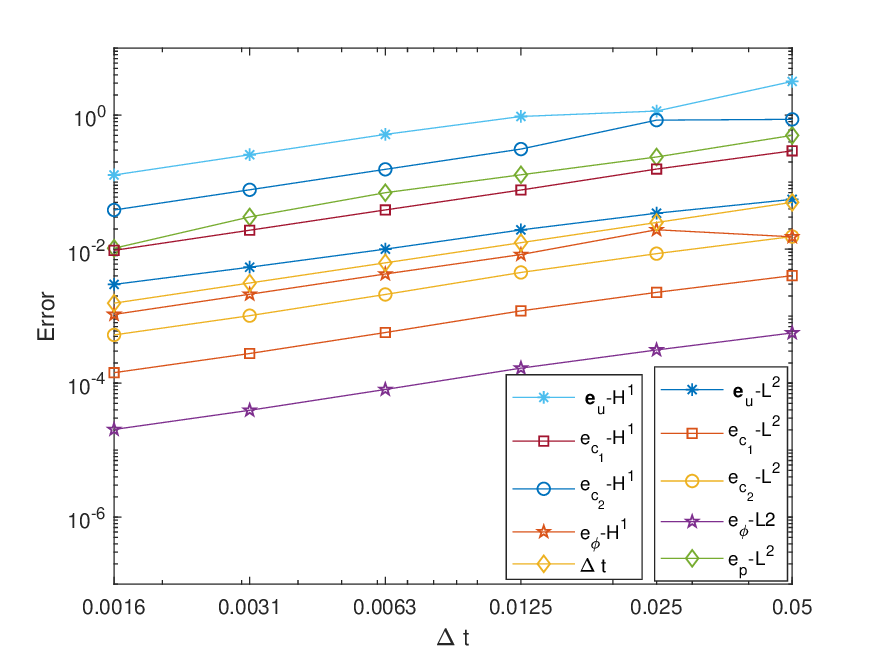}
	\end{subfigure}
	\hfill
	\begin{subfigure}[b]{0.48\textwidth}
		\includegraphics[scale=0.50]{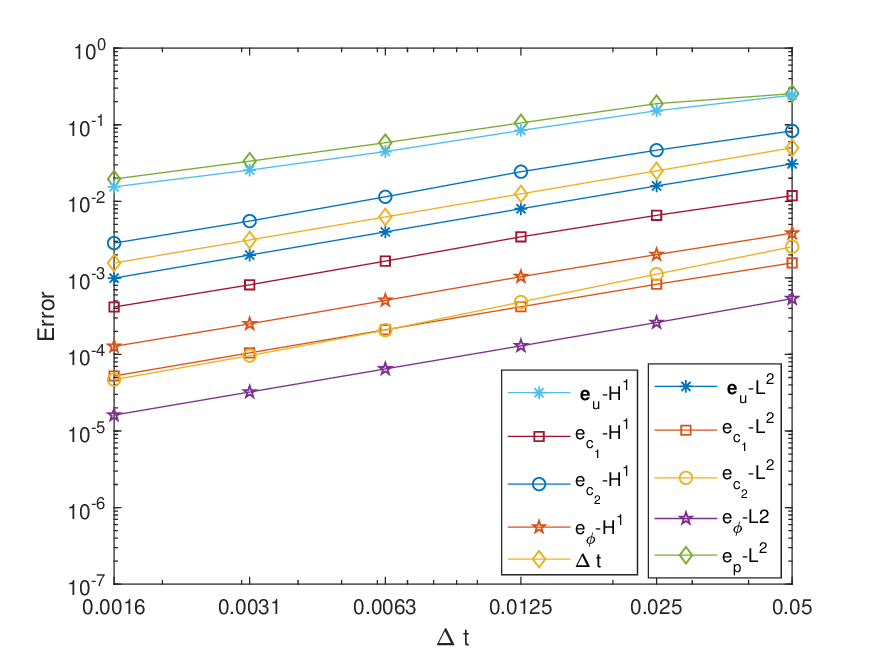}
	\end{subfigure}
	\caption{Numerical errors with respect to time variable for Example \ref{exm1} with DG method.}
	\label{fig:P1time_ex1}
\end{figure}

\subsection{Mass Conservation and Energy Dissipation}
In this subsection, we provide a few numerical results for the fully discrete system  \eqref{fds11}-\eqref{fds5}. 
\begin{example}
	To verify the physical property, we consider the following initial data:
	\begin{align*}
		u_{10} &= 10\sin(2\pi x)\cos(2\pi y), & 		u_{20} &= -10\sin(2\pi y)\cos(2\pi x), \\
		c_{10} &= \cos(2\pi x) + 1, &
		c_{20} &= \cos(2\pi y) + 1. 
	\end{align*}
	
\end{example}

We consider similar parameter values and domain for this simulation as Example \ref{exm1}. We set the space discretization parameter $h=\frac{1}{64}$ and time step $\Delta t = 0.001$. In Figure \ref{fig:mass}, we present the mass and the deviation of mass for both the concentration of positive and negative charge ions for the fully discrete DG schemes \eqref{fds11}-\eqref{fds5}.
We now observe that the mass deviation for both ions is minimal (within machine error). This verifies the mass conservation property of the fully discrete concentration of positive and negative charge ions as shown in Lemma \ref{lem:mass}. 
In Figure \ref{fig:minmax}, we present the minimum and maximum values at each time level for the ion concentrations, the electric potential energy, and total free energy. We observe that the minimum value of ions at each time level remains non-negative; further, both energies decrease as time increases.  

% Mass DG
\begin{figure}[!h]
	\centering
	\begin{subfigure}[b]{0.48\textwidth}
		\includegraphics[scale=0.450]{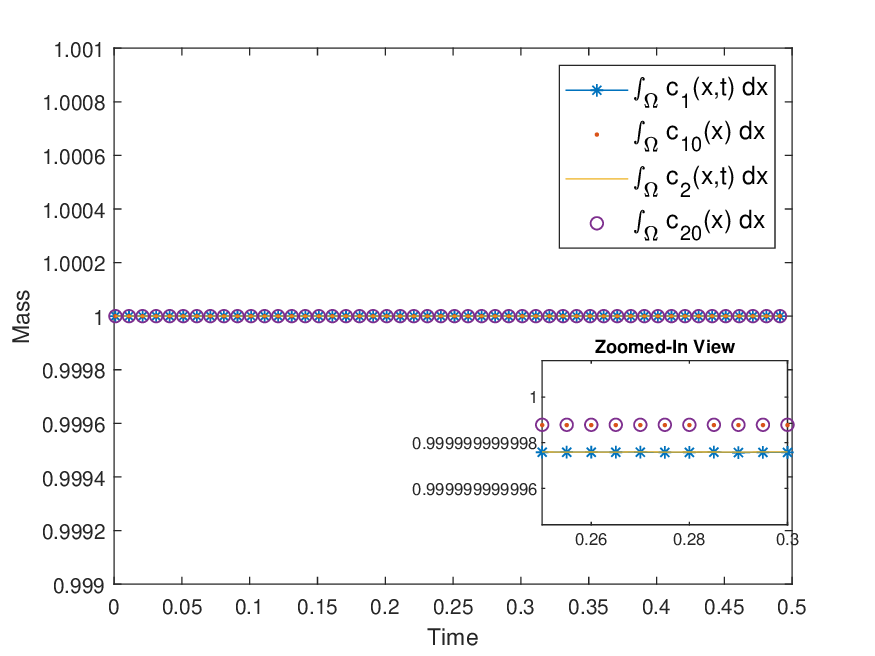}
	\end{subfigure}
	\hfill
	\begin{subfigure}[b]{0.48\textwidth}
		\includegraphics[scale=0.450]{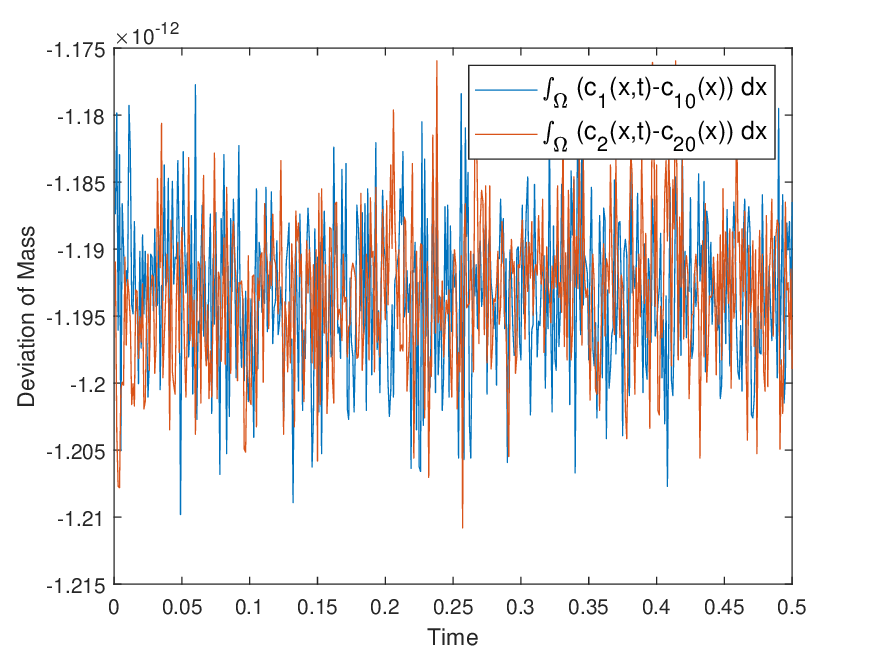}
	\end{subfigure}
	\caption{Mass (left) and deviation of mass (right) for both the positive and negative charge ions concentration with the DG method.}
	\label{fig:mass}
\end{figure}
%
% Min max and energy DG
\begin{figure}[!h]
	\centering
	\begin{subfigure}[b]{0.48\textwidth}
		\includegraphics[scale=0.450]{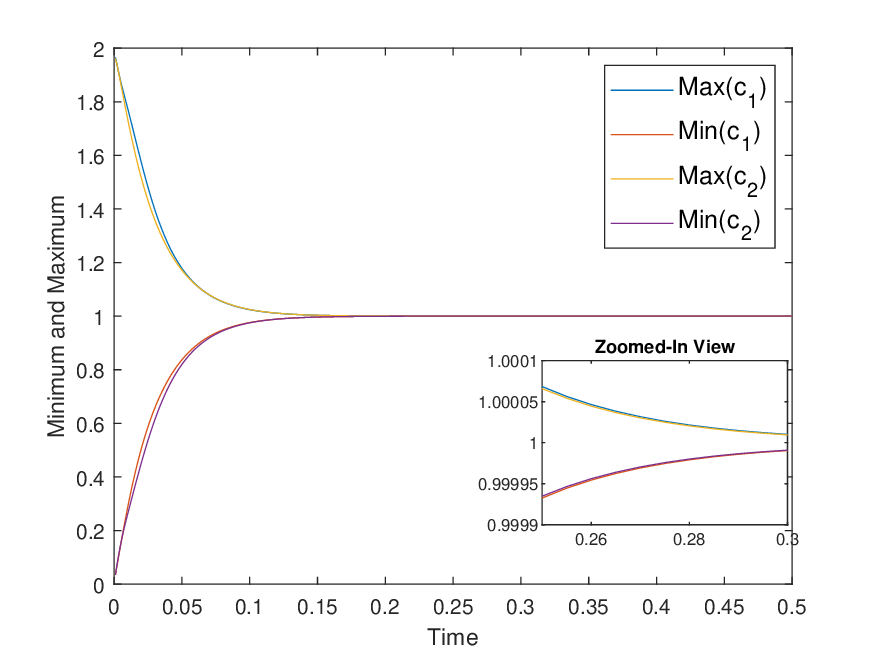}
	\end{subfigure}
	\hfill
	\begin{subfigure}[b]{0.48\textwidth}
		\includegraphics[scale=0.450]{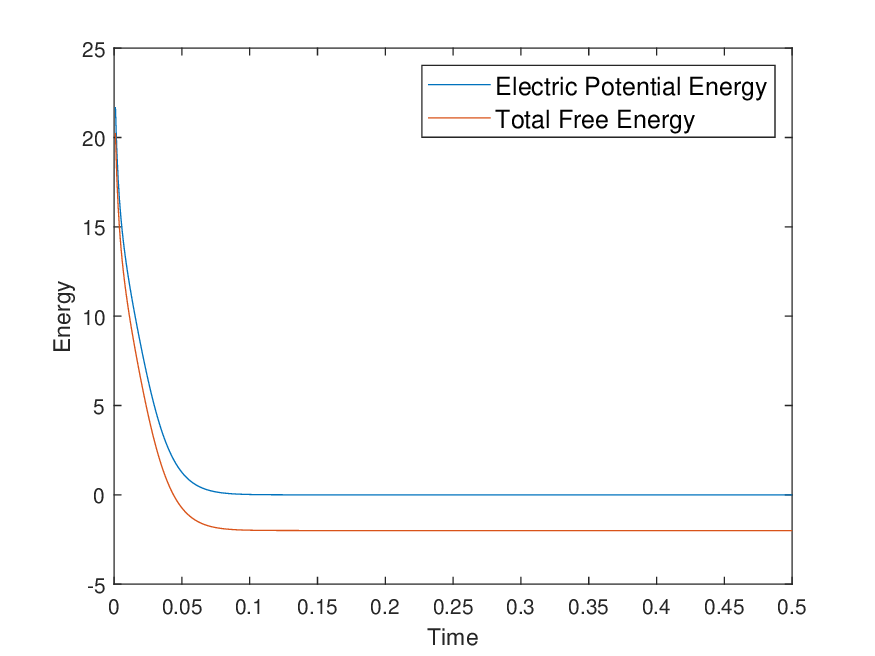}
	\end{subfigure}
	\caption{Minimum and maximum values at time $t$ for both the positive and negative charge ions concentration (left) and electric potential and total energy (right) of the system with DG method.}
	\label{fig:minmax}
\end{figure}

\subsection{Ion Spreading in a Charged Reservoir}
This subsection presents a numerical example of how ions spread and move in the presence of an electric field within a closed micro-channel with charged walls, inspired by \cite[Section 5.2]{LBM20}. With the micro-channel $\Omega = [0, 1]\times [0, 2]$ and the geometry given in Figure \ref{fig:geo1}, we solve the system \eqref{eqphi}-\eqref{eqdiv} with no-slip boundary everywhere for fluid flow ($\bu = 0$) and zero flux for both the ions ($\nabla c_i\cdot\bn = 0$) given in \eqref{eqbd}. However, we consider a sub-boundary condition for the electrostatic potential: we assume a uniformly positively charged surface ($\sigma = 1$) on the lower boundary, the upper boundary is grounded ($\phi = 0$), and the other two walls are insulators ($\nabla\phi\cdot\bn=0$).
We also consider the Gaussian profile as the initial data for the ions, with a positive profile on the lower left side of the channel and an equal negative profile on the upper right portion of the channel. The initial concentration profile is given by
\begin{equation}
	c_i(x,0) = \frac{r_0}{2\pi R^2}\exp\left\{{-\frac{((x-\frac{L_x}{2}+\frac{q_i}{8})^2+(y-\frac{L_y}{2}+\frac{q_i}{2})^2)}{2R^2}}\right\}
\end{equation}
where $L_x = 1, L_y = 2, q_1 = 1, q_2 = -1, R =0.25, r_0 = 3$. Set the initial fluid velocity and pressure to zero.

For this test, we choose the parameter $\nu = 0.08, \mu = 0.01, \kappa = 0.5, \beta_1 = 0.01, \beta_2 = -0.01$ similar to \cite{LBM20}. For space discretization, we use $(\mathbb{P}_2,\mathbb{P}_1,\mathbb{P}_2)$ pair to approximate the discrete spaces $(\bm{V}_h\times M_h\times X_h^0)$. We chose the mesh discretizing parameter $h=\frac{1}{64}$, time discretizing parameter $\Delta t = 0.01$ and final time $T=10.$  
We plot the net charge of ion concentrations $(c_1-c_2)$, electrostatic potential $(\phi)$ and streamlines of the fluid velocity at different time levels $t=0, 0.01, 0.25, 0.5, 0.75, 1.0, 2.5, 5.0, 10.0$.

From  Figures \ref{fig:c1-c2}-\ref{fig:psi}, we observe that initially the positive and negative ions are placed in the left button and the right top portion of the channel. Then, the system evolves with time due to the opposite attractions between ions and the presence of an electric field at the bottom. The negatively charged ions migrate downward from the right side of the boundary, while the positively charged ions move upward from the left side of the boundary. These create a local charge imbalance due to redistribution and generate an electric body force that drives fluid motion through the channel and forms vortices.

Over time, the system gradually approaches an equilibrium state in which negative ions accumulate near the bottom, positively charged ions are uniformly distributed near the grounded top, and fluid motion becomes stable as the system's total energy reaches its minimum.

We also observed for this example that the mass of the ions (both positive and negative) is conserved, as shown in Figure \ref{fig:massion}. In Figure \ref{fig:minmaxion}, we present minimum and maximum values of the ion concentration, the electric potential energy and the total energy at each time level, and we observe that the minimum and maximum values remain positive, and the energy initially increases up to a peak, and then it tends downwards. 

\raggedbottom

\begin{figure}[!h]
	\centering
	\begin{tikzpicture}[scale=2.5]
		
		\usetikzlibrary{arrows.meta}
		
		% Domain dimensions
		\def\Lx{1}
		\def\Ly{2}
		
		% Draw rectangular domain
		\draw[thick] (0,0) rectangle (\Lx,\Ly);
		
		% Top electrode
		\draw[thick] (0,\Ly) -- (\Lx,\Ly);
		\draw[thick] (0.5,\Ly) -- (0.5,\Ly+0.15);
		\draw[thick] (0.4,\Ly+0.15) -- (0.6,\Ly+0.15);
		\node at (0.5,\Ly+0.28) {$\phi=0$};
		
		% Dimension arrows
		\draw[<->] (-0.15,0) -- (-0.15,\Ly);
		\node[left] at (-0.15,\Ly/2) {$L_y=2$};
		
		\draw[<->] (0,-0.15) -- (\Lx,-0.15);
		\node[below] at (\Lx/2,-0.15) {$L_x=1$};
		
		% Bottom surface charge line
		\draw[red,thick] (0,0) -- (\Lx,0);
		\node[below] at (0.5,0.17) {Surf. charge $\sigma$};
		
		% Positive surface charges
		\foreach \x in {0.1,0.3,0.5,0.7,0.9}
		\node at (\x,-0.06) {$+$};
		
		%%%%%%%%%%%%%%%%%%%%%%%%%%%%%%%%%%%%%%%%
		% POSITIVE ION (from Gaussian center)
		% Center = (0.375 , 0.5)
		%%%%%%%%%%%%%%%%%%%%%%%%%%%%%%%%%%%%%%%%
		
		\filldraw[red!60] (0.375,0.5) circle (0.1);
		\node at (0.375,0.5) {$+$};
		
		\filldraw[red!40, even odd rule]
		(0.375,0.5) circle (0.17)
		(0.375,0.5) circle (0.1);
		
		\filldraw[red!20, even odd rule]
		(0.375,0.5) circle (0.25)
		(0.375,0.5) circle (0.17);
		
		% Arrow
		\draw[-{Latex}] (1,0.5) -- (0.625,0.5);
		
		\node[right] at (1,0.5)
		{positive ions $c_{+}$};
		
		%%%%%%%%%%%%%%%%%%%%%%%%%%%%%%%%%%%%%%%%
		% NEGATIVE ION (mirror Gaussian center)
		% Center = (0.625 , 1.5)
		%%%%%%%%%%%%%%%%%%%%%%%%%%%%%%%%%%%%%%%%
		
		\filldraw[blue!60] (0.625,1.5) circle (0.1);
		\node at (0.625,1.5) {$-$};
		
		\filldraw[blue!40, even odd rule]
		(0.625,1.5) circle (0.17)
		(0.625,1.5) circle (0.1);
		
		\filldraw[blue!20, even odd rule]
		(0.625,1.5) circle (0.25)
		(0.625,1.5) circle (0.17);
		
		% Arrow
		\draw[-{Latex}] (1,1.5) -- (0.875,1.5);
		
		\node[right] at (1,1.5)
		{negative ions $c_{-}$};
		
		% Radius indication
		\node[above] at (0.780,1.75) {$R=0.25$};
		
	\end{tikzpicture}
	\caption{Domain for ion spreading. Set up of the geometry, boundary conditions	for the electrostatic potential, and the initial distribution of	positively and negatively charged ion concentrations}
	\label{fig:geo1}
\end{figure}
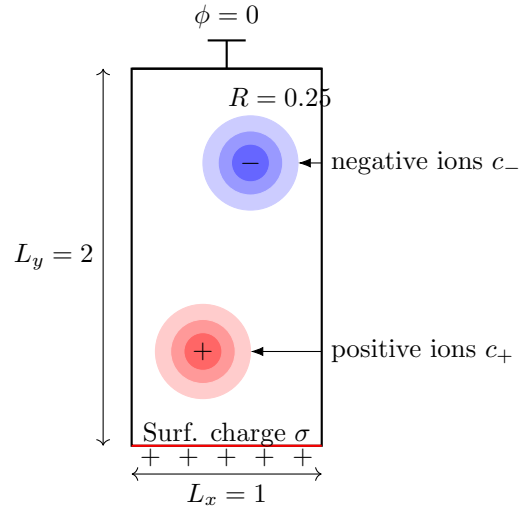
%
% c1-c2
\begin{figure}[!h]
	\centering
	\includegraphics[scale=0.24]{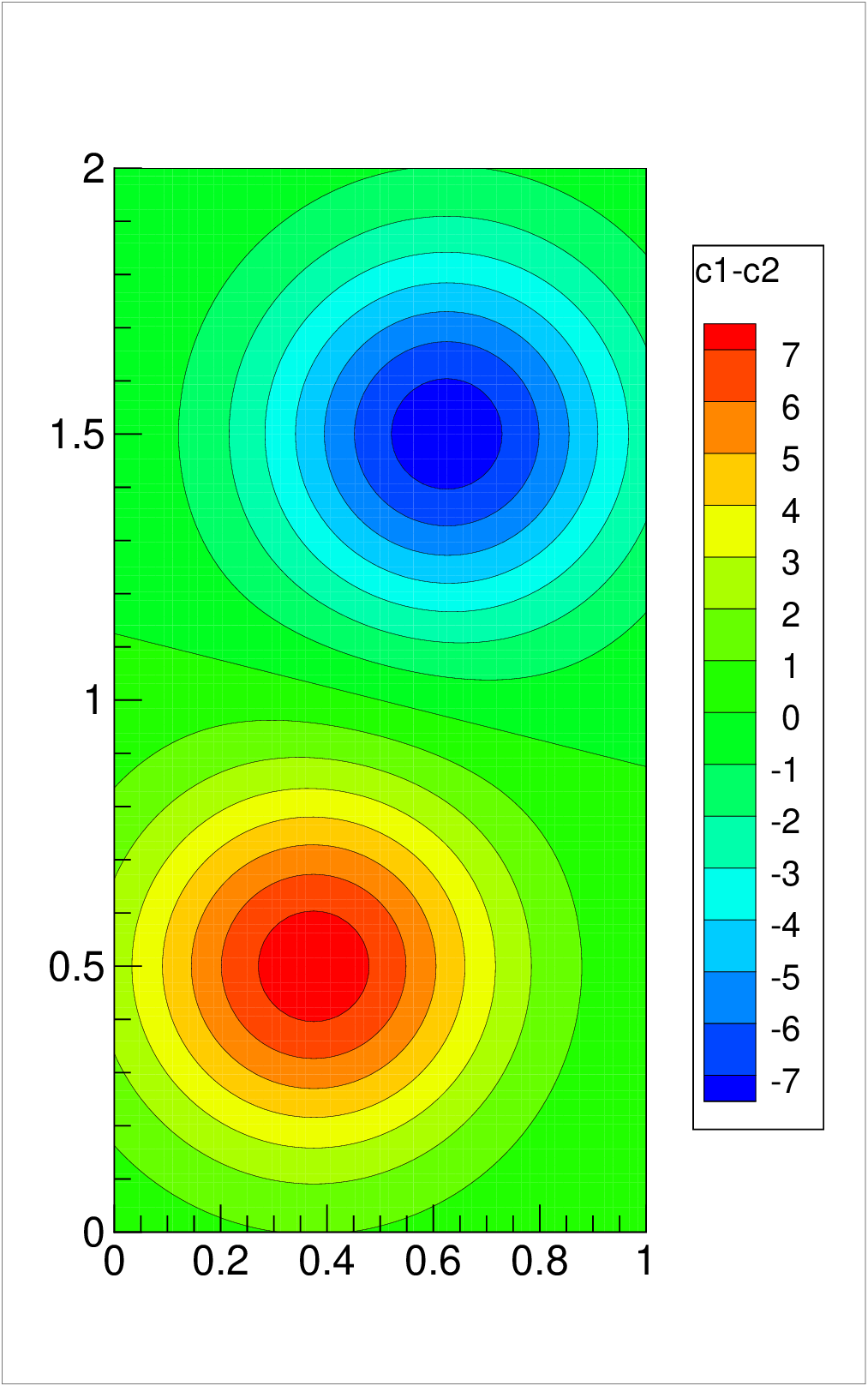}
	\includegraphics[scale=0.24]{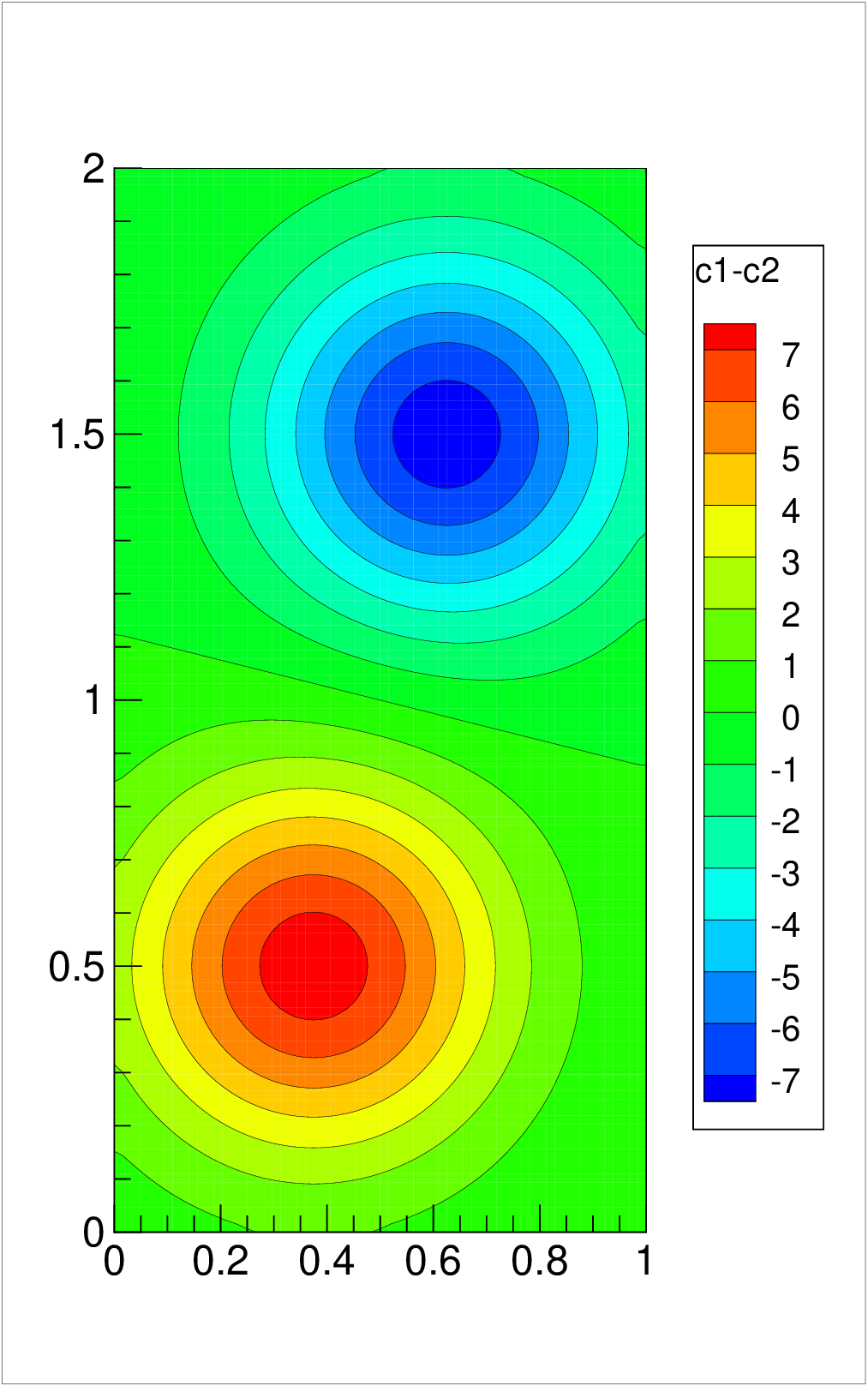}
	\includegraphics[scale=0.24]{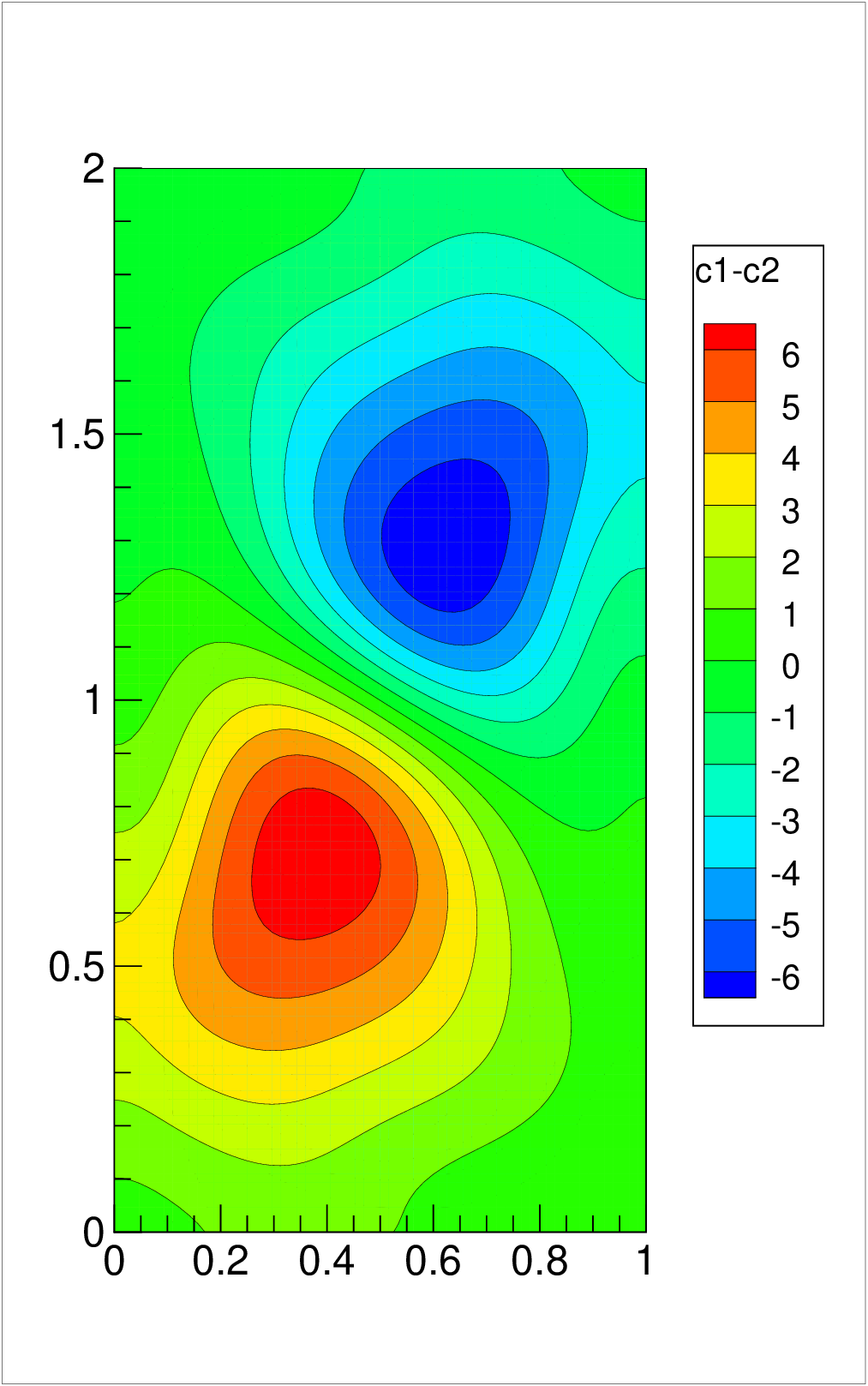}
	\includegraphics[scale=0.24]{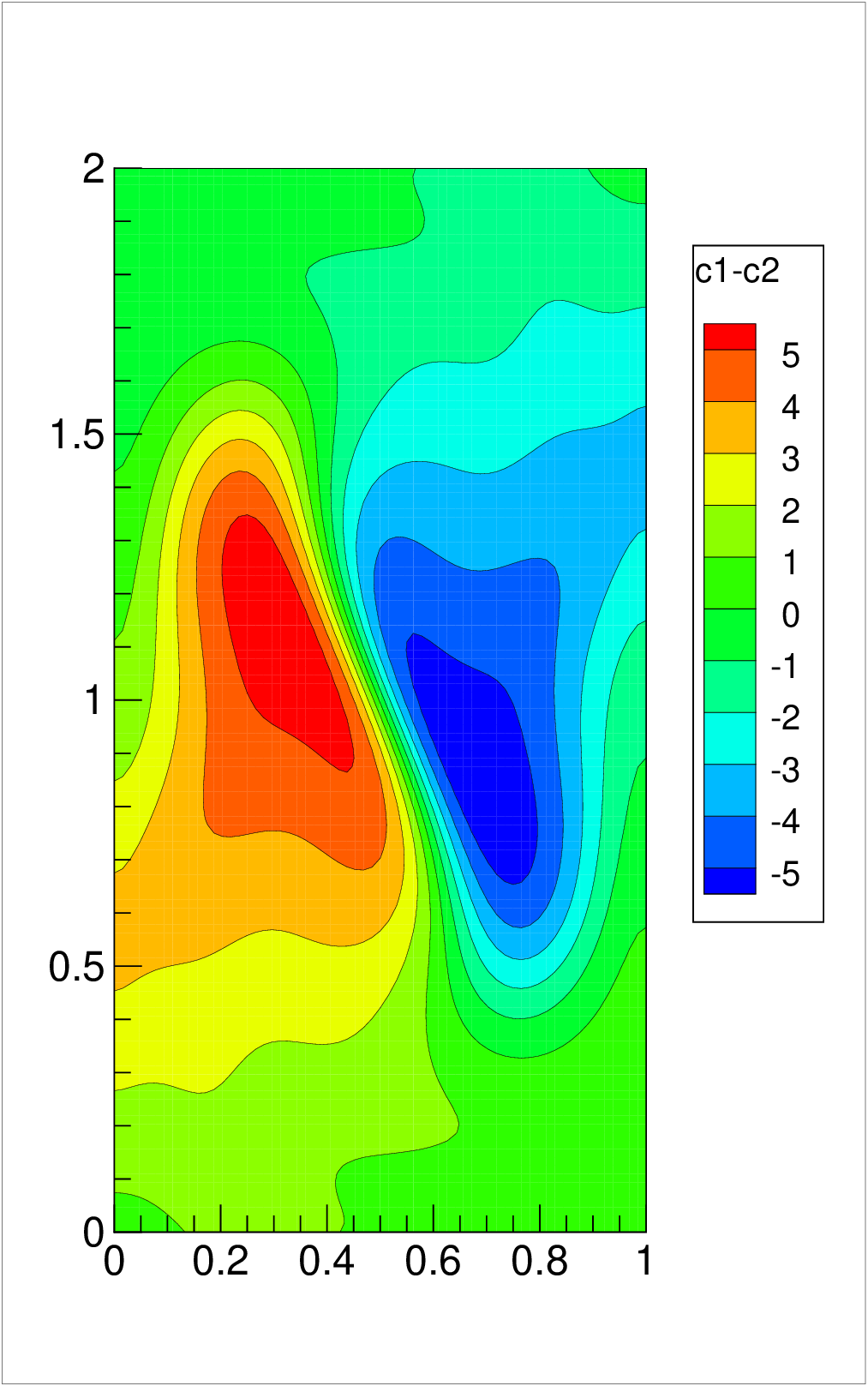}
	\includegraphics[scale=0.24]{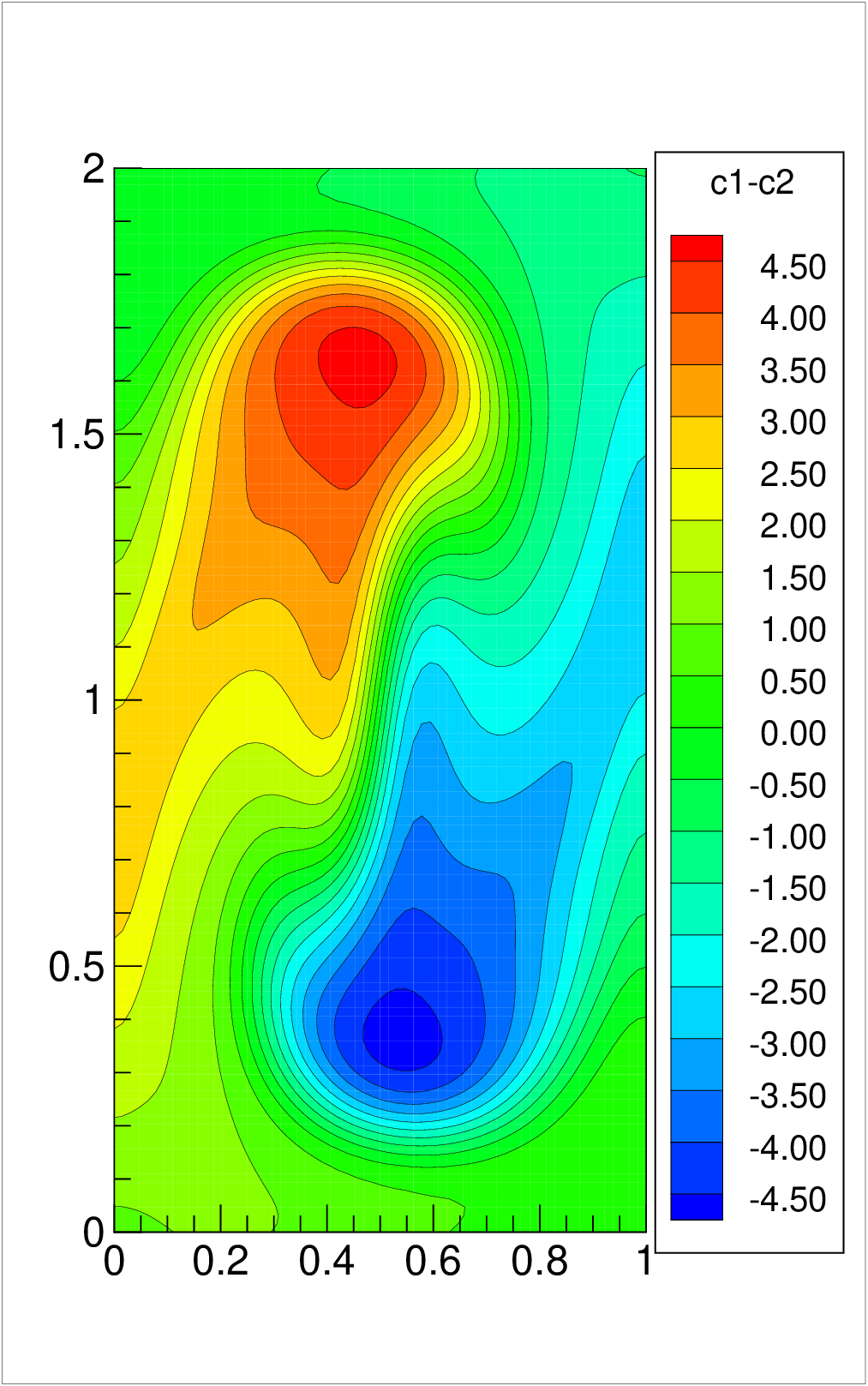}
	\includegraphics[scale=0.24]{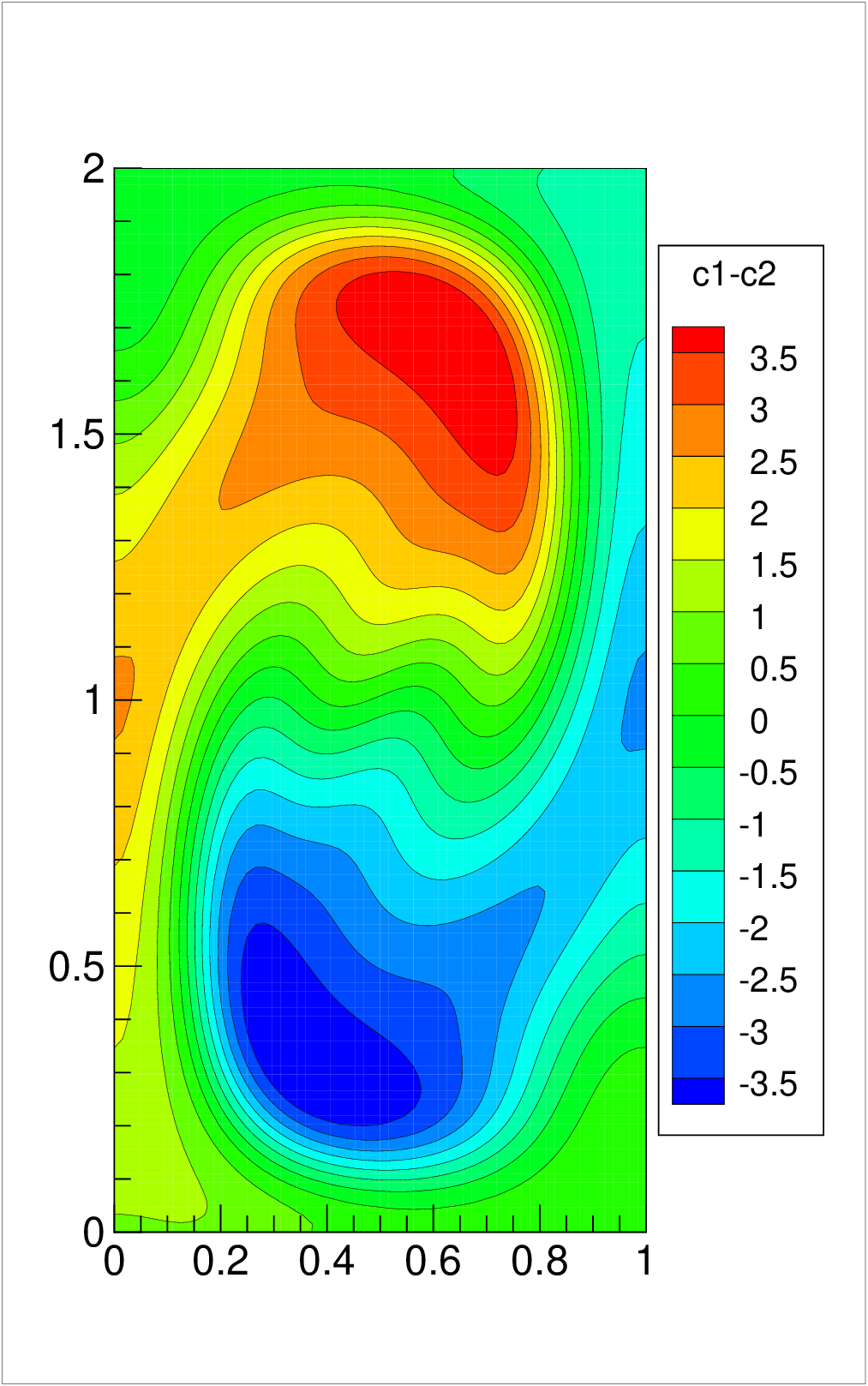}
	\includegraphics[scale=0.24]{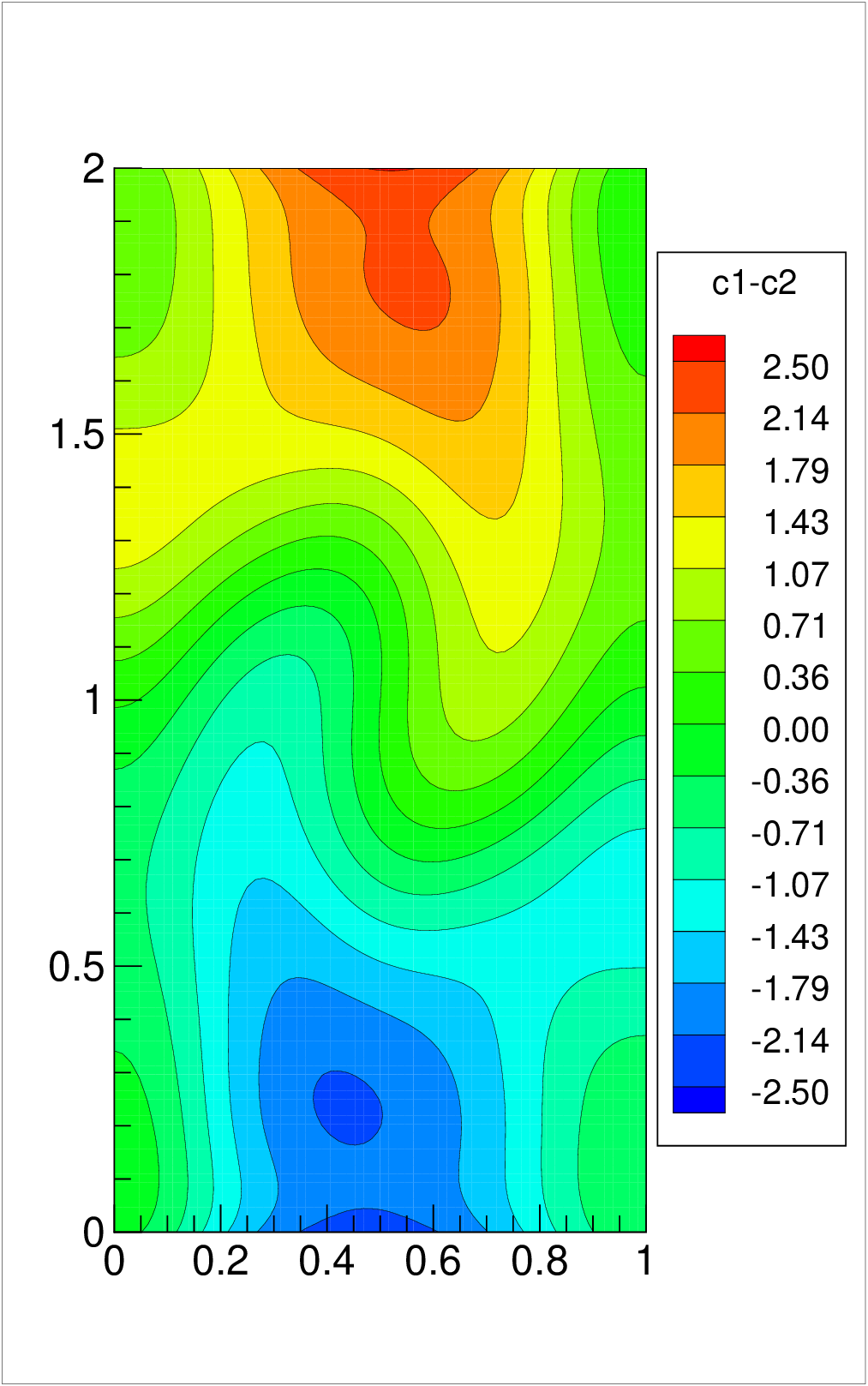}
	\includegraphics[scale=0.24]{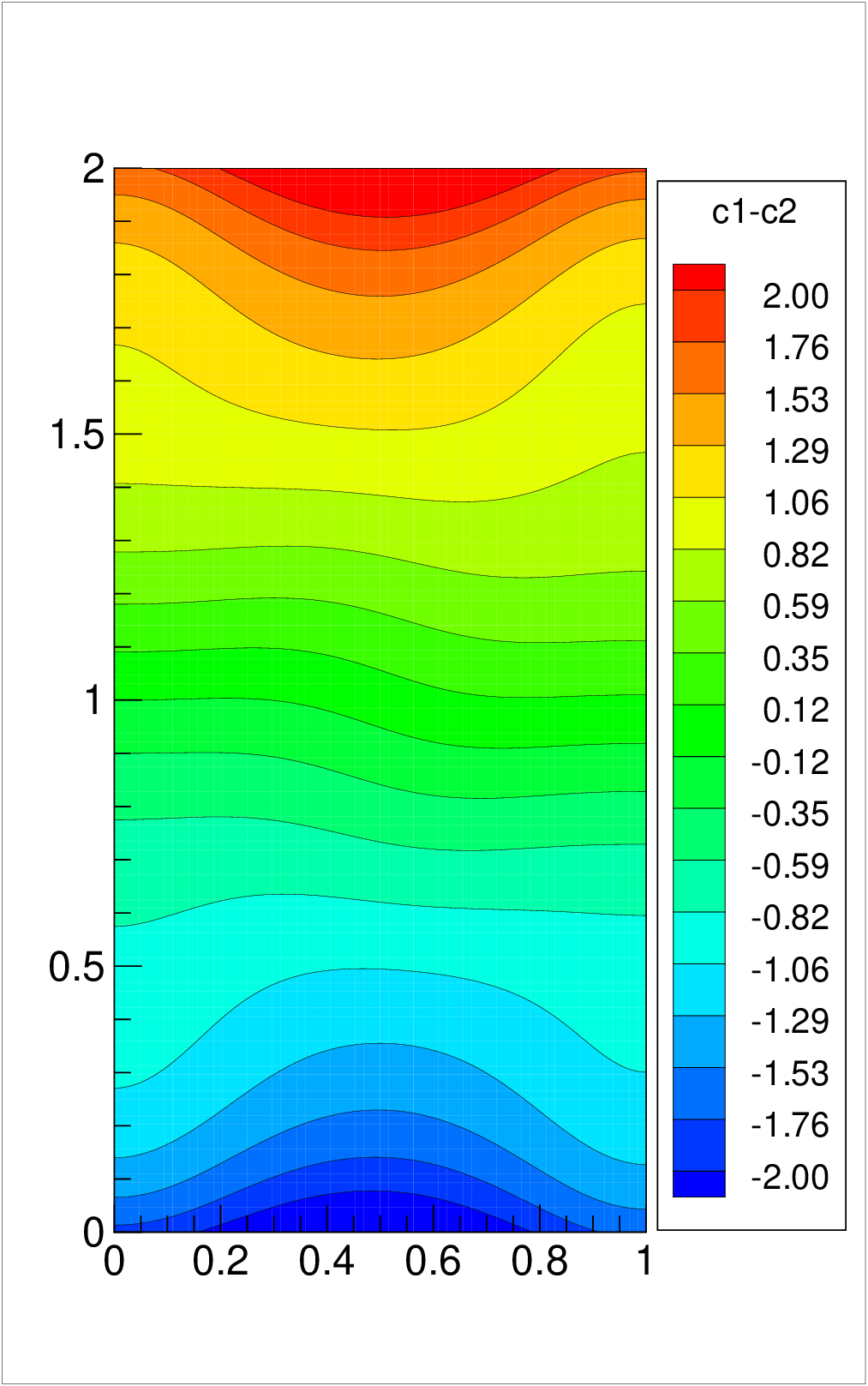}
	\includegraphics[scale=0.24]{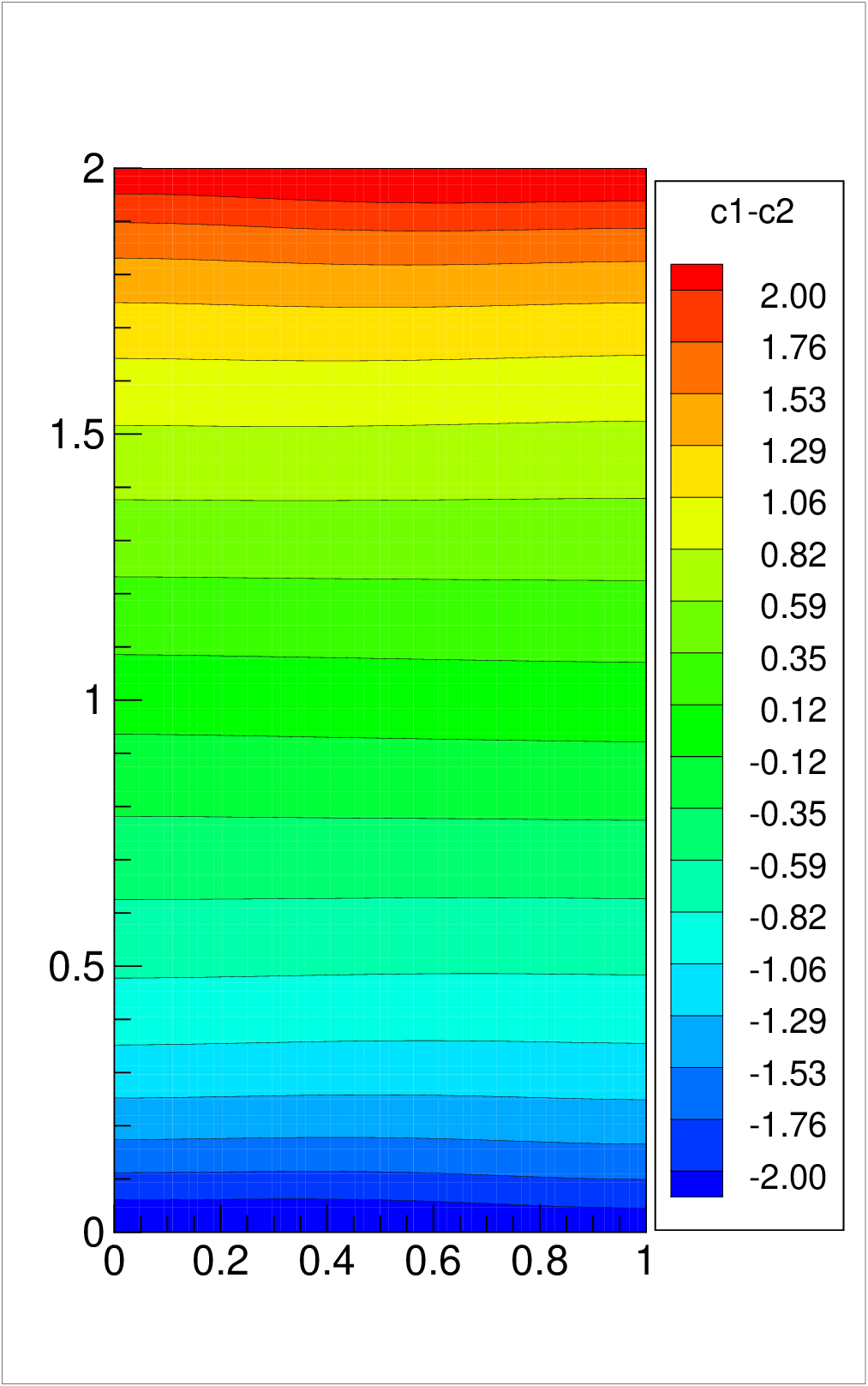}
	\caption{The evolution of net charge of ion concentrations ($c_1-c_2$) in different time levels.}		
	\label{fig:c1-c2}
\end{figure}

% phi
\begin{figure}[!h]
	\centering
	\includegraphics[scale=0.24]{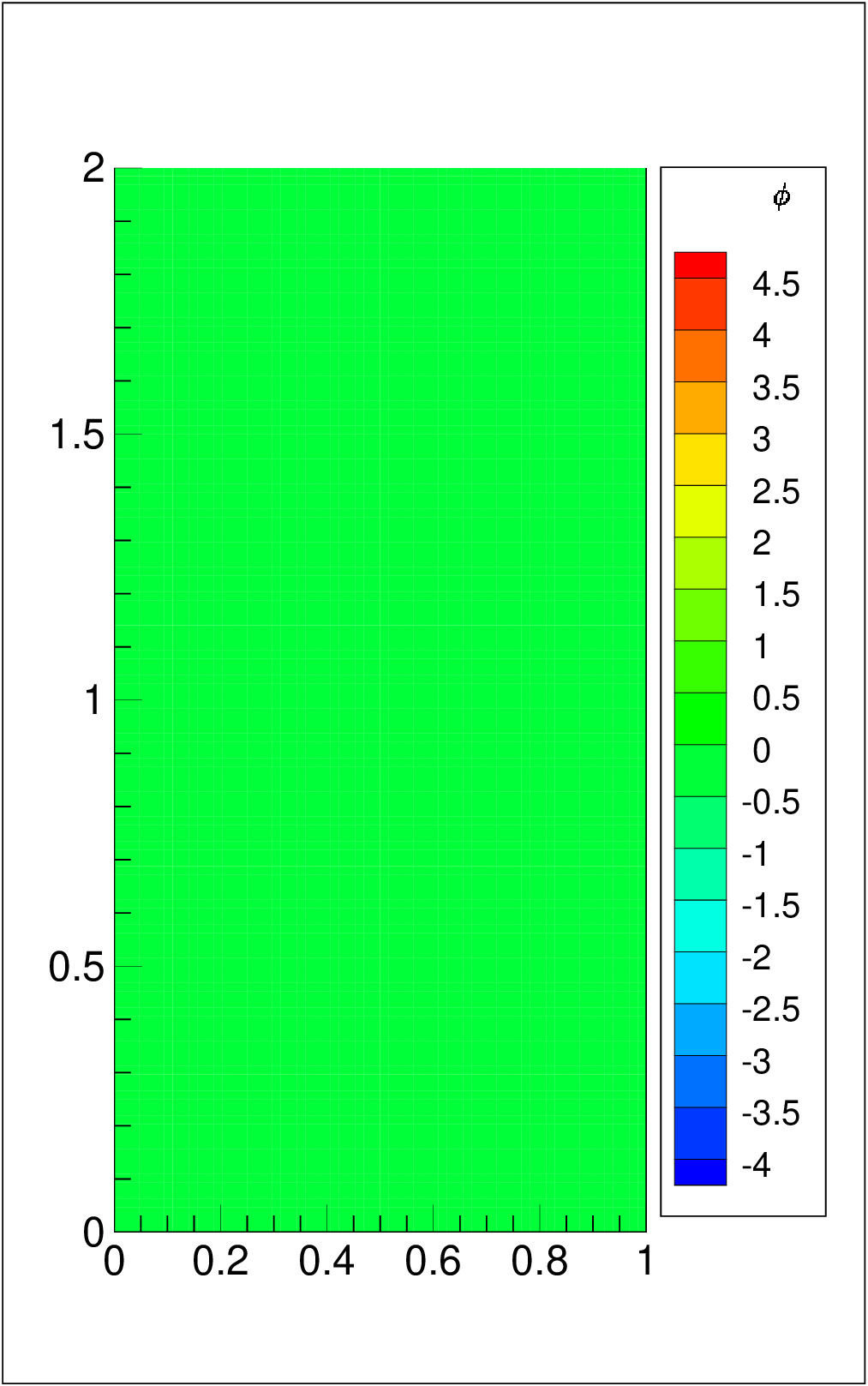}
	\includegraphics[scale=0.24]{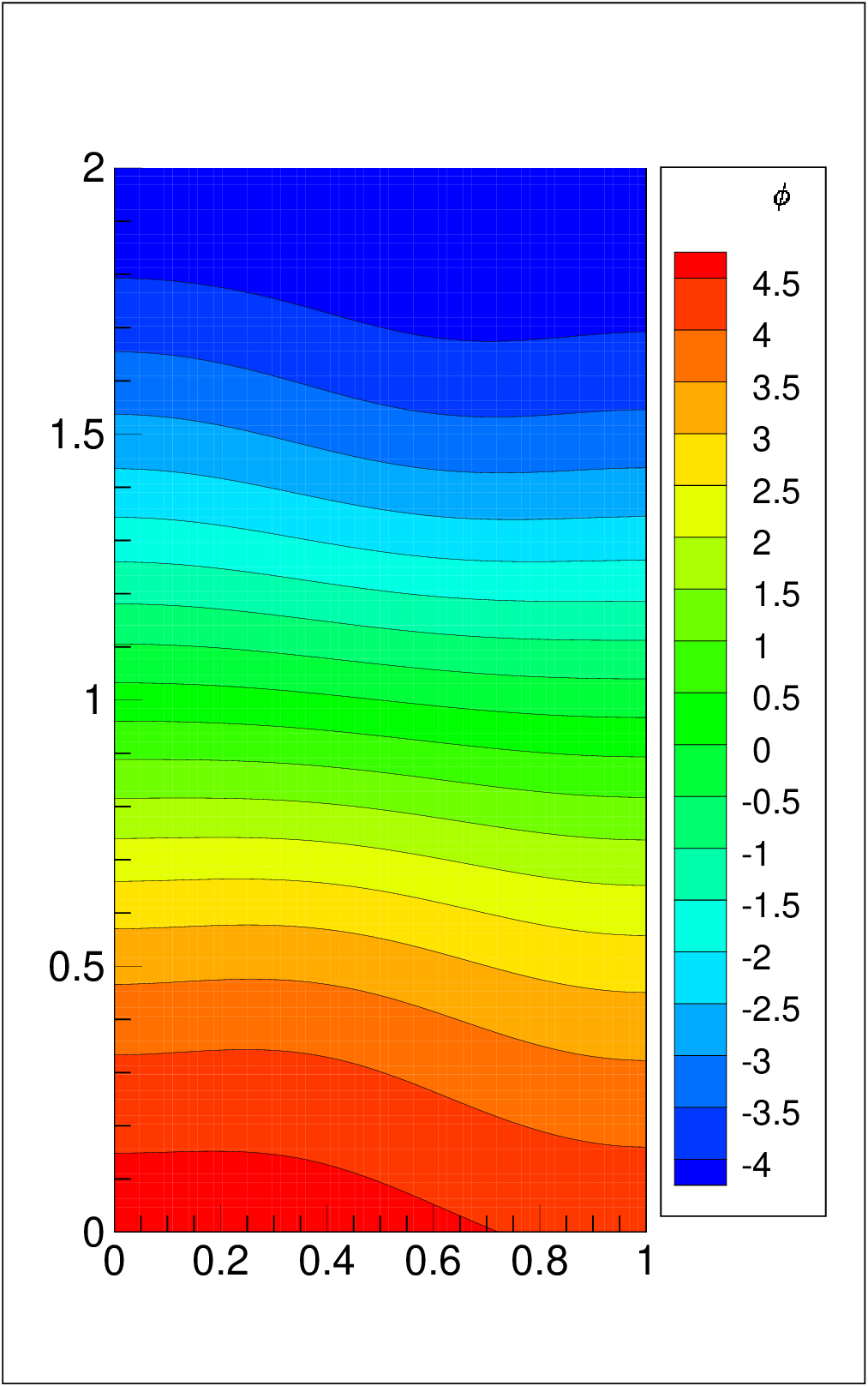}
	\includegraphics[scale=0.24]{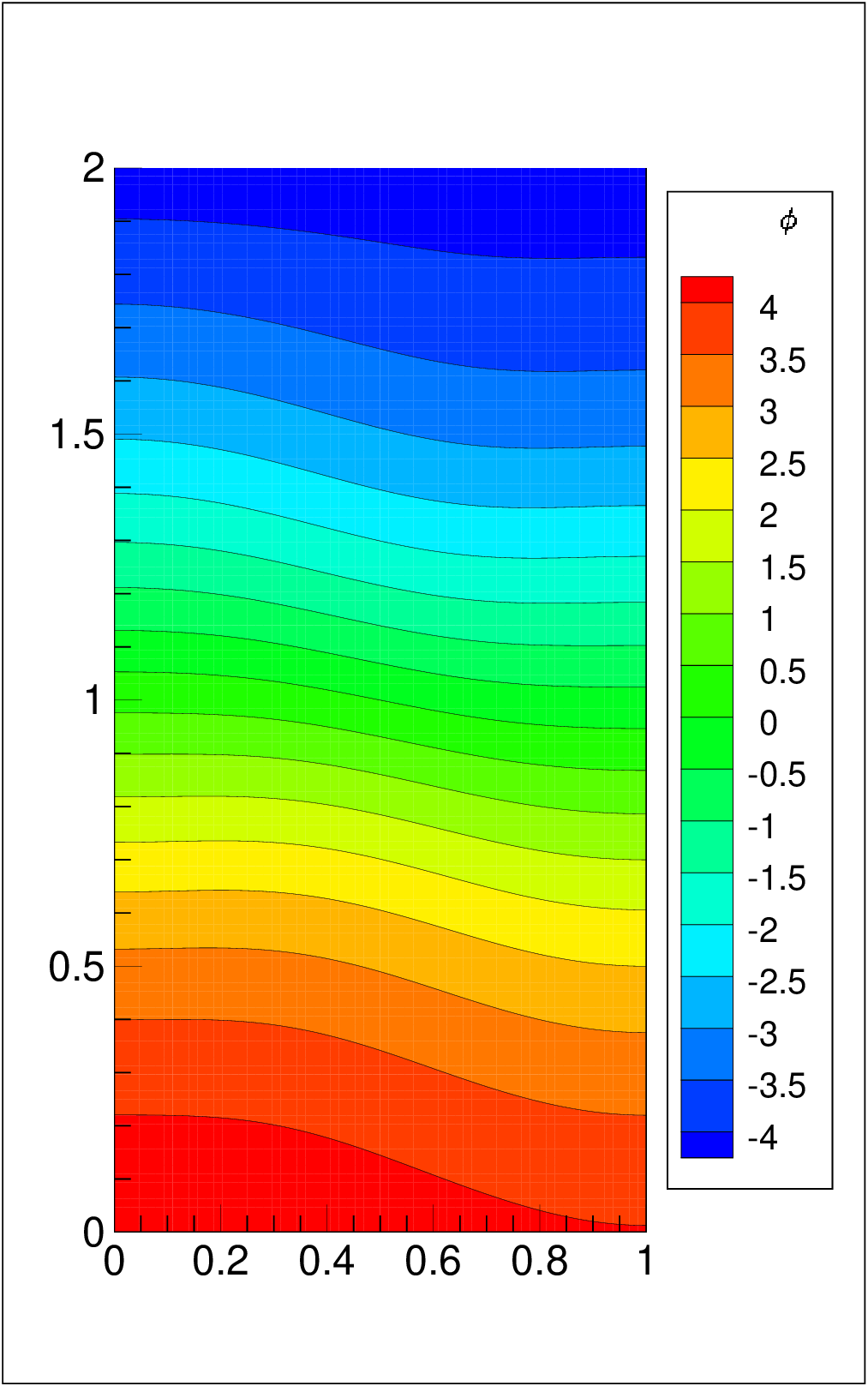}
	\includegraphics[scale=0.24]{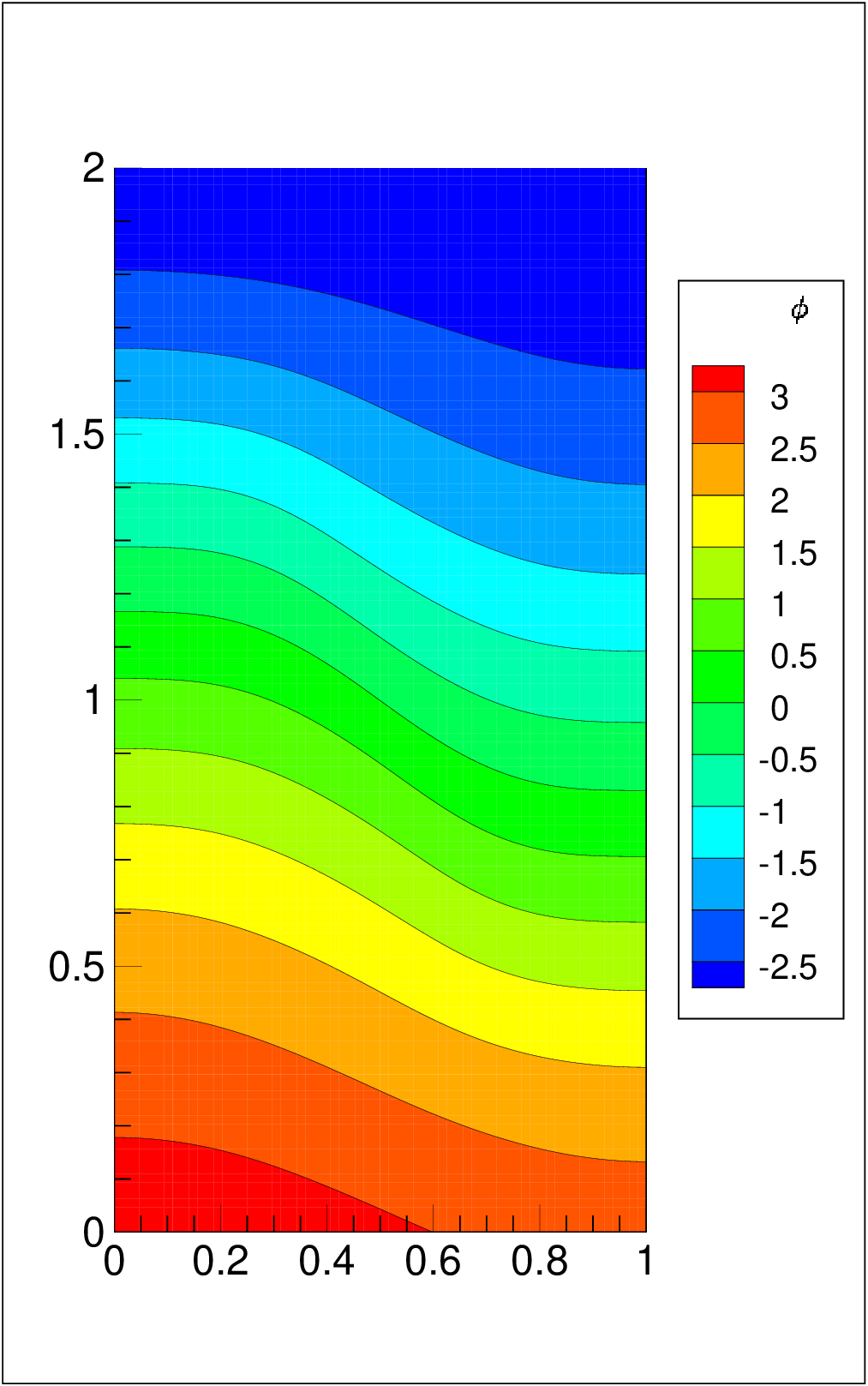}
	\includegraphics[scale=0.24]{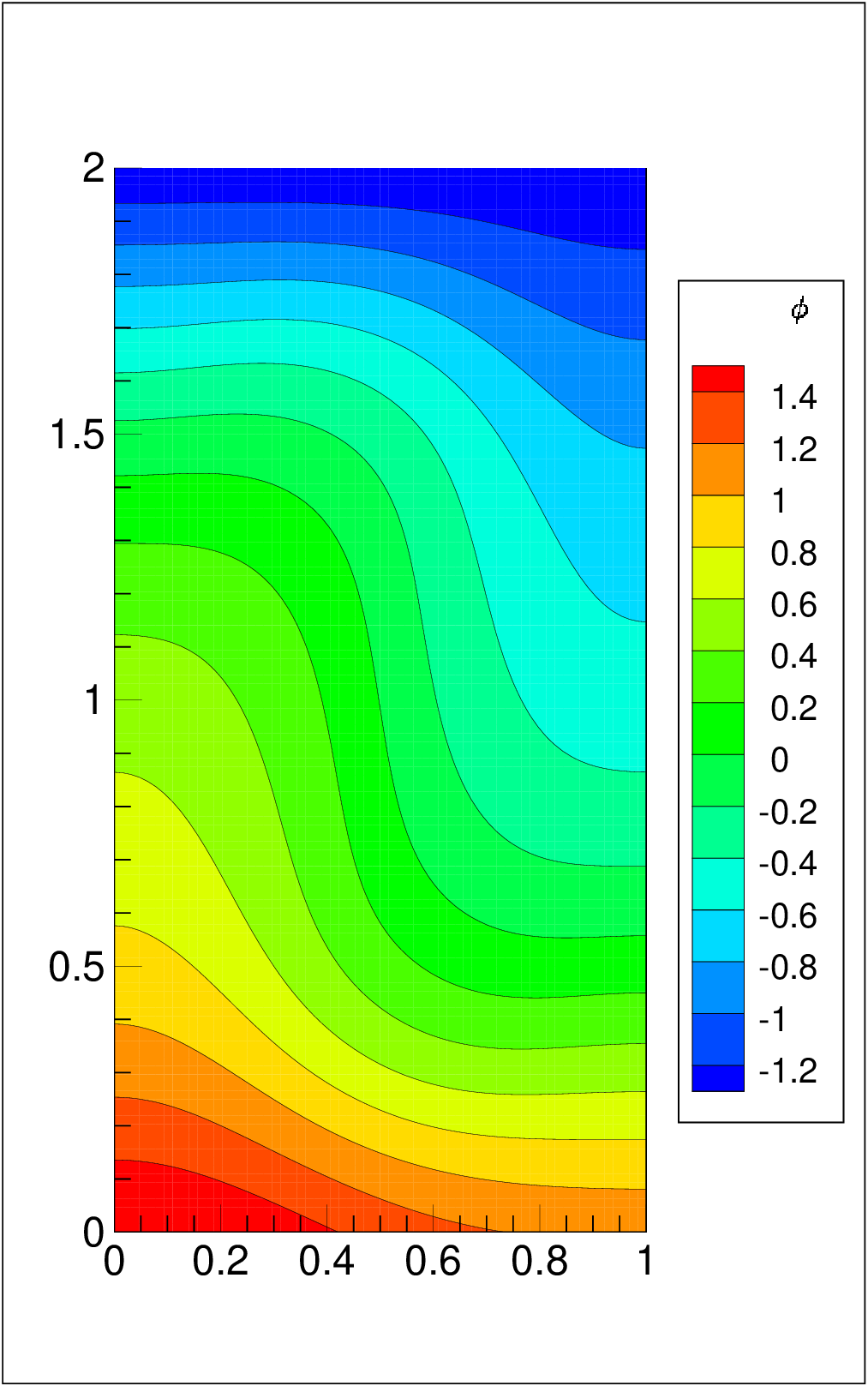}
	\includegraphics[scale=0.24]{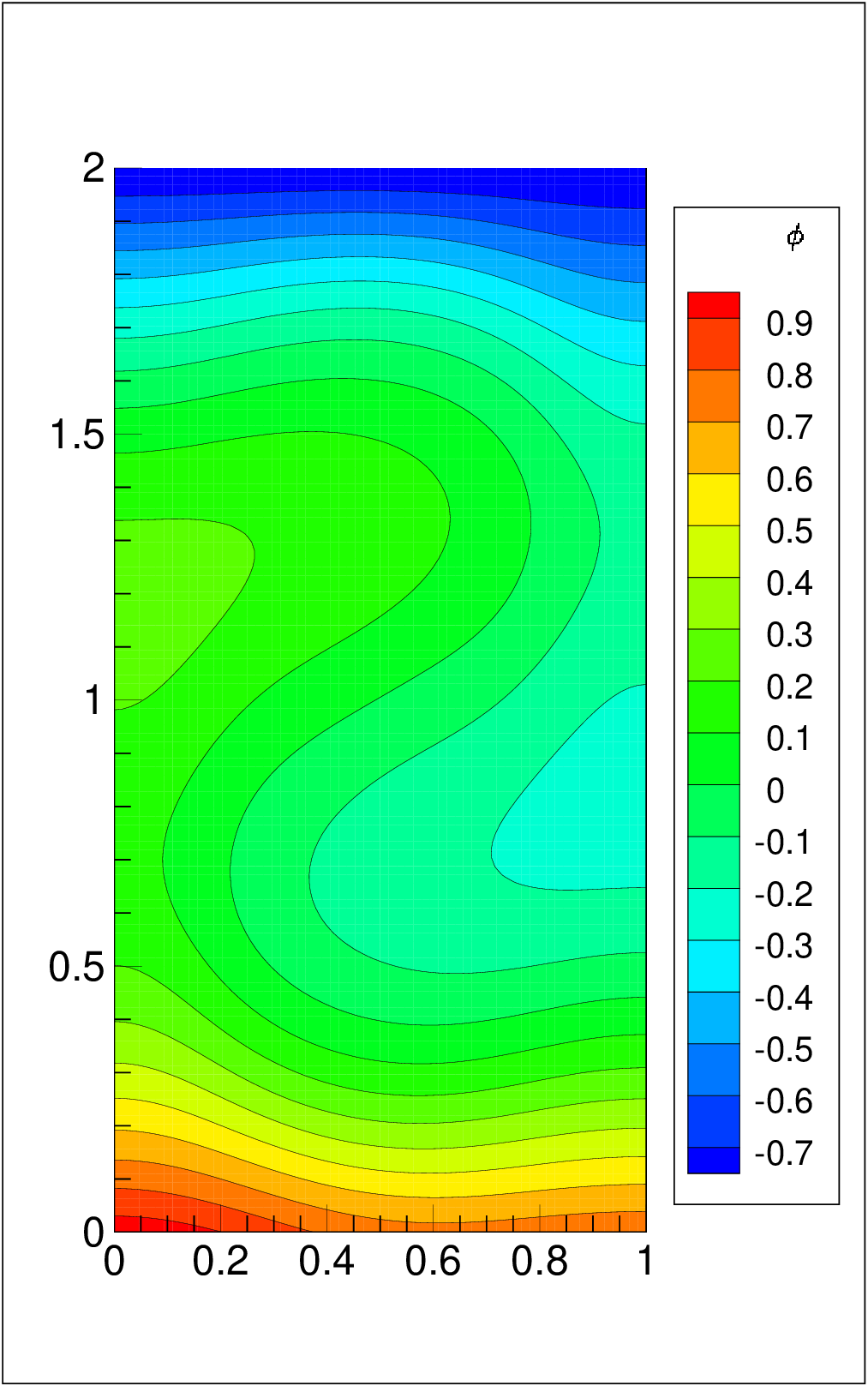}
	\includegraphics[scale=0.24]{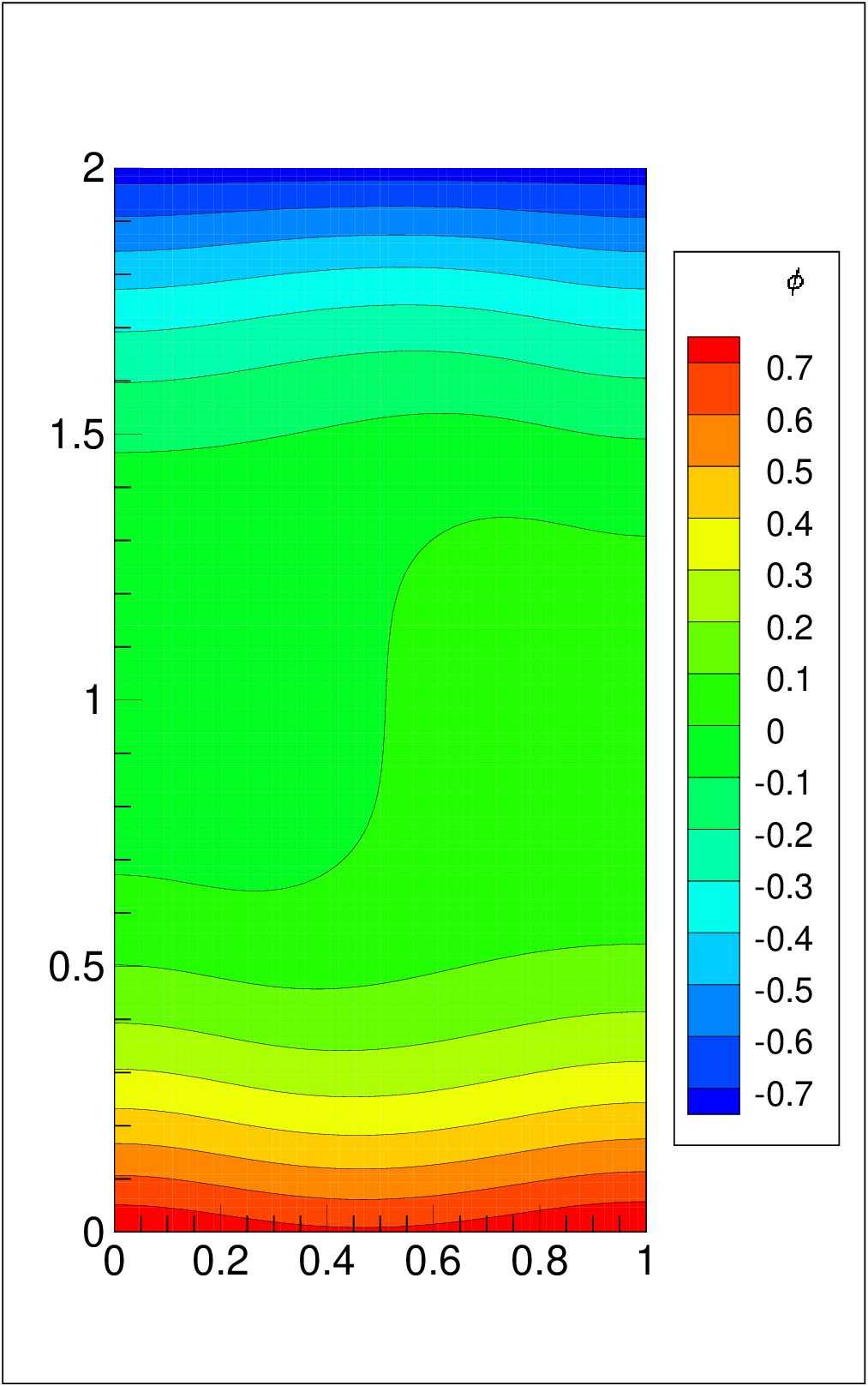}
	\includegraphics[scale=0.24]{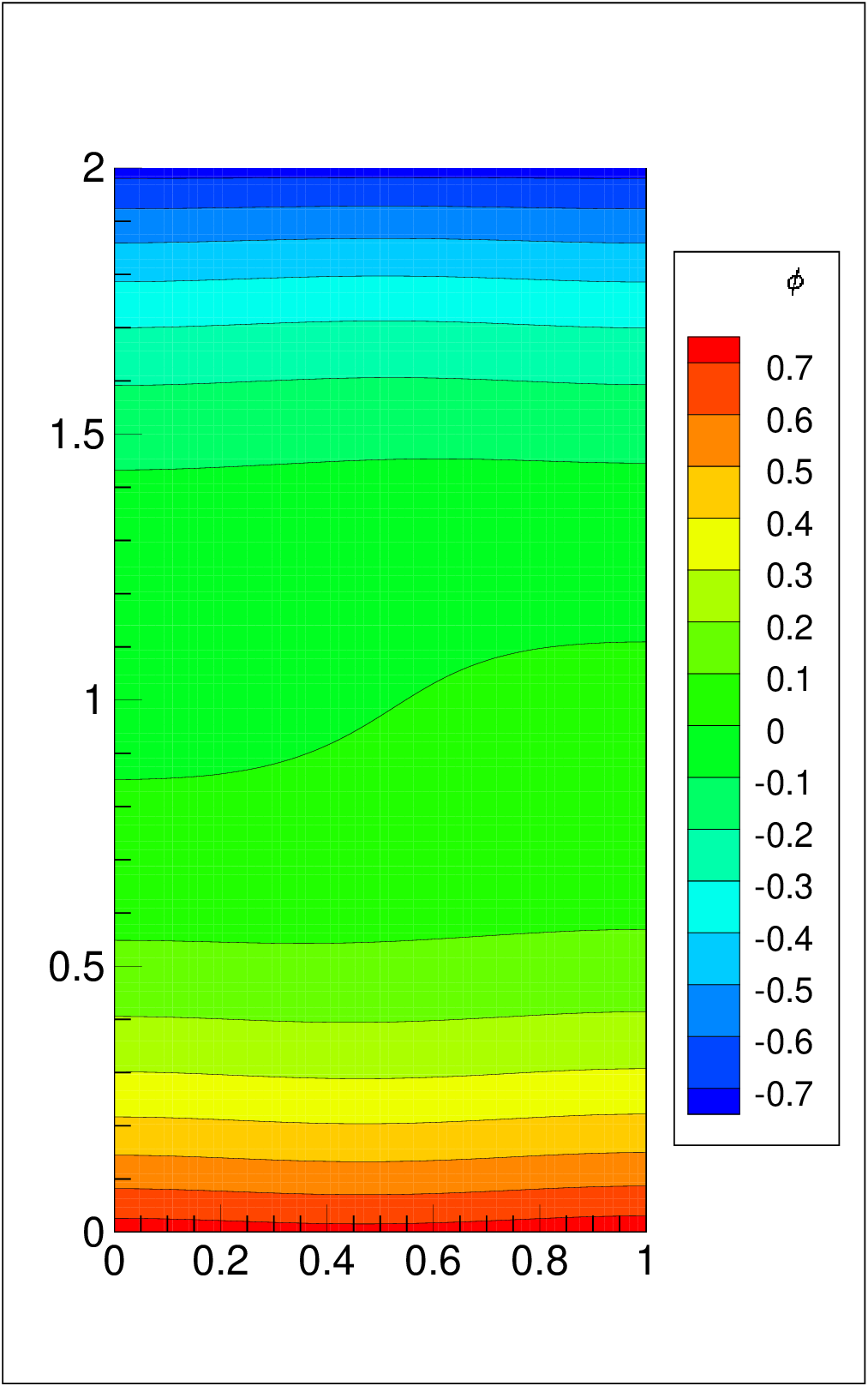}
	\includegraphics[scale=0.24]{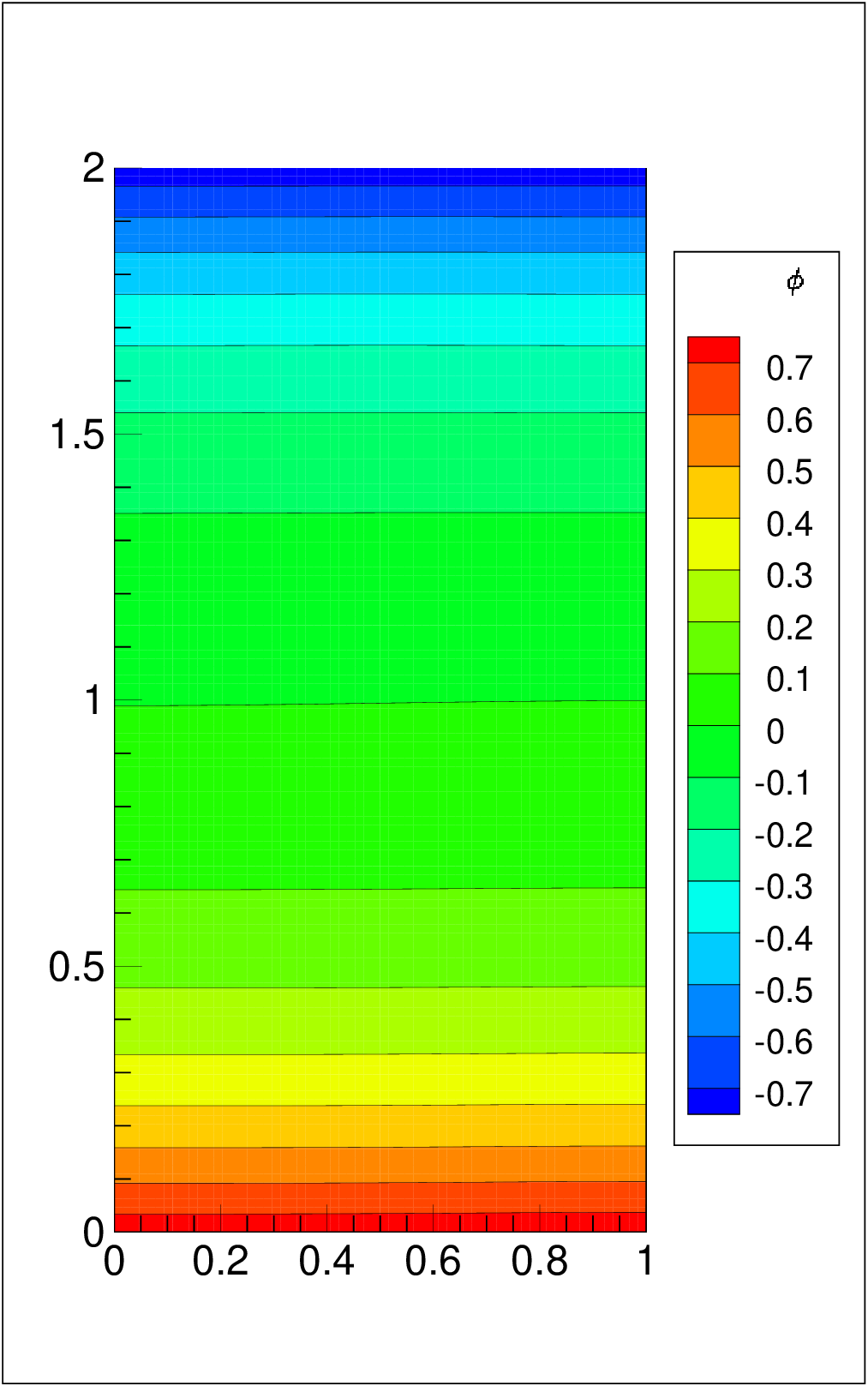}
	\caption{The evolution of electrostatic potential ($\phi$) in different time levels. }	
	\label{fig:phi}	
\end{figure}
% psi
\begin{figure}[!h]
	\centering
	\includegraphics[scale=0.24]{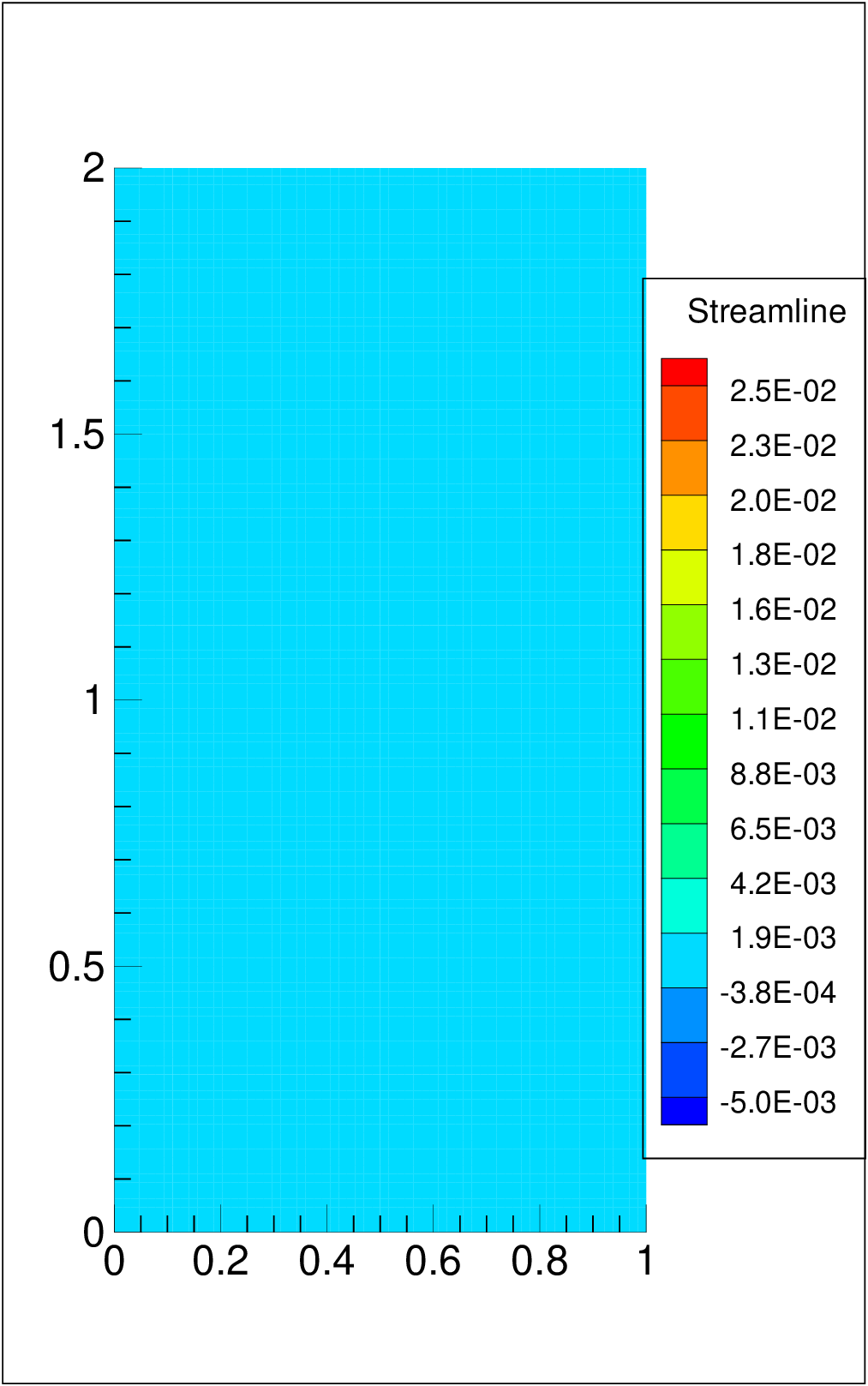}
	\includegraphics[scale=0.24]{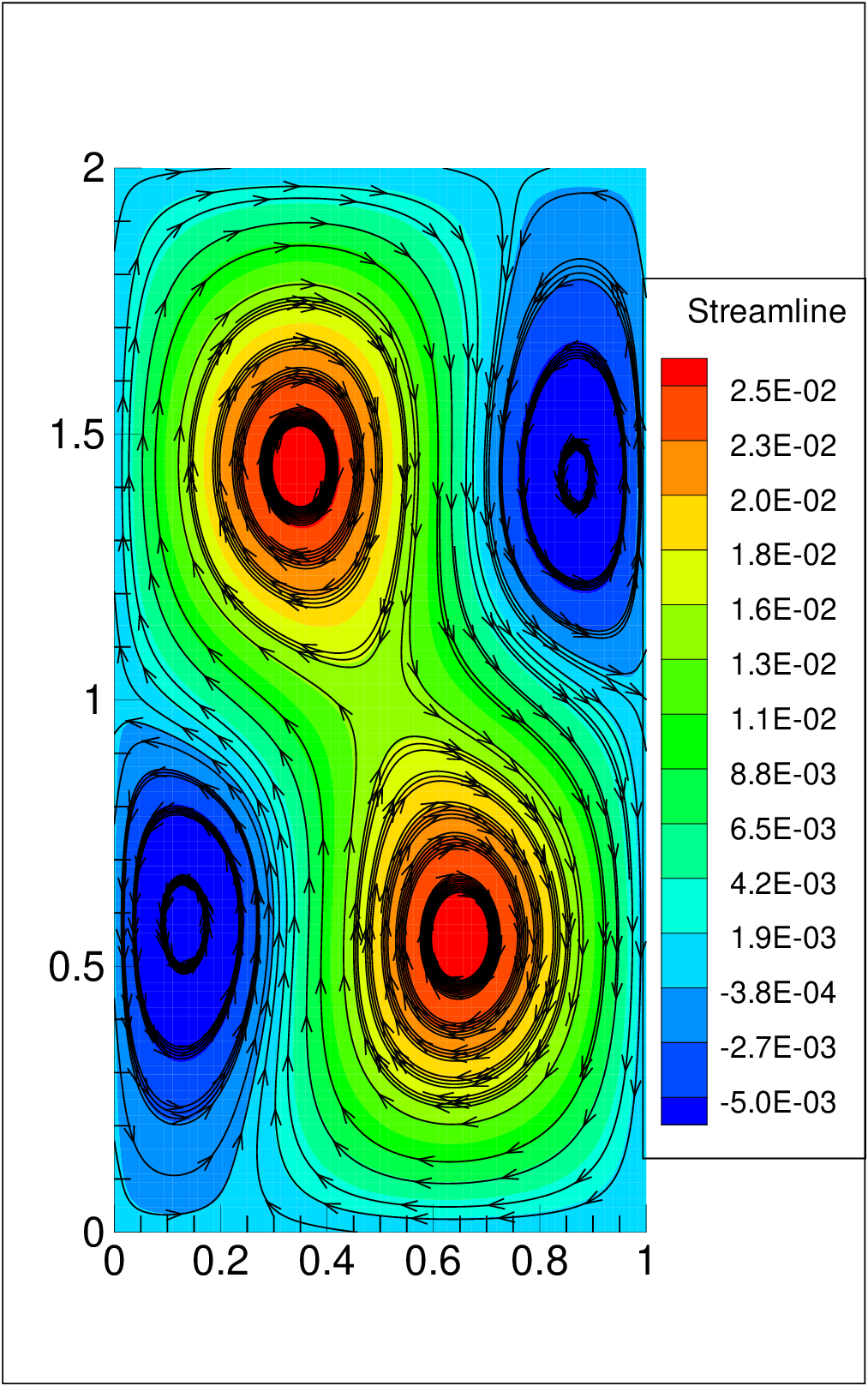}
	\includegraphics[scale=0.24]{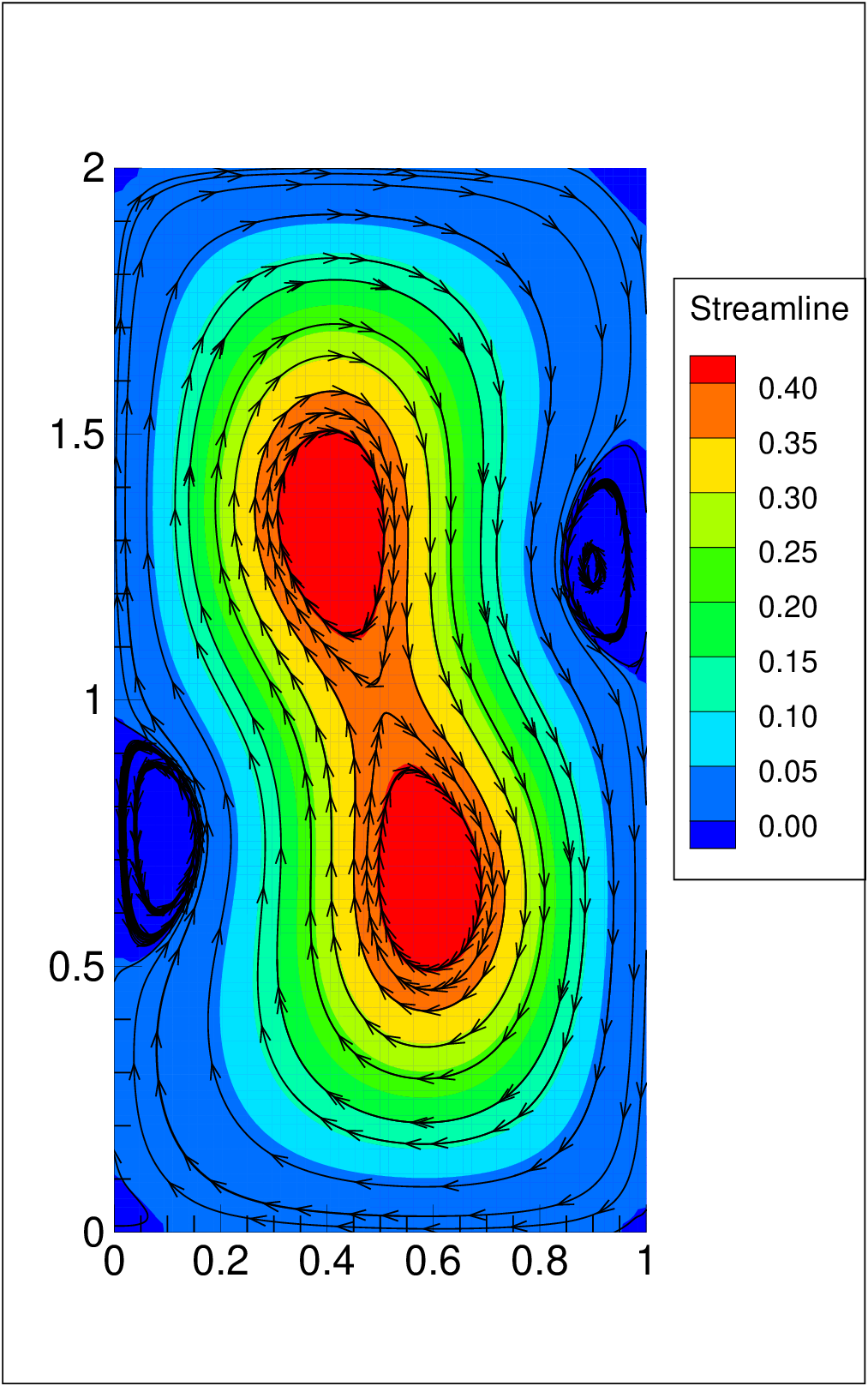}
	\includegraphics[scale=0.24]{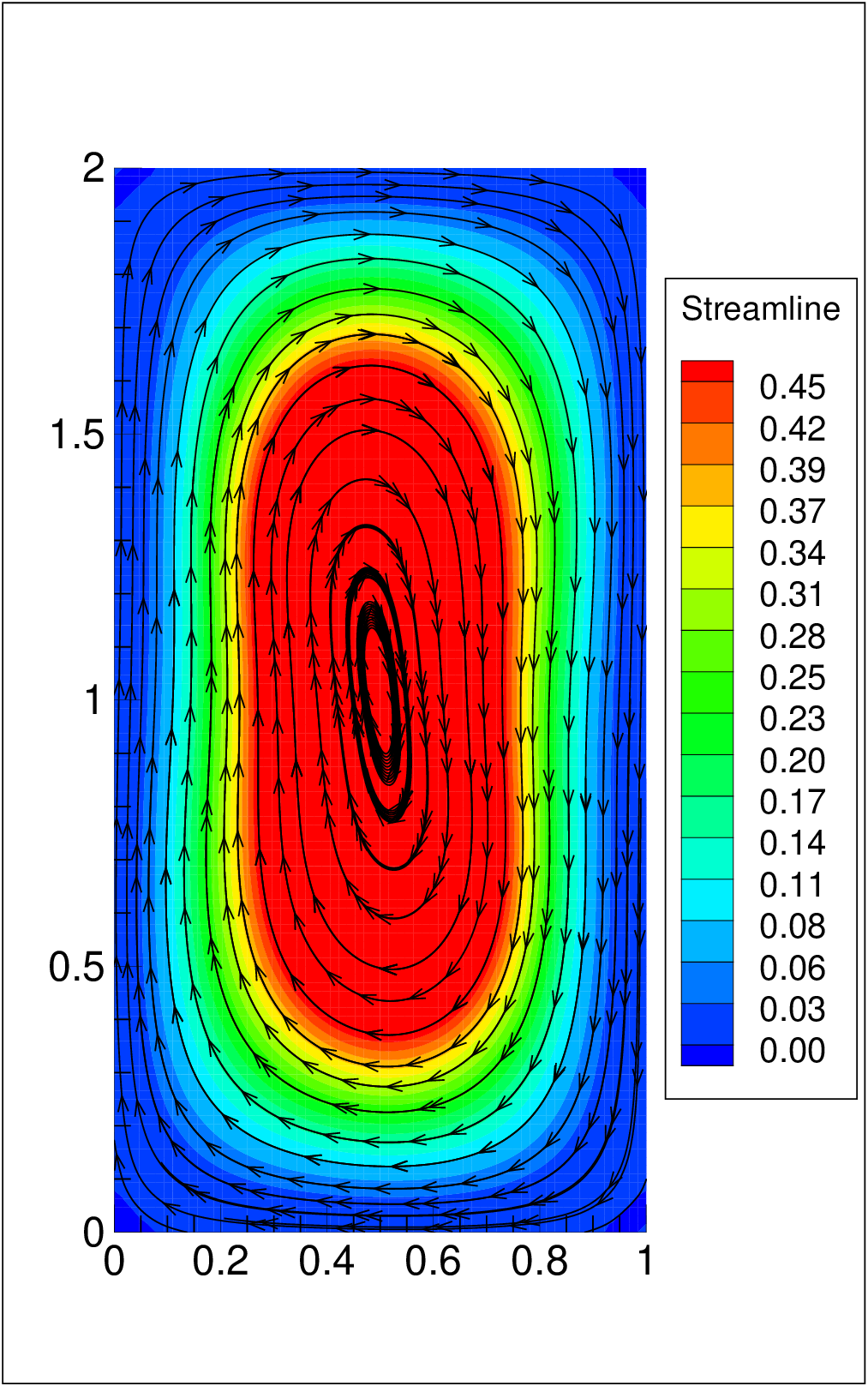}
	\includegraphics[scale=0.24]{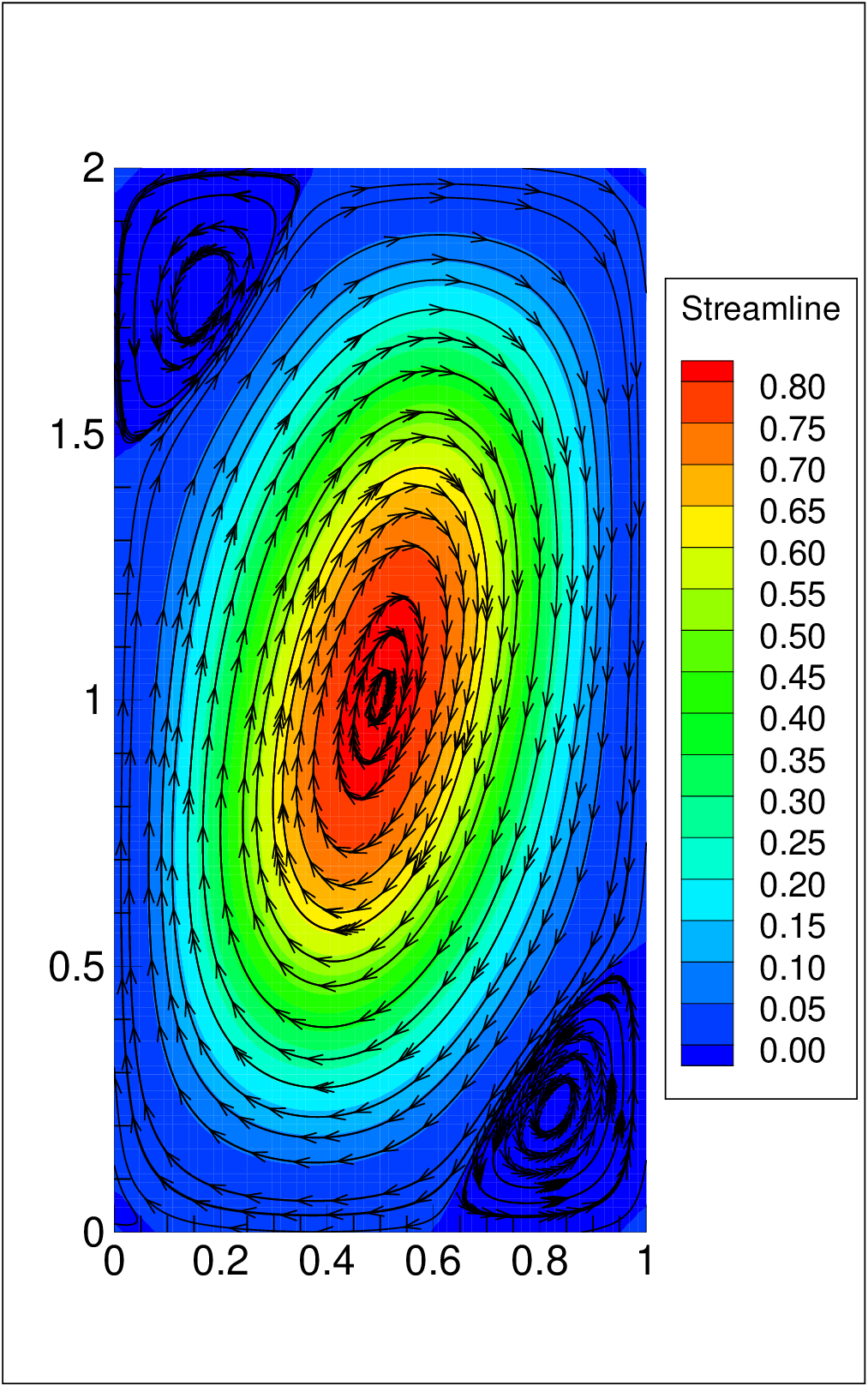}
	\includegraphics[scale=0.24]{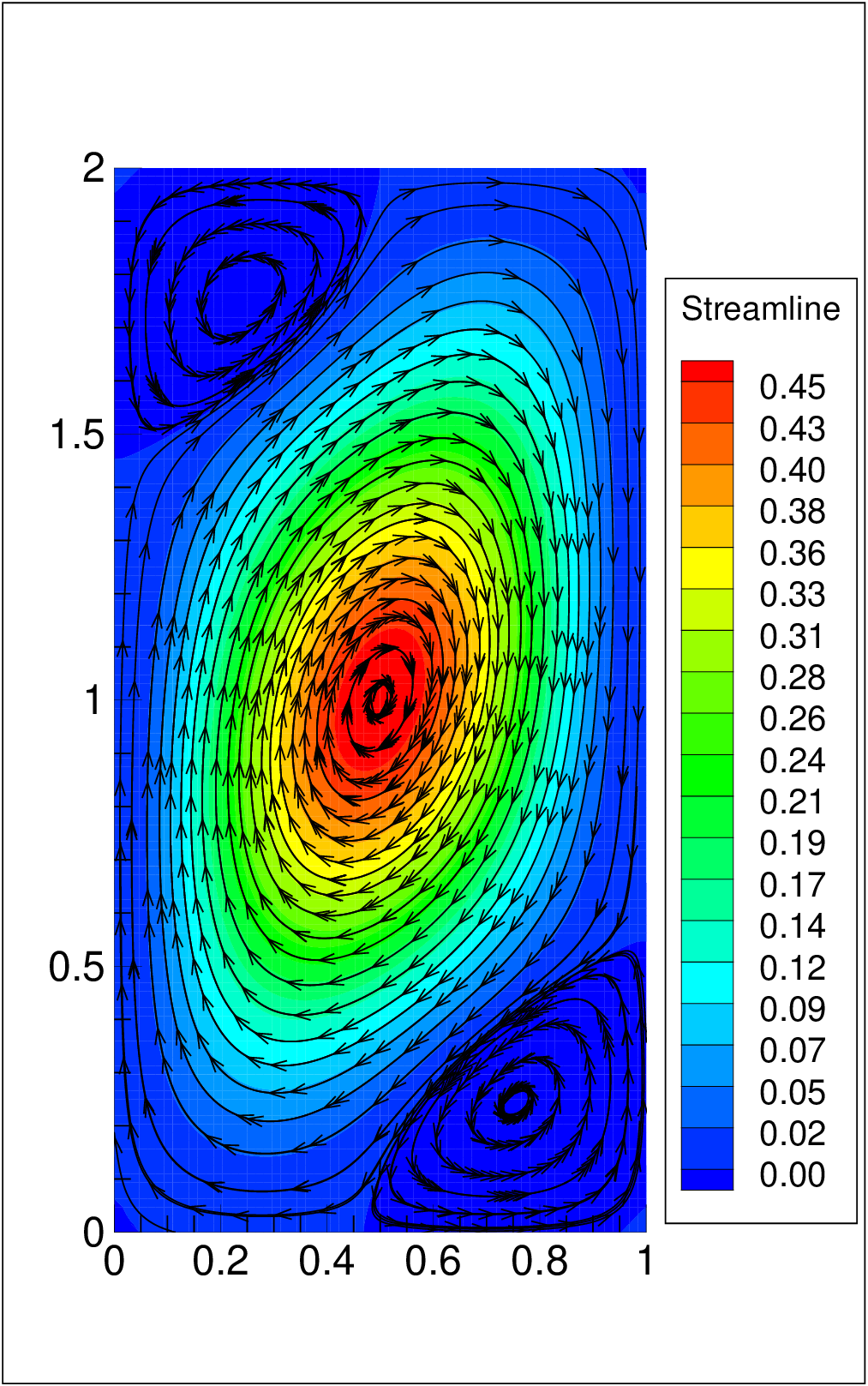}
	\includegraphics[scale=0.24]{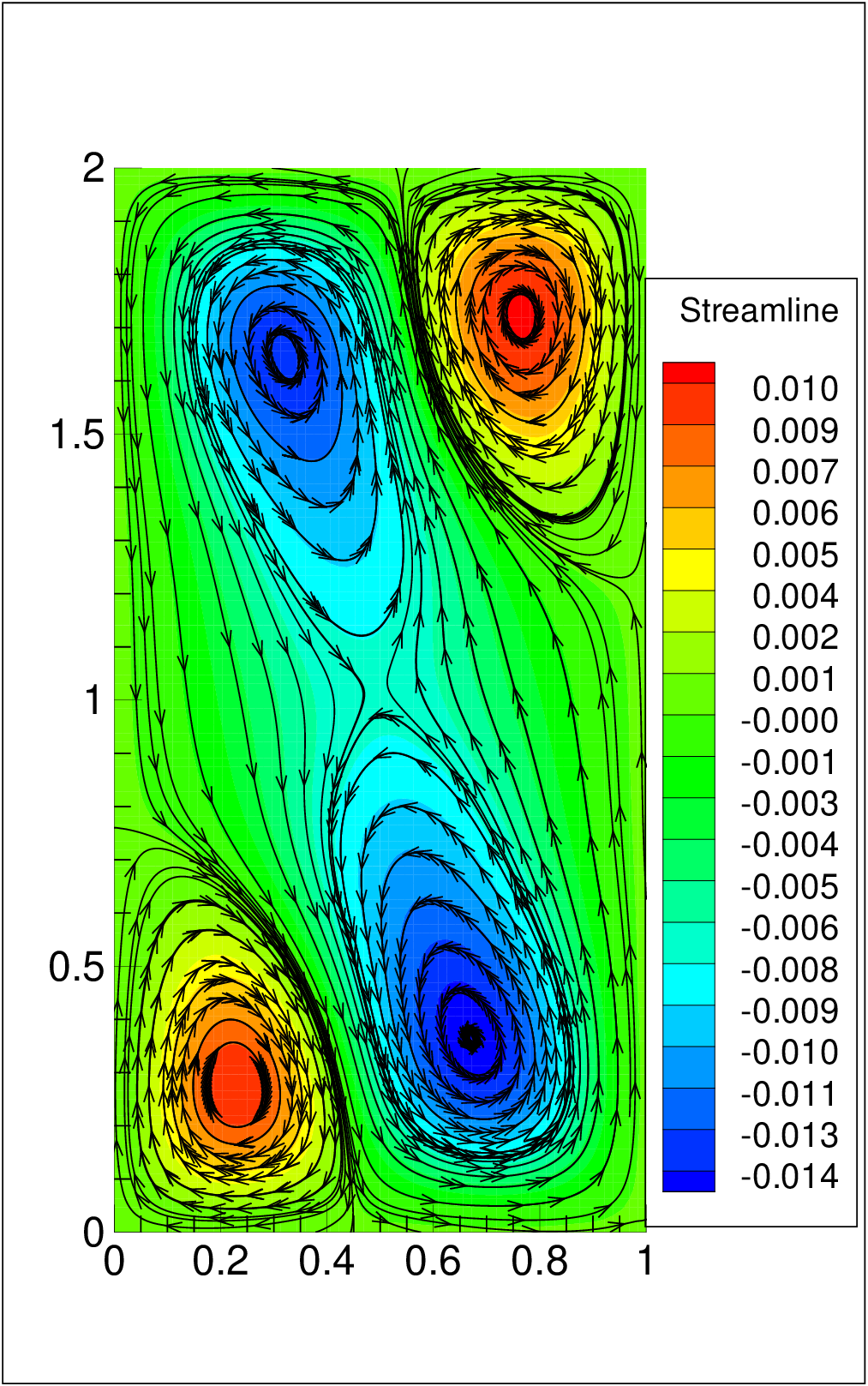}
	\includegraphics[scale=0.24]{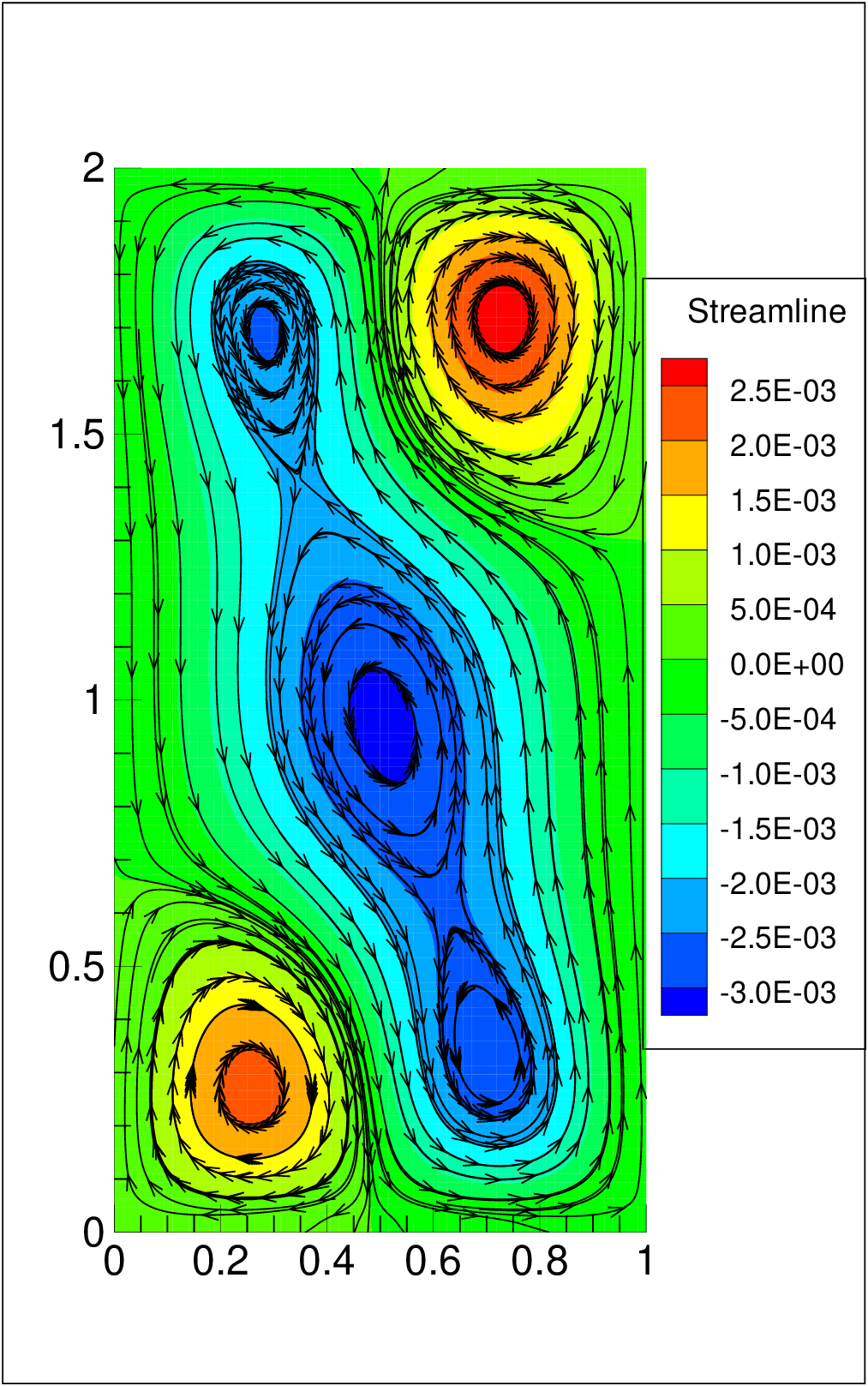}
	\includegraphics[scale=0.24]{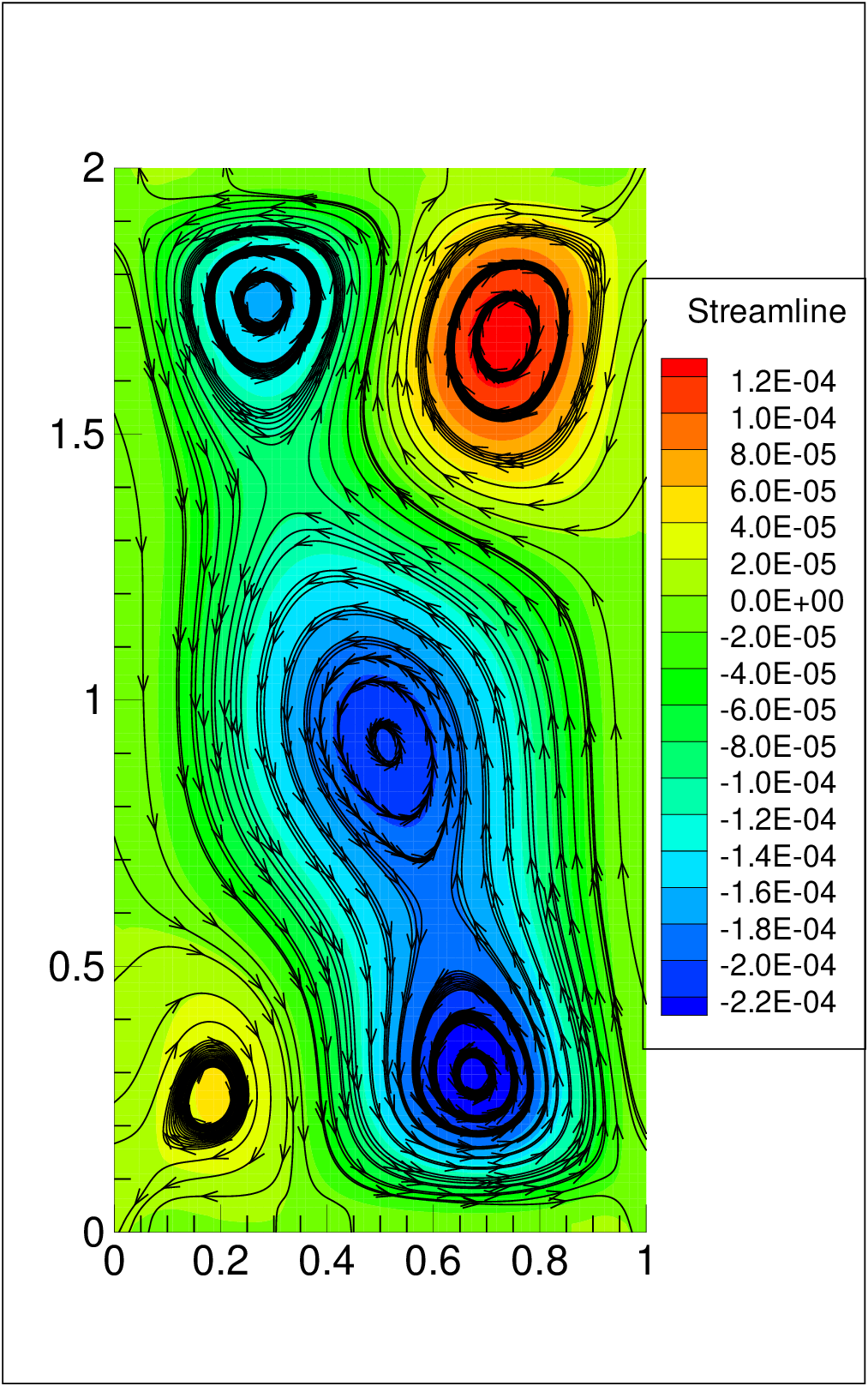}
	\caption{The evolution of streamlines of fluid velocity in different time levels.}	
	\label{fig:psi}	
\end{figure}
%
%
% Mass DG
\begin{figure}[!h]
	\centering
	\begin{subfigure}[b]{0.48\textwidth}
		\includegraphics[scale=0.450]{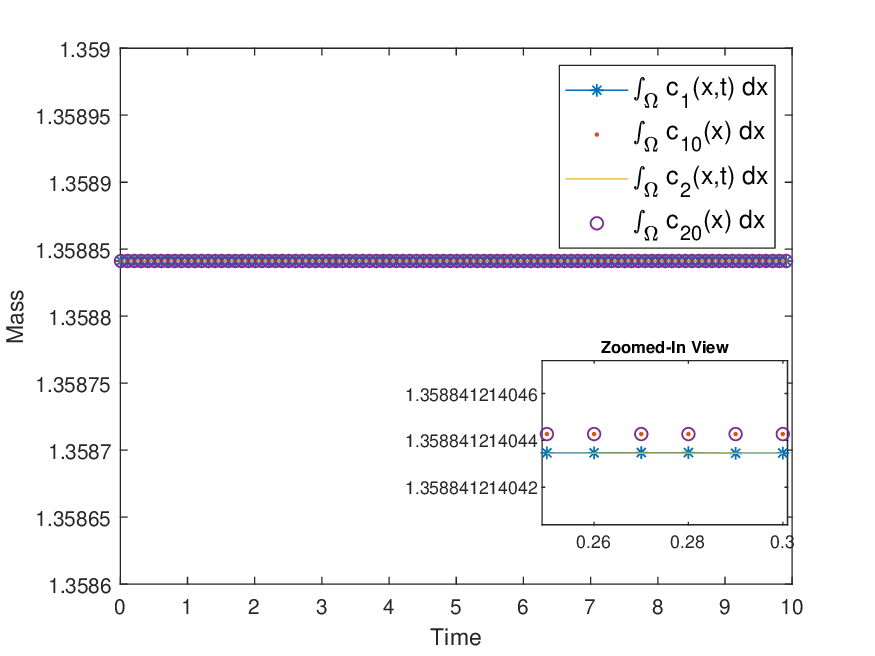}
	\end{subfigure}
	\hfill
	\begin{subfigure}[b]{0.48\textwidth}
		\includegraphics[scale=0.450]{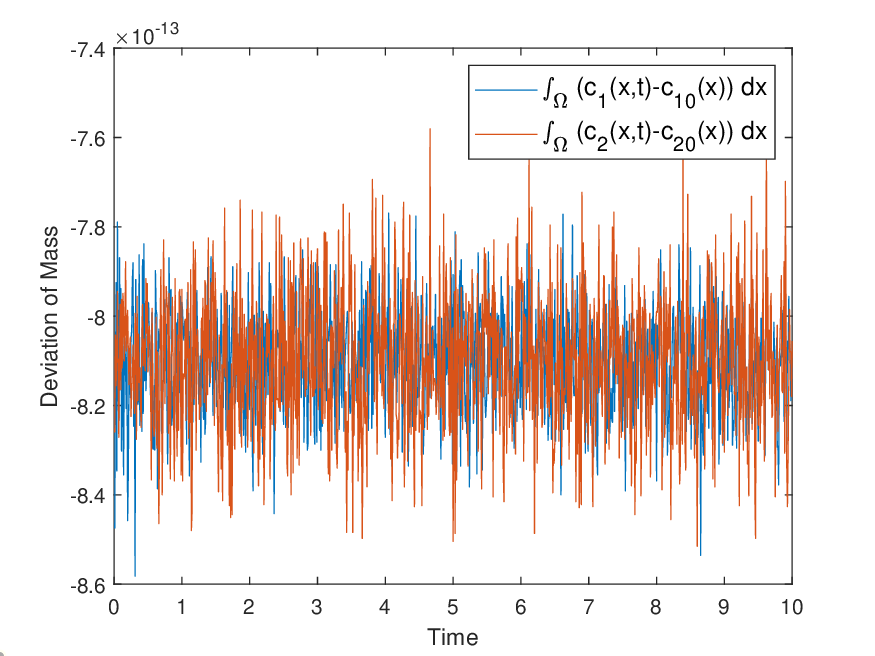}
	\end{subfigure}
	\caption{Mass (left) and deviation of mass (right) for both the positive and negative charge ions concentration with the DG method.}
	\label{fig:massion}
\end{figure}
%
% Min max and energy DG
\begin{figure}[!h]
	\centering
	\begin{subfigure}[b]{0.48\textwidth}
		\includegraphics[scale=0.450]{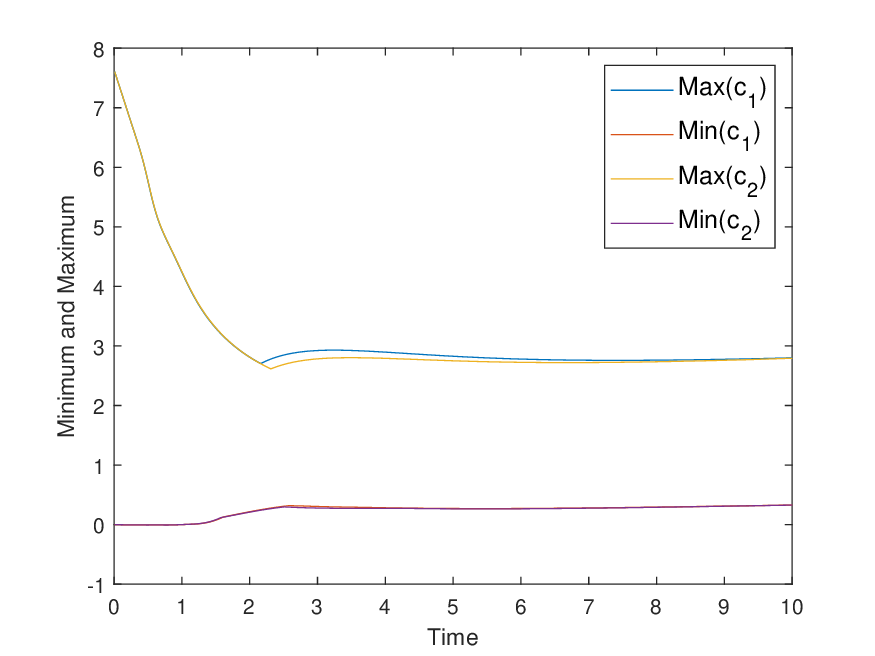}
	\end{subfigure}
	\hfill
	\begin{subfigure}[b]{0.48\textwidth}
		\includegraphics[scale=0.450]{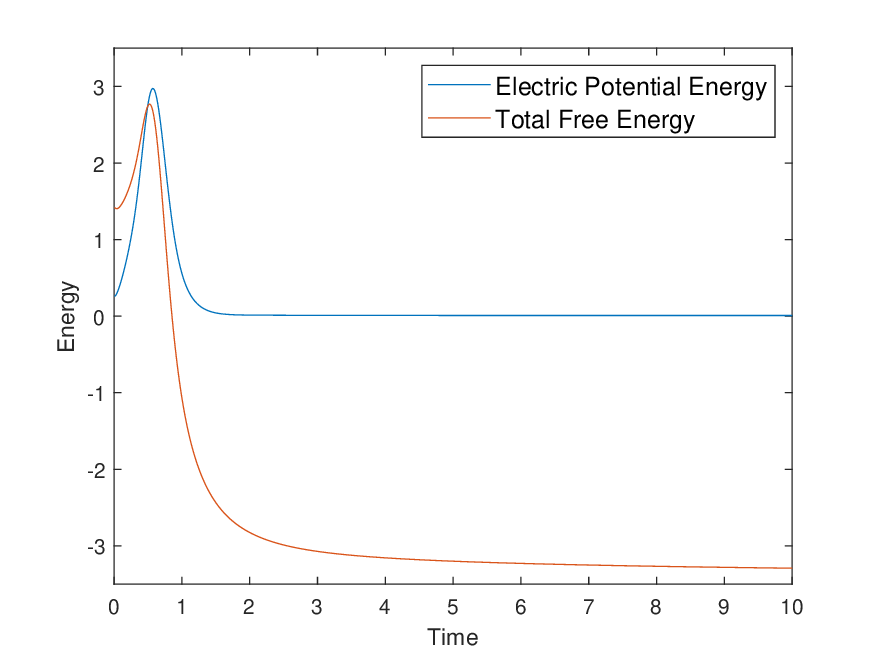}
	\end{subfigure}
	\caption{Minimum and maximum values at time $t$ for both the positive and negative charge ions concentration (left) and electric potential and total energy (right) of the system.}
	\label{fig:minmaxion}
\end{figure}

\section{Conclusion}
In this paper, we have analyzed a fully discrete discontinuous Galerkin method for the coupled Poisson-Nernst-Planck-Navier-Stokes system based on a first-order pressure correction scheme. We have obtained optimal error estimates in the $L^2$-norm with $\mathcal{O}(h^{k+1}+\Delta t)$ and the energy-norm with $\mathcal{O}(h^k+\Delta t)$ for the fully discrete concentrations of positive and negative ions, electrostatic potential, and fluid velocity. We have also derived an error bound in $L^2$-norm with $\mathcal{O}(h^k+\Delta t)$ for the fluid pressure. Numerical simulations are performed to verify the accuracy of our proposed fully discrete scheme and conservation properties.

\vspace{2em}
\noindent
{\bf {Acknowledgement}}. 
The second author acknowledges the support by SERB, Govt. India,  via the Project No. CRG/2023/000721.

\vspace*{0.3cm}
\noindent{\tt Funding} This work was supported by  SERB, Govt. India,  via the Project No. CRG/2023/000721.

\vspace*{0.3cm}
\noindent{\tt  Data Availability} All data supporting the findings of this study are available within the article.

\vspace*{0.3cm}

\noindent{\bf Declarations}

\vspace*{0.3cm}

\noindent {\tt Conflict of interest} The authors declare that they have no Conflict of interest.

\bibliography{ref_PNPNSE}

@article {Arn82,
	AUTHOR = {Arnold, D. N.},
	TITLE = {An interior penalty finite element method with discontinuous
	elements},
	JOURNAL = {SIAM J. Numer. Anal.},
	FJOURNAL = {SIAM Journal on Numerical Analysis},
	VOLUME = {19},
	YEAR = {1982},
	NUMBER = {4},
	PAGES = {742--760},
	ISSN = {0036-1429},
	MRCLASS = {65N30},
	MRNUMBER = {664882},
	DOI = {10.1137/0719052},
	URL = {https://doi.org/10.1137/0719052},
}

@article {BGR23,
	AUTHOR = {Bajpai, S. and Goswami, D. and Ray, K.},
	TITLE = {A priori error estimates of a discontinuous {G}alerkin method	for the {N}avier-{S}tokes equations},
	JOURNAL = {Numer. Algorithms},
	FJOURNAL = {Numerical Algorithms},
	VOLUME = {94},
	YEAR = {2023},
	NUMBER = {2},
	PAGES = {937--1002},
	ISSN = {1017-1398,1572-9265},
	MRCLASS = {65M12 (65M60 76D05)},
	MRNUMBER = {4636811},
	DOI = {10.1007/s11075-023-01525-w},
	URL = {https://doi.org/10.1007/s11075-023-01525-w},
}

@book{Cia78,
	title={The Finite Element Method for Elliptic Problem},
	author={Ciarlet, P. G.},
	year={1978},
	publisher={North-Holland Publ. Comp., Amsterdam}
}

@book{DE11,
	title={Mathematical aspects of discontinuous Galerkin methods},
	author={Di Pietro, D. A. and Ern, A.},
	volume={69},
	year={2011},
	publisher={Springer Science \& Business Media}
}

@incollection {DD75,
	AUTHOR = {Douglas, J. and Dupont, T.},
	TITLE = {Interior penalty procedures for elliptic and parabolic {G}alerkin methods},
	BOOKTITLE = {Computing methods in applied sciences ({S}econd {I}nternat.
	{S}ympos., {V}ersailles, 1975)},
	SERIES = {Lecture Notes in Phys., Vol. 58},
	PAGES = {207--216},
	PUBLISHER = {Springer, Berlin-New York},
	YEAR = {1976},
	MRCLASS = {65N30},
	MRNUMBER = {440955},
	MRREVIEWER = {J. R. Cannon},
}

@book{GR79,
	title={Finite element approximation of the {N}avier-{S}tokes equations},
	author={Girault, V. and Raviart, P. A.},
	volume={749},
	year={1979},
	publisher={Springer Berlin}
}

@article {GRW05spliting,
	AUTHOR = {Girault, V. and Rivi\`ere, B. and Wheeler, M.	F.},
	TITLE = {A splitting method using discontinuous {G}alerkin for the transient incompressible {N}avier-{S}tokes equations},
	JOURNAL = {M2AN Math. Model. Numer. Anal.},
	FJOURNAL = {M2AN. Mathematical Modelling and Numerical Analysis},
	VOLUME = {39},
	YEAR = {2005},
	NUMBER = {6},
	PAGES = {1115--1147},
	ISSN = {0764-583X,1290-3841},
	MRCLASS = {76D05 (65M60 76M10)},
	MRNUMBER = {2195907},
	MRREVIEWER = {Laurent\ E.\ Gosse},
	DOI = {10.1051/m2an:2005048},
	URL = {https://doi.org/10.1051/m2an:2005048},
}

@article {GRW05domain,
	AUTHOR = {Girault, V. and Rivi\`ere, B. and Wheeler, M.	F.},
	TITLE = {A discontinuous {G}alerkin method with nonoverlapping domain decomposition for the {S}tokes and {N}avier-{S}tokes problems},
	JOURNAL = {Math. Comput.},
	FJOURNAL = {Mathematics of Computation},
	VOLUME = {74},
	YEAR = {2005},
	NUMBER = {249},	
	PAGES = {53--84},
	ISSN = {0025-5718,1088-6842},
	MRCLASS = {65N30 (65N55 76D05 76D07 76M25)},
	MRNUMBER = {2085402},
	MRREVIEWER = {Karsten\ Urban},
	DOI = {10.1090/S0025-5718-04-01652-7},
	URL = {https://doi.org/10.1090/S0025-5718-04-01652-7},
}

@article{Hec12,
  AUTHOR = {Hecht, F.},
  TITLE = {New development in FreeFem++},
  JOURNAL = {J. Numer. Math.},
  FJOURNAL = {Journal of Numerical Mathematics},
  VOLUME = {20}, YEAR = {2012},
  NUMBER = {3-4}, PAGES = {251--265},
  ISSN = {1570-2820},
  MRCLASS = {65Y15},
  MRNUMBER = {3043640},
  URL = {https://freefem.org/}
}

@article{HS00,
	AUTHOR = {Hill, A. T. and S\"{u}li, E.},
	TITLE = {Approximation of the global attractor for the incompressible
	{N}avier-{S}tokes equations},
	JOURNAL = {IMA J. Numer. Anal.},
	FJOURNAL = {IMA Journal of Numerical Analysis},
	VOLUME = {20},
	YEAR = {2000},
	NUMBER = {4},
	PAGES = {633--667},
	ISSN = {0272-4979},
	MRCLASS = {37L30 (65P40 76D05 76M10)},
	MRNUMBER = {1795301},
	MRREVIEWER = {Marcel Oliver},
	DOI = {10.1093/imanum/20.4.633},
	URL = {https://doi.org/10.1093/imanum/20.4.633},
	}

@article {MLR23,
	AUTHOR = {Masri, R. and Liu, C. and Riviere, B.},
	TITLE = {Improved a priori error estimates for a discontinuous {G}alerkin pressure correction scheme for the {N}avier-{S}tokes equations},
	JOURNAL = {Numer. Methods Partial Differential Equations},
	FJOURNAL = {Numerical Methods for Partial Differential Equations. An	International Journal},
	VOLUME = {39},
	YEAR = {2023},
	NUMBER = {4},
	PAGES = {3108--3144},
	ISSN = {0749-159X,1098-2426},
	MRCLASS = {65M60 (65M12 76D05)},
	MRNUMBER = {4596555},
}

@book{Riv08,
	title={Discontinuous {G}alerkin methods for solving elliptic and parabolic equations: theory and implementation},
	author={Rivi{\`e}re, B.},
	year={2008},
	publisher={SIAM}
}

@techreport{RH73,
	title={Triangular mesh methods for the neutron transport equation},
	author={Reed, W. H. and Hill, T. R.},
	year={1973},
	institution={Los Alamos Scientific Lab., N. Mex.(USA)}
}

@article{WZWZ23,
	AUTHOR = {Wang, M. and Zou, G. and Wang, B. and Zhao, W.},
	TITLE = {Unconditionally energy-stable discontinuous {G}alerkin method
	for the chemo-repulsion-{N}avier-{S}tokes system},
	JOURNAL = {Comput. Math. Appl.},
	FJOURNAL = {Computers \& Mathematics with Applications. An International
	Journal},
	VOLUME = {150},
	YEAR = {2023},
	PAGES = {132--155},
	ISSN = {0898-1221},
	MRCLASS = {65M60 (35Q30 76V05 92C17)},
	MRNUMBER = {4649001},
	DOI = {10.1016/j.camwa.2023.09.012},
	URL = {https://doi.org/10.1016/j.camwa.2023.09.012},
	}

@article{Jer02,
  title={Analytical approaches to charge transport in a moving medium},
  author={Jerome, J. W.},
  journal={Transp. Theory Stat. Phys.},
  fjournal={Transport Theory and Statistical Physics},
  volume={31},
  number={4-6},
  pages={333--366},
  year={2002},
  publisher={Taylor \& Francis}
}

@article{Ryh09,
  title={Existence, uniqueness, regularity and long-term behavior for dissipative systems modeling electrohydrodynamics},
  author={Ryham, R. J.},
  journal={arXiv preprint arXiv:0910.4973},
  year={2009}
}

@article{ZY15,
  title={Global well-posedness for the {N}avier--{S}tokes--{N}ernst--{P}lanck--{P}oisson system in dimension two},
  author={Zhang, Z. and Yin, Z.},
  journal={Appl. Math. Lett.},
  fjournal={Applied Mathematics Letters},
  volume={40},
  pages={102--106},
  year={2015},
  publisher={Elsevier}
}

@article{BFS14,
  title={Global well-posedness and stability of electrokinetic flows},
  author={Bothe, D. and Fischer, A. and Saal, J.},
  journal={SIAM J. Math. Anal.},
  fjournal={SIAM Journal on Mathematical Analysis},
  volume={46},
  number={2},
  pages={1263--1316},
  year={2014},
  publisher={SIAM}
}

@article{LW20,
  title={Global existence for {N}ernst--{P}lanck--{N}avier--{S}tokes system in $\mathbb{R}^n$},
  author={Liu, J. G. and Wang, J.},
  journal={Commun. Math. Sci.},
  fjournal={Communications in Mathematical Sciences},
  volume={18},
  number={6},
  year={2020}
}

@article{Sch09,
  title={Analysis of the {N}avier--{S}tokes--{N}ernst--{P}lanck--{P}oisson system},
  author={Schmuck, M.},
  journal={Math. Models Methods Appl. Sci.},
  fjournal={Mathematical Models and Methods in Applied Sciences},
  volume={19},
  number={06},
  pages={993--1014},
  year={2009},
  publisher={World Scientific}
}

@article{CIL22,
  title={{N}ernst--{P}lanck--{N}avier--{S}tokes system near equilibrium},
  author={Constantin, P. and Ignatova, M. and Lee, F. N.},
  journal={Pure Appl. Func. Anal.},
  fjournal={Pure and Applied Functional Analysis},
  volume={7},
  number={1},
  pages={175--196},
  year={2022},
  publisher={Yokohama Publications}
}

@article{Lee22,
  title={Global regularity for {N}ernst--{P}lanck--{N}avier--{S}tokes systems with mixed boundary conditions},
  author={Lee, F. N.},
  journal={Nonlinearity},
  volume={36},
  number={1},
  pages={255},
  year={2022},
  publisher={IOP Publishing}
}

@article{PS10,
  title={Convergent finite element discretizationsof the {N}avier-{S}tokes-{N}ernst-{P}lanck-{P}oisson system},
  author={Prohl, A. and Schmuck, M.},
  journal={ESAIM: Math. Model. Numer. Anal.},
  fjournal={ESAIM: Mathematical Modelling and Numerical Analysis},
  volume={44},
  number={3},
  pages={531--571},
  year={2010},
  publisher={EDP Sciences}
}

@article{HS18,
  title={Mixed finite element analysis for the {P}oisson--{N}ernst--{P}lanck/{S}tokes coupling},
  author={He, M. and Sun, P.},
  fjournal={Journal of Computational and Applied Mathematics},
  journal={J. Comput. Appl. Math.},
  volume={341},
  pages={61--79},
  year={2018},
  publisher={Elsevier}
}

@article{LX17,
  title={Efficient time-stepping/spectral methods for the {N}avier-{S}tokes-{N}ernst-{P}lanck-{P}oisson equations},
  author={Liu, X. and Xu, C.},
  journal={Commun. Comput. Phys.},
  fjournal={Communications in Computational Physics},
  volume={21},
  number={5},
  pages={1408--1428},
  year={2017},
  publisher={Cambridge University Press}
}

@article{ZX23,
  title={Efficient time-stepping schemes for the {N}avier-{S}tokes-{N}ernst-{P}lanck-{P}oisson equations},
  author={Zhou, X. and Xu, C.},
  journal={Comput. Phys. Commun.},
  fjournal={Computer Physics Communications},
  volume={289},
  pages={108763},
  year={2023},
  publisher={Elsevier}
}

@article{QW25,
  title={A second-order accurate, positivity-preserving numerical scheme for the {P}oisson--{N}ernst--{P}lanck--{N}avier--{S}tokes system},
  author={Qin, Y. and Wang, C.},
  journal={IMA J. Numer. Anal.},
  fjournal={IMA Journal of Numerical Analysis},
  pages={draf027},
  year={2025},
  publisher={Oxford University Press}
}

@article{LL24,
  title={Error estimates for the finite element method of the {N}avier-{S}tokes-{P}oisson-{N}ernst-{P}lanck equations},
  author={Li, M. and Li, Z.},
  journal={Appl. Numer. Math.},
  fjournal={Applied Numerical Mathematics},
  volume={197},
  pages={186--209},
  year={2024},
  publisher={Elsevier}
}

@article{HC25,
  title={Stability and error analysis of pressure-correction scheme for the {N}avier--{S}tokes--{P}lanck--{N}ernst--{P}oisson equations},
  author={He, Y. and Chen, H.},
  journal={ESAIM: Math. Model. Numer. Anal.},
  fjournal= {ESAIM: Mathematical Modelling and Numerical Analysis},
  volume={59},
  number={3},
  pages={1471--1504},
  year={2025},
  publisher={EDP Sciences}
}

@article{LSL22,
  title={New SAV-pressure correction methods for the {N}avier-{S}tokes equations: stability and error analysis},
  author={Li, X. and Shen, J. and Liu, Z.},
  journal={Math. Comput.},
  fjournal={Mathematics of Computation},
  volume={91},
  number={333},
  pages={141--167},
  year={2022}
}

@article{BHP25,
  title={On a completely discrete discontinuous {G}alerkin method for incompressible {C}hemotaxis-{N}avier-{S}tokes {E}quations},
  author={Bir, B. and Hutridurga, H. and Pani, A. K.},
  journal={J. Sci. Comput.},
  fjournal={Journal of Scientific Computing},
  volume={105},
  number={2},
  pages={1--41},
  year={2025},
}

@article{HC24,
  title={Efficiently high-order time-stepping {R-GSAV} schemes for the {N}avier--{S}tokes--{P}oisson--{N}ernst--{P}lanck equations},
  author={He, Y. and Chen, H.},
  journal={Physica D: Nonlinear Phenomena},
  volume={466},
  pages={134233},
  year={2024},
  publisher={Elsevier}
}

@article{LZW26,
  title={A new fully-decoupled energy-stable BDF2-FEM scheme for the electro-hydrodynamic equations},
  author={Li, M. and Zou, G. and Wang, B.},
  journal={Math. and Comput. Simul.},
  fjournal={Mathematics and Computers in Simulation},
  volume={239},
  pages={172--191},
  year={2026},
  publisher={Elsevier}
}

@article{PLZJH24,
  title={A linear, second-order accurate, positivity-preserving and unconditionally energy stable scheme for the {N}avier--{S}tokes--{P}oisson--{N}ernst--{P}lanck system},
  author={Pan, M. and Liu, S. and Zhu, W. and Jiao, F. and He, D.},
  journal={Commun. Nonlinear Sci. Numer.},
  fjournal={Communications in Nonlinear Science and Numerical Simulation},
  volume={131},
  pages={107873},
  year={2024},
  publisher={Elsevier}
}

@article{BTA04,
  title={Diffuse-charge dynamics in electrochemical systems},
  author={Bazant, M. Z. and Thornton, K. and Ajdari, A.},
  journal={Phys. Rev. E.},
  fjournal={Physical Review E—Statistical, Nonlinear, and Soft Matter Physics},
  volume={70},
  number={2},
  pages={021506},
  year={2004},
  publisher={APS}
}

@article{ZL24,
  title={Global existence of large solutions for the three-dimensional incompressible {N}avier--{S}tokes--{P}oisson--{N}ernst--{P}lanck equations},
  author={Zhao, J. and Li, Y.},
  journal={Math. Method. Appl. Sci.},
  fjournal={Mathematical Methods in the Applied Sciences},
  volume={47},
  number={15},
  pages={11933--11952},
  year={2024},
  publisher={Wiley Online Library}
}

@article{CGHRS24,
  title={Banach spaces-based mixed finite element methods for the coupled {N}avier--{S}tokes and {P}oisson--{N}ernst--{P}lanck equations},
  author={Correa, C. I. and Gatica, G. N. and Henr{\'\i}quez, E. and Ruiz-Baier, R. and Solano, M.},
  journal={Calcolo},
  volume={61},
  number={2},
  pages={31},
  year={2024},
  publisher={Springer}
}

@article{LBM20,
  title={Transient electrohydrodynamic flow with concentration-dependent fluid properties: modelling and energy-stable numerical schemes},
  author={Linga, G. and Bolet, A. and Mathiesen, J.},
  journal={J. Comput. Phys.},
  fjournal={Journal of Computational Physics},
  volume={412},
  pages={109430},
  year={2020},
  publisher={Elsevier}
}

@article{HS21,
  title={Mixed finite element method for modified {P}oisson--{N}ernst--{P}lanck/{N}avier--{S}tokes equations},
  author={He, M. and Sun, P.},
  journal={J. Sci. Comput.},
  fjournal={Journal of Scientific Computing},
  volume={87},
  number={3},
  pages={80},
  year={2021},
  publisher={Springer}
}
\bibliographystyle{plain}

\end{document}